\documentclass[10pt]{amsart}

\usepackage{xypic,amscd,amssymb,latexsym}
\usepackage{amsmath}	

\usepackage[breaklinks=true]{hyperref}

\usepackage{mathrsfs}

\usepackage{pifont}

\usepackage{url}

\usepackage{graphicx}

\usepackage{cancel}

\usepackage{caption}
\usepackage{subcaption}

\begin{document}
\pagenumbering{arabic}

\newtheorem{theorem}{Theorem}[section]
\newtheorem{lemma}[theorem]{Lemma}
\newtheorem{proposition}[theorem]{Proposition}
\newtheorem{corollary}[theorem]{Corollary}
\newtheorem{definition}[theorem]{Definition}
\newtheorem{remark}[theorem]{Remark}
\newtheorem{notation}[theorem]{Notation}
\newtheorem{question}[theorem]{Question}

\newcommand{\vs}[0]{\vspace{2mm}}

\newcommand{\mcal}[1]{\mathcal{#1}}
\newcommand{\ul}[1]{\underline{#1}}
\newcommand{\ulu}[2]{\underset{#2}{\underline{#1}}}
\newcommand{\ol}[1]{\overline{#1}}
\newcommand{\til}[1]{\widetilde{#1}}
\newcommand{\wh}[1]{\widehat{#1}}

\newcommand{\smallmattwo}[4]{
\left(
\begin{smallmatrix}
#1 & #2 \\
#3 & #4 \\
\end{smallmatrix}
\right)
}

\newcommand{\smallvectwo}[2]{
\left(
\begin{smallmatrix}
#1 \\
#2 \\
\end{smallmatrix}
\right)
}

\newcommand{\smallvecthree}[3]{
\left(
\begin{smallmatrix}
#1 \\
#2 \\
#3 \\
\end{smallmatrix}
\right)
}

\newcommand{\vectwo}[2]{
\left(
\begin{array}{cc}
#1 \\
#2 \\
\end{array}
\right)
}

\newcommand{\vecthree}[3]{
\ensuremath{\left(
\renewcommand{\arraystretch}{1.2}
\begin{array}{cc}
#1 \\
#2 \\
#3
\end{array}}
\right)
}

\newcommand{\mattwo}[4]{
\left(
\begin{array}{cc}
#1 & #2 \\
#3 & #4 \\
\end{array}
\right)
}

\newcommand{\matthree}[9]{
\left(
\begin{array}{ccc}
#1 & #2 & #3 \\
#4 & #5 & #6 \\
#7 & #8 & #9
\end{array}
\right)
}



\newcommand{\cone}[0]{\mbox{\ding{192}}}
\newcommand{\ctwo}[0]{\mbox{\ding{193}}}
\newcommand{\cthree}[0]{\mbox{\ding{194}}}
\newcommand{\cfour}[0]{\mbox{\ding{195}}}
\newcommand{\cfive}[0]{\mbox{\ding{196}}}
\newcommand{\csix}[0]{\mbox{\ding{197}}}
\newcommand{\cseven}[0]{\mbox{\ding{198}}}
\newcommand{\ceight}[0]{\mbox{\ding{199}}}
\newcommand{\cnine}[0]{\mbox{\ding{200}}}
\newcommand{\cten}[0]{\mbox{\ding{201}}}

\makeatletter
\newcommand\footnoteref[1]{\protected@xdef\@thefnmark{\ref{#1}}\@footnotemark}
\makeatother

\address{Department of Mathematics, Yale University, New Haven, Connecticut 06511, USA}

\address{Department of Mathematics, Ewha Womans University, 52 Ewhayeodae-gil, Seodaemun-gu, Seoul 03760, Republic of Korea}

\email[I.~B.~Frenkel]{frenkel-igor@yale.edu}

\email[H.~Kim]{hyunkyukim@ewha.ac.kr, hyunkyu87@gmail.com}

\author{Igor B. Frenkel}
\author{Hyun Kyu Kim}

\numberwithin{equation}{section}

\title[Three dimensional construction of the Virasoro-Bott group]{\resizebox{150mm}{!}{Three dimensional construction of the Virasoro-Bott group}}

\begin{abstract}
We present a three-dimensional geometric construction of the Virasoro-Bott group, which is a central extension of the group of diffeomorphisms of the circle. Our approach is analogous to the well-known construction of a central extension of the loop group by means of the Wess-Zumino topological term. In particular, the Virasoro-Bott group is realized as a quotient group of diffeomorphisms of the disc with special boundary conditions. We identify the Lie algebra corresponding to our group with the Virasoro algebra. We also show that for generalized boundary conditions the Virasoro algebra is extended to a semidirect product with the Heisenberg algebra. We discuss the relation between our construction, the Chern-Simons theory, and the three-dimensional gravity.
\end{abstract}

\maketitle

\tableofcontents

\section{Introduction}

There is a profound analogy between the structure and the representation theories of the affine Lie algebras and the Virasoro algebra, as well as of the groups corresponding to these Lie algebras. In particular, the affine Lie groups are central extensions of the groups of smooth maps ${\rm Map}(S^1,G)$ from the unit circle $S^1$ to simple compact simply-connected Lie groups $G$ with pointwise multiplication, while the Virasoro-Bott group is a central extension of the orientation-preserving diffeomorphism group ${\rm Diff}_+(S^1)$. Similarly, the corresponding Lie algebras are central extensions of the Lie algebras of smooth maps ${\rm Map}(S^1,\frak{g})$, where $\frak{g}$ is the Lie algebra of $G$, and the Lie algebra of vector fields ${\rm Vect}(S^1)$. By definition the elements of both Lie algebras and groups depend on one parameter, i.e. these Lie algebras and groups have functional dimension $1$. However, the explicit construction of the affine Lie groups, first obtained in \cite{M} based on the previous work \cite{W1}, requires a realization of $\mathrm{Map}(S^1,G)$ as a factorgroup of $\mathrm{Map}(D^2,G)$, where $D^2$ is the unit disc with the boundary $S^1$, thus lifting the functional dimension to $2$. Moreover, the factorization of the central extension of $\mathrm{Map}(D^2,G)$ by an appropriate normal subgroup can be written explicitly only if we extend trivially an element in $\mathrm{Map}(D^2,G)$ to $\mathrm{Map}(S^2,G)$ and then further to $\mathrm{Map}(B^3,G)$, where $B^3$ is the unit ball in $\mathbb{R}^3$ with the boundary $S^2$ containing the unit disc $D^2$ as a hemisphere. This brings the functional dimension of the realization of affine Lie groups to $3$.

\vs

In the present paper we will construct the central extension of the group $\mathrm{Diff}_+(S^1)$ in complete parallel to the construction of the central extension of $\mathrm{Map}(S^1,G)$ mentioned above, replacing the groups of maps by the corresponding groups of diffeomorphisms.

\vs

It is important to note that there is also a ``complexification" of the above construction of the central extension of $\mathrm{Map}(S^1,G)$ to $\mathrm{Map}(\Sigma^2,G)$, where $\Sigma^2$ is a Riemann surface \cite{EF} \cite{FK96}. This paper combined with \cite{FK96} suggests that there should be a ``complexification" of the present construction that yields a central extension of the group $\mathrm{Diff}_+(\Sigma^2)$. The corresponding central extension of the Lie algebra $\mathrm{Vect}(\Sigma^2)$ has been considered in \cite{EF}. One can combine the four groups that admit similar constructions of the central extensions in the following table.
$$
{\renewcommand{\arraystretch}{1.4} \begin{array}{|c|c|}
\hline
\mathrm{Map}(S^1,G) & \mathrm{Map}(\Sigma^2,G) \\ \hline
\mathrm{Diff}_+(S^1) & \mathrm{Diff}_+(\Sigma^2) \\ \hline
\end{array}}
$$

\vs

There is another explicit construction of the Virasoro-Bott group based on the so-called Bott-Thurston 2-cocycle \cite{B}, which does not require a use of higher functional dimensions. This might partly explain why the three dimensional functional realization of the Virasoro group has not been studied before. However for a certain natural completion of $\mathrm{Diff}_+(S^1)$ by a Sobolev $\frac{3}{2}$ norm that arises in the theory of the universal Teichm\"uller space \cite{GR}, Bott-Thurston $2$-cocycle is no longer defined. Though in this paper we are working for simplicity with smooth diffeomorphisms, we strongly believe that our construction still extends to the above completion (see e.g. \cite{HK}).

\vs

Now we explain our construction in more detail. First we recall the construction in \cite{M} (see also \cite[II.1.3]{KW09} and \cite{FK96}) of the central extension by $\mathrm{U}(1) \approx S^1$ of $\mathrm{Map}(S^1,G)$ as a quotient group. Let ${\rm Map}(D^2, G)$ be the group of $C^\infty$ maps from the open unit disc $D^2$ to $G$ whose boundary value is an element of $\mathrm{Map}(S^1,G)$, which perhaps should be required to have certain analytic conditions near the boundary $S^1$ of $D^2$ in order for the construction to completely make sense. By letting $\mathrm{Map}(D^2,G)_{S^1} := \{ g \in \mathrm{Map}(D^2,G) \, : \, g|_{S^1} = 1_G \}$, where $1_G$ means the constant function whose value is the identity element of $G$, one finds that $\mathrm{Map}(D^2,G)_{S^1}$ is a normal subgroup of $\mathrm{Map}(D^2,G)$, yielding a natural isomorphism $\mathrm{Map}(S^1,G) \cong \mathrm{Map}(D^2,G)/\mathrm{Map}(D^2,G)_{S^1}$. Define a function $\gamma : \mathrm{Map}(D^2,G) \times \mathrm{Map}(D^2,G) \to \mathbb{R}$ by the formula
$$
\gamma(g_1,g_2) = \frac{\vartheta^2}{16\pi^2} 
\int_{D^2} {\rm tr}( g_1^{-1} dg_1 \wedge d g_2 \, g_2^{-1}), \qquad \forall g_1,g_2 \in \mathrm{Map}(D^2,G),
$$
where ${\rm tr}$ is any symmetric invariant $\mathbb{R}$-bilinear form on $\frak{g}$ and $\vartheta$ is the length of the longest root of $\frak{g}$ with respect to $\mathrm{tr}$. Here, $g_1^{-1} dg_1$ and $dg_2 \, g_2^{-1}$ are $\frak{g}$-valued $1$-forms on $D^2$; we take the wedge of the differential form parts and evaluate $\mathrm{tr}$ to the $\frak{g}$-valued parts, to obtain a differential $2$-form on $D^2$. Then $\gamma$ is a normalized $\mathbb{R}$-valued group $2$-cocycle of $\mathrm{Map}(D^2,G)$, i.e. satisfies the $2$-cocycle property
$$
\gamma(g_1, g_2) + \gamma(g_1 g_2, g_3) = \gamma(g_2,g_3) + \gamma(g_1, g_2 g_3), \qquad \forall g_1,g_2,g_3 \in \mathrm{Map}(D^2,G),
$$
and satisfies the normalization $\gamma(g,g^{-1})=0$, $\forall g\in \mathrm{Map}(D^2,G)$. One can then construct a central extension $\mathrm{Map}(D^2,G)^\wedge$ of $\mathrm{Map}(D^2,G)$ by ${\rm U}(1)$, corresponding to the ${\rm U}(1)$-valued $2$-cocycle $e^{2\pi {\rm i}  \gamma}$, where ${\rm i}=\sqrt{-1}$ is a fixed square root of unity. More precisely, $\mathrm{Map}(D^2,G)^\wedge := \mathrm{Map}(D^2,G) \times {\rm U}(1)$ as a manifold, with the multiplication given by
$$
(g_1,\alpha_1) \cdot (g_2,\alpha_1) = (g_1g_2, \, \alpha_1\alpha_2 e^{2\pi {\rm i} \,\gamma(g_1,g_2)}), \qquad \forall g_1,g_2 \in \mathcal{G}, \quad \forall \alpha_1,\alpha_2 \in {\rm U}(1).
$$
Next, consider the embedding $\iota : \mathrm{Map}(D^2,G)_{S^1} \to \mathrm{Map}(D^2,G)^{\wedge}$, $h\mapsto (h, e^{2\pi {\rm i} \, \omega_0(h)})$, where the function $\omega_0 : \mathrm{Map}(D^2,G)_{S^1} \to \mathbb{R}$ is defined as follows. For $h \in \mathrm{Map}(D^2,G)_{S^1}$, extend it to an element of $\mathrm{Map}(S^2,G)$ by setting it to be the constant function with value $1_G$ on the complement of $D^2$ in $S^2$. Then extend it further to an element $\til{h}$ of $\mathrm{Map}(B^3,G)$ in \emph{any} way, which is possible at all because $\pi_2(G)=1$, and define
$$
\omega_0(h) := \frac{\vartheta^2}{48\pi^2} \int_{B^3} {\rm tr} \left(\til{h}^{-1} d\til{h} \wedge \frac{1}{2} [\til{h}^{-1} d\til{h}, \til{h}^{-1} d\til{h}]\right), \qquad \forall h \in \mathrm{Map}(D^2,G)_{S^1},
$$
where the bracket $[\, ,\, ]$ of two $\frak{g}$-valued $1$-forms means to take the wedge of the differential form parts and take the Lie bracket of the $\frak{g}$-value parts. The freedom of the choice of an extension $\til{h}$ results in the ambiguity of the number $\omega_0(h)$ up to the addition of an integer; so $e^{2\pi {\rm i} \, \omega_0(h)} \in \mathrm{U}(1)$ is well-defined. One shows by using Stokes' Theorem that 
$$
\gamma(h_1,h_2) = \omega_0(h_1h_2) - \omega_0(h_1) - \omega_0(h_2) \quad \mathrm{mod}~\mathbb{Z}, \quad \mbox{$\forall h_1,h_2\in \mathrm{Map}(D^2,G)_{S^1}$,}
$$
from which it follows that $\iota$ is a group homomorphism. It also turns out that
$$
\omega_0(h) + \gamma(g,h) + \gamma(gh, g^{-1}) = \omega_0(ghg^{-1}) \quad \mathrm{mod}~\mathbb{Z}, \quad \mbox{$\forall g\in \mathrm{Map}(D^2,G)$, \,\, $\forall h\in\mathrm{Map}(D^2,G)_{S^1}$,}
$$
which in \cite{M89} is proved by using the perfectness of $\mathrm{Map}(D^2,G)_{S^1}$, and from which it follows that the image of $\iota$ is a normal subgroup of $\mathrm{Map}(D^2,G)^\wedge$. Then one observes that the group $\mathrm{Map}(D^2,G)^\wedge/\iota(\mathrm{Map}(D^2,G)_{S^1})$ is a central extension of $\mathrm{Map}(S^1,G) \cong \mathrm{Map}(D^2,G)/\mathrm{Map}(D^2,G)_{S^1}$ by $\mathrm{U}(1)$; one thus obtains a factorgroup realization of the `basic' central extension $\mathrm{Map}(S^1,G)^\wedge$ of the `loop group' $\mathrm{Map}(S^1,G)$. 

\vs

In our case, the main group of interest is $\mathrm{Diff}_+(S^1)$, the group of all orientation-preserving $C^\infty$ self-diffeomorphisms of the unit circle $S^1$. We first realize it as the factorgroup
$$
\mathrm{Diff}_+(S^1) = \mathrm{Diff}_*(D^2) / \mathrm{Diff}_*(D^2)_{S^1},
$$
where $\mathrm{Diff}_*(D^2)$ (Def.\ref{def:our_model_of_the_disc_diffeomorphism_group}) is the group of all orientation-preserving $C^\infty$ self-diffeomorphisms of the open unit disc $D^2$ satisfying a certain natural analytic condition near the boundary $S^1$ which we call `asymptotically radial', and $\mathrm{Diff}_*(D^2)_{S^1}$ is the subgroup of all elements of $\mathrm{Diff}_*(D^2)$ whose boundary value is the identity diffeomorphism of $S^1$. One can easily show that $\mathrm{Diff}_*(D^2)_{S^1}$ is normal in $\mathrm{Diff}_*(D^2)$. For each $g\in \mathrm{Diff}_*(D^2)$, we take its Jacobian matrix to obtain a map $J_{D^2}(g) : D^2 \to {\rm GL}(2,\mathbb{R})$, and define a $\frak{gl}(2,\mathbb{R})$-valued $1$-form on $D^2$ by
$$
\theta_{D^2}(g) := J_{D^2}(g)^{-1} \, dJ_{D^2}(g),
$$
which can be viewed as the pullback of the Maurer-Cartan $1$-form ${\rm g}^{-1} d{\rm g}$ on ${\rm GL}(2,\mathbb{R})$ along the map $J_{D^2}(g)$. Such can be defined for any manifold other than $D^2$ too. Then, for any fixed constant $c_0 \in \mathbb{R}$, we check that
$$
\gamma(g_1,g_2) := 3c_0 \, \int_{D^2} {\rm tr}_{{\rm mat}_2} (\theta_{D^2}(g_1) \wedge \theta_{D^2}(g_2^{-1})), \qquad \forall g_1,g_2\in \mathrm{Diff}_*(D^2),
$$
where ${\rm tr}_{{\rm mat}_n}$ is the usual trace of $n\times n$ matrices, defines a normalized $\mathbb{R}$-valued $2$-cocycle of $\mathrm{Diff}_*(D^2)$; the issue of convergence of the integral is dealt with by the analytic property of elements of $\mathrm{Diff}_*(D^2)$, while the $2$-cocycle property is proven by using the composition formula for $\theta_{D^2}$:
$$
\theta_{D^2}(g_1g_2) = J_{D^2}(g_2)^{-1} \, [\theta_{D^2}(g_1) \circ g_2] \, J_{D^2}(g_2) + \theta_{D^2}(g_2), \qquad \forall g_1,g_2 \in \mathrm{Diff}_*(D^2),
$$
which also holds for other manifolds. This $2$-cocycle allows us to construct a central extension $\mathrm{Diff}_*(D^2)^\wedge$ of $\mathrm{Diff}_*(D^2)$ by $\mathbb{R}$:
\begin{align*}
& \mathrm{Diff}_*(D^2)^\wedge := \mathrm{Diff}_*(D^2) \times \mathbb{R} \qquad\mbox{as a manifold, with the multiplication defined as} \\
& (g_1,a_1) \cdot(g_2,a_2) = (g_1g_2, a_1+a_2+\gamma(g_1,g_2)), \qquad \forall g_1,g_2 \in \mathrm{Diff}_*(D^2), \quad \forall a_1,a_2 \in \mathbb{R}.
\end{align*}
We then construct an $\mathbb{R}$-valued $1$-cochain $\omega_0$ of $\mathrm{Diff}_*(D^2)_{S^1}$ as follows. For $h\in \mathrm{Diff}_*(D^2)_{S^1}$, extend $h$ to a self-diffeomorphism $B_h$ of $S^2 \approx \wh{\mathbb{C}}=\mathbb{C}\cup\{\infty\}$ by letting it to be the identity diffeomorphism on the complement of $D^2$ in $S^2$. Viewing $S^2$ as the boundary of the unit $3$-ball $B^3$, extend $B_h$ to a self-diffeomorphism $\til{B}_h$ of $B^3$ in {\em any} way, only making sure that it satisfies a certain analytic property near the boundary $S^2$. Then the Jacobian matrix $J_{B^3}(\til{B}_h)$ of $\til{B}_h$ is a map from $B^3$ to ${\rm GL}(3,\mathbb{R})$; just like in the two dimensional case, let $\theta_{B^3}(\til{B}_h)$ be the pullback of the Maurer-Cartan $1$-form $\mathrm{g}^{-1}d\mathrm{g}$ on ${\rm GL}(3,\mathbb{R})$ along $J_{B^3}(\til{B}_h)$. Define
$$
\omega_0(h) := c_0 \int_{B^3} {\rm tr}_{{\rm mat}_3} ( (\theta_{B^3}(\til{B}_h))^{\wedge 3} ), \qquad \forall h \in \mathrm{Diff}_*(D^2)_{S^1}.
$$
The convergence of this integral is guaranteed by the analytic property of $\til{B}_h$. We prove that $\omega_0(h)$ does not depend on the choice of an extension $\til{B}_h$ of $B_h$, being well-defined as an element of $\mathbb{R}$, not just up to an integer multiple of some nonzero constant. Then, using the composition formula of $\theta_{B^3}$ and the Stokes' Theorem, we show that the corresponding $2$-coboundary of $\omega_0$ coincides with the $2$-cocycle $\gamma$ restricted to $\mathrm{Diff}_*(D^2)_{S^1} \times \mathrm{Diff}_*(D^2)_{S^1}$, namely the following holds:
$$
\gamma(h_1,h_2) = \omega_0(h_1h_2) - \omega_0(h_1) - \omega_0(h_2), \qquad \forall h_1,h_1\in \mathrm{Diff}_*(D^2)_{S^1}.
$$
Therefore one obtains an embedding of groups
$$
\iota : \mathrm{Diff}_*(D^2)_{S^1} \hookrightarrow \mathrm{Diff}_*(D^2)^\wedge, \qquad h \mapsto (h, \omega_0(h)), \quad \forall h \in \mathrm{Diff}_*(D^2)_{S^1}.
$$
To establish the normality of the image of $\iota$ in $\mathrm{Diff}_*(D^2)^\wedge$ it suffices to show
$$
\omega_0(h) + \gamma(g,h) + \gamma(gh,g^{-1}) = \omega_0(ghg^{-1}), \qquad \forall g\in \mathrm{Diff}_*(D^2), \quad \forall h\in \mathrm{Diff}_*(D^2)_{S^1},
$$
which can be proved by a similar argument as in \cite{M89} provided that $\mathrm{Diff}_*(D^2)_{S^1}$ is perfect. We observe that the analytic condition imposed on elements of $\mathrm{Diff}_*(D^2)$ is such that the boundary trivial subgroup $\mathrm{Diff}_*(D^2)_{S^1}$ coincides with the identity component of the group of all smooth self-diffeomorphisms of the open disc $D^2$ with compact support, i.e. acting as identity outside a compact subset of $D^2$. The latter group is known to be perfect, by a theorem usually credited to William Thurston. Thus one obtains a central extension $\mathrm{Diff}_+(S^1)^\wedge$ of $\mathrm{Diff}_+(S^1)$ by $\mathbb{R}$ as the factorgroup
$$
\mathrm{Diff}_+(S^1)^\wedge = \mathrm{Diff}_*(D^2)^\wedge / \iota(\mathrm{Diff}_*(D^2)_{S^1})
$$
via three-dimensional geometry, especially using the so-called Wess-Zumino term $\omega_0(h)$. By `exponentiating' the cocycles and cochains, one can also turn this into a central extension of $\mathrm{Diff}_+(S^1)$ by $\mathrm{U}(1)$, which is usually called the \emph{Virasoro-Bott group}. Our construction of the Virasoro-Bott group can be summarized by the following diagram:
$$
\xymatrix@R-2mm{
& & 1 \ar[d] & 1 \ar[d] & \\
1 \ar[r] & {\rm Diff}_*(D^2)_{S^1} \ar[d]^{=} \ar[r] & {\rm Diff}_*(D^2) \ar[d] \ar[r] & {\rm Diff}_+(S^1) \ar[d] \ar[r] & 1 \\
1 \ar[r] & {\rm Diff}_*(D^2)_{S^1} \ar[r] & {\rm Diff}_*(D^2)^\wedge \ar[d] \ar[r] & {\rm Diff}_+(S^1)^\wedge = {\rm Diff}_*(D^2)^\wedge / {\rm Diff}_*(D^2)_{S^1} \ar[d] \ar[r] & 1 \\
&  & \mathbb{R} \ar[d] & \mathbb{R} \ar[d] & \\
&  & 1 & 1 &
}
$$
By `differentiating' one gets a central extension by $\mathbb{R}$ of $\mathrm{Vect}(S^1)$, the Lie algebra of smooth vector fields on $S^1$, which is the Lie algebra corresponding to $\mathrm{Diff}_+(S^1)$. A complexified version of this central extension of $\mathrm{Vect}(S^1)$ is what is referred to as the \textrm{Virasoro algebra}. Additionally, we note that one can compare our construction of the Virasoro-Bott group with the construction in [Bi] of the Bott-Thurston cocycle which implicitly uses an extension to two dimensions but not three.

\vs

The three dimensional functional realization of the central extension of $\mathrm{Diff}_+(S^1)$ allows various modifications that do not have analogues for $\mathrm{Map}(S^1,G)$. In particular, we can consider a more general Lie algebra $\mathrm{GVect}(D^2)$ by suspending the tangent condition at $S^1$, then the Lie subalgebra ${\rm Vect}(D^2)_{S^1}$ is still an ideal and the corresponding factor algebra is isomorphic to the semidirect product
$$
\mathrm{Vect}(S^1) \ltimes \mathrm{Map}(S^1,\mathbb{R}).
$$

\vs

The explicit geometric realization of the central extensions of $\mathrm{Map}(S^1,G)$ and $\mathrm{Diff}_+(S^1)$ suggests a construction of the regular representations and their irreducible components as a generalization of the classical Peter-Weyl and Borel-Weil constructions in the case of the compact group $G$. The geometric construction of the regular representation of $\mathrm{Map}(S^1,G)$ is known in the physics literature as WZW model and it provides a basic example of two dimensional conformal field theory. Much less is understood about the geometric construction of the regular representation of $\mathrm{Diff}_+(S^1)$, but the corresponding two-dimensional conformal field theory is constructed by algebraic methods. It is an interesting problem to construct an analogue of WZW model using our geometric approach and determine the geometric meaning of the special central charges.

\vs

In the physics literature there is also a geometric construction of spaces of intertwining operators for the tensor products of the irreducible representations of $\mathrm{Map}(S^1,G)$ known as the CSW model and it provides a basic example of three dimensional topological field theory. It is based on the following functional
\begin{align*}
  S(A) = \frac{k}{4\pi} \int_{M^3} \mathrm{tr}(A\wedge dA + \frac{2}{3} A \wedge A \wedge A),
\end{align*}
where $A$ is a $\frak{g}$-valued connection $1$-form on a three dimensional manifold $M^3$, possibly with a boundary. Note that for a connection $A = g^{-1} dg$ induced by $g\in \mathrm{Map}(M^3,G)$ this functional is reduced to the topological term in the construction of the central extension of $\mathrm{Map}(S^1,G)$. By considering the manifolds with boundaries and corners it is even possible to extract in CSW theory an explicit information about the individual irreducible representations of $\mathrm{Map}(S^1,G)$ \cite{EMSS}. 

\vs

The analogy between the geometric constructions of ${\rm Map}(S^1,G)$ and ${\rm Diff}_+(S^1)$ suggests that there is a version of the CSW theory for the latter group. The natural candidate for the functional $S(A)$ is the original Chern-Simons functional with $A$ being now the Christoffel connection and the trace is taken in the space ${\rm End}(\mathbb{R}^3)$. Again when the metrics induced from a trivial metric by a diffeomorphism, e.g. $g \in {\rm Diff}_+(B^3)$, the functional is reduced to our topological term. These geometric constructions have direct relation to the three dimensional gravity (see e.g. \cite{W2}, \cite{ABDGPR}). It is interesting to note that for the basic Chern-Simons theory with $k=1$ in $S(A)$, the central charge of the Virasoro algebra is $c=24$ and the partition function of the pure three dimensional gravity must be the modular invariant $J(q)=q^{-1} + 0 + 196884 + ...$ (cf. the discussion in \cite{W2} and also \cite{ABDGPR}). This suggests that the Monstrous moonshine is somehow hidden behind the three dimensional construction of the Virasoro-Bott group and related structures. The additional indication comes from the framings of the Chern-Simons theory which form a cyclic group of order $24$.  ``Sine qua non" of all framings in one formation might explain the unexpectedly large partition function of the pure three dimensional gravity. It is natural to conjecture that further understanding of the three dimensional construction of the Virasoro-Bott group and related structures will eventually reveal the true meaning of the present mysteries.

\vs

{\bf Comments.} The present paper was mostly written by 2014. The results of the paper were announced in the talk of the first author under the title ``Geometric constructions of central extensions of certain infinite dimensional groups" in General Mathematics Seminar of the St Petrsburg Division of Steklov Institute of Mathematics, RAS (Russian Academy of Sciences), for which a video is available in the internet.

\vs

{\bf Acknowledgments.} I.F. thanks Leon Takhtajan and Vlad Sergiescu for discussions. H.K. thanks Sang-hyun (Sam) Kim, Kathryn Mann and Sam Nariman for discussions.

\section{General construction of a central extension group as a factor group}
\label{sec:general_construction_of_a_central_extension_group}

Let $\mathbf{G}$ be a (possibly infinite-dimensional) real Lie group presented as a factor group
\begin{align}
\label{eq:G_as_factor}
\mathbf{G} \cong \mathcal{G}/\mathcal{H},
\end{align}
for some connected real Lie group $\mathcal{G}$ and its closed normal connected subgroup $\mathcal{H}$. A motivation to consider such a factor presentation is that sometimes the quotient may be easier to study or is more natural than $\mathbf{G}$ in a certain sense. Suppose we have a {\em $\mathbb{R}$-valued group $2$-cocycle} $\gamma$ of $\mathcal{G}$, that is, a smooth function
$$
\gamma : \mathcal{G} \times \mathcal{G} \to \mathbb{R}
$$
satisfying the {\em $2$-cocycle property}
\begin{align}
\label{eq:gamma_2-cocycle_property}
\gamma(g_1, g_2) + \gamma(g_1 g_2, g_3)
= \gamma(g_2, g_3) + \gamma(g_1, g_2 g_3), \qquad \forall g_1,g_2,g_3\in\mathcal{G}.
\end{align}
We also assume the normalization condition
\begin{align}
\label{eq:gamma_normalization}
\gamma(g,g^{-1})= 0, \qquad \forall g\in \mathcal{G}.
\end{align}
Then we define a central extension $\wh{\mathcal{G}}$ of $\mathcal{G}$ by $\mathbb{R}$, 
which is 
$$
\wh{\mathcal{G}} := \mathcal{G} \times \mathbb{R} = \{ (g,\alpha) \, | \, g\in \mathcal{G}, ~ \alpha \in \mathbb{R} \}
$$
as a manifold and whose multiplication is defined by
\begin{align}
\label{eq:central_extension_multiplication}
(g_1, \alpha_1) \cdot (g_2, \alpha_2) = (g_1 g_2, \, \alpha_1 +\alpha_2+ \gamma(g_1, g_2) ), \qquad \forall g_1, g_2 \in \mathcal{G}, \quad \forall\alpha_1, \alpha_2 \in \mathbb{R},
\end{align}
One can easily show using \eqref{eq:gamma_2-cocycle_property} that \eqref{eq:central_extension_multiplication} is indeed a well-defined group multiplication.

\vs

On the other hand, we consider an {\em $\mathbb{R}$-valued trivial $2$-cocycle} $\gamma_0$ of the subgroup $\mathcal{H}$, that is, a smooth function $\gamma_0 : \mathcal{H} \times \mathcal{H} \to \mathbb{R}$ given by
\begin{align}
\label{eq:trivial_cocycle}
\gamma_0(g_1,g_2) = \omega_0(g_1g_2) - \omega_0(g_1) - \omega_0(g_2), \qquad \forall g_1, g_2 \in \mathcal{H},
\end{align}
for some smooth function $\omega_0 : \mathcal{H} \to \mathbb{R}$, i.e. a {\em $1$-cochain} of $\mathcal{H}$. We say that $\gamma_0$ is the {\em $2$-coboundary} of $\omega_0$.
\begin{remark}
Although the `coboundary' $d\omega_0$ of a $1$-cochain $\omega_0(g)$ is usually thought of as being given by the formula $(d\omega_0)(g,h) = \omega_0(h) - \omega_0(gh) + \omega_0(g)$, it does not harm to use  \eqref{eq:trivial_cocycle} for our purposes.
\end{remark}

In the present paper, we shall require that the restriction of $\gamma$ to $\mathcal{H} \times \mathcal{H}$ coincides with $\gamma_0$:
\begin{align}
\label{eq:gamma_equals_gamma_0}
\gamma(g,h) = \gamma_0(g,h), \qquad \forall g, h \in \mathcal{H}.
\end{align}
Furthermore, we assume that the group $\mathcal{H}$ is perfect, i.e. it equals its commutator subgroup:
\begin{align}
\label{eq:perfect}
[\mathcal{H}, \mathcal{H}] = \mathcal{H}.
\end{align}

\vs

Under these conditions, we now wish to embed $\mathcal{H}$ as a normal subgroup of $\wh{\mathcal{G}}$, via the following embedding $\iota : \mathcal{H} \hookrightarrow \wh{\mathcal{G}}$
\begin{align}
\label{eq:embedding_of_H}
\iota : \mathcal{H}\ni h \longmapsto (h,\, \omega_0(h)) \in \wh{\mathcal{G}}.
\end{align}
Then we will be able to construct a central extension $\wh{\mathbf{G}}$ of $\mathbf{G} = \mathcal{G}/\mathcal{H}$ by $\mathbb{R}$ in the form of the quotient $\wh{\mathcal{G}}/\iota(\mathcal{H})$ (see Rem.\ref{rem:extension_by_U1} for an extension by ${\rm U}(1)$).

\begin{proposition}
\label{prop:embedding_group_homomorphism}
The map $\iota : \mathcal{H} \hookrightarrow \wh{\mathcal{G}}$ in \eqref{eq:embedding_of_H} is a group homomorphism.
\end{proposition}

\begin{proposition}
\label{prop:embedding_normality}
The image of $\mathcal{H}$ under the map $\iota$ in \eqref{eq:embedding_of_H} is normal in $\wh{\mathcal{G}}$.
\end{proposition}

In order to prove Prop.\ref{prop:embedding_group_homomorphism}, we first need a small lemma:
\begin{lemma}
\label{lem:small}
The normalization condition \eqref{eq:gamma_normalization} implies
\begin{align}
\label{eq:small_lemma}
\gamma(\mathrm{id}_\mathcal{G},g) = \gamma(g,\mathrm{id}_\mathcal{G}) = \omega_0(\mathrm{id}_\mathcal{G})= 0, \qquad \forall g\in \mathcal{G}.
\end{align}
\end{lemma}

\begin{proof}[Proof of Lem.\ref{lem:small}]
By putting $g=\mathrm{id}_\mathcal{G}$ in \eqref{eq:gamma_normalization} we get $\gamma(\mathrm{id}_\mathcal{G},\mathrm{id}_\mathcal{G})=0$. It is easy to show that \eqref{eq:gamma_2-cocycle_property} implies $\gamma(\mathrm{id}_\mathcal{G},\mathrm{id}_\mathcal{G}) = \gamma(\mathrm{id}_\mathcal{G},g) = \gamma(g,\mathrm{id}_\mathcal{G})$ for all $g\in \mathcal{G}$; hence $\gamma(\mathrm{id}_\mathcal{G},g)=\gamma(g,\mathrm{id}_\mathcal{G})=0$. Now we get $\omega_0(\mathrm{id}_\mathcal{G})=0$ by putting $g_1=\mathrm{id}_\mathcal{G}$ in \eqref{eq:trivial_cocycle}.
\end{proof}

\begin{proof}[Proof of Prop.\ref{prop:embedding_group_homomorphism}]
For $h_1,h_2 \in \mathcal{H}$ we have
\begin{align*}
(h_1h_2, \omega_0(h_1h_2)) \stackrel{\eqref{eq:trivial_cocycle}, \eqref{eq:gamma_equals_gamma_0}}{=} (h_1 h_2, \omega_0(h_1)+\omega_0(h_2) + \gamma(h_1, h_2))
\stackrel{\eqref{eq:central_extension_multiplication}}{=} (h_1, \omega_0(h_1)) \cdot (h_2, \omega_0(h_2)),
\end{align*}
so the map \eqref{eq:embedding_of_H} preserves the product. Since $\omega_0(\mathrm{id}_\mathcal{G})=0$ (from \eqref{eq:small_lemma}), this map sends $\mathrm{id}_\mathcal{H} \in \mathcal{H}$ to $\mathrm{id}_{\wh{\mathcal{G}}} \in \wh{\mathcal{G}}$, hence it is indeed a group homomorphism.
\end{proof}

\begin{proof}[Proof of Prop.\ref{prop:embedding_normality}] This proof mimicks that for the loop group case, where $\mathcal{G}$ is $C^\infty(\ol{D^2}, G)$, the group of smooth maps from the closed unit disc $\ol{D^2}$ to a simple Lie group $G$ with pointwise multiplication, and $\mathcal{H}$ is $C^\infty(\ol{D^2},G)_{S^1}$, the subgroup of all elements whose restriction to $S^1$ is a constant map having the identity of $G$ as its value there. The argument for the loop groups is in the work of Mickelsson \cite[Lemma 4.2.7]{M89}.

\vs

First, let $h \in \mathcal{H}$ and $(g,\, \alpha) \in \wh{\mathcal{G}}$. Since $\gamma(g,g^{-1})=0$ (see \eqref{eq:gamma_normalization}), it's easy to see that $(g^{-1}, -\alpha)$ is the inverse of $(g,\, \alpha)$. We will now show that  $(g,\, \alpha) (h, \omega_0(h)) (g^{-1}, - \alpha)$ is in the image of the map \eqref{eq:embedding_of_H}. Note that
\begin{align}
\nonumber
(g,\, \alpha) (h,\omega_0(h)) (g, -\alpha)
& = (gh, \alpha \, + \, \omega_0(h) + \gamma(g,h))(g^{-1}, - \alpha) \\
\label{eq:normality_proof_1}
& = (ghg^{-1}, \, \omega_0(h) + \gamma(g,h) +  \gamma(gh, g^{-1})).
\end{align}
Since $\mathcal{H} \lhd \mathcal{G}$, we get $ghg^{-1} \in \mathcal{H}$. In order for \eqref{eq:normality_proof_1} to be in the image of \eqref{eq:embedding_of_H}, we must have
\begin{align}
\label{eq:normality_proof_2}
\omega_0(h) + \gamma(g,h) + \gamma(gh,g^{-1}) = \omega_0(ghg^{-1}). 
\end{align}
So we should show that \eqref{eq:normality_proof_2} is true for all $h\in \mathcal{H}$ and all $g\in \mathcal{G}$. For each given $g\in \mathcal{G}$, define a mapping $\Delta_g : \mathcal{H} \to \mathbb{R}$ by
\begin{align}
\label{eq:Delta_definition}
\Delta_g(h) := \omega_0(h) + \gamma(g,h) + \gamma(gh,g^{-1}) -  \omega_0(ghg^{-1}), \qquad \forall h \in \mathcal{H},
\end{align}
which is continuous. We shall show that $\Delta_g$ satisfies $\Delta_g(h_1h_2) = \Delta_g(h_1) + \Delta_g(h_2)$ for all $h_1,h_2\in \mathcal{H}$, therefore $\Delta_g$ is a group homomorphism. Because $\mathbb{R}$ is abelian, this homomorphism $\Delta_g : \mathcal{H} \to \mathbb{R}$ should factor through the abelianization map $\mathcal{H} \to \mathcal{H}/[\mathcal{H}, \mathcal{H}]$, yielding a homomorphism 
$\mathcal{H}/[\mathcal{H}, \mathcal{H}] \to \mathbb{R}$. Since $\mathcal{H}/[\mathcal{H}, \mathcal{H}]=1$ by hypothesis \eqref{eq:perfect}, 
this map $\mathcal{H}/[\mathcal{H},\mathcal{H}] \to \mathbb{R}$, and hence also $\Delta_g : \mathcal{H} \to \mathbb{R}$, must be the trivial homomorphisms, sending everything to 
$0 \in \mathbb{R}$. So $\Delta_g(h) =0$, $\forall h \in \mathcal{H}$, hence \eqref{eq:normality_proof_2}.

\vs

Let's now show $\Delta_g(h_1h_2) = \Delta_g(h_1) + \Delta_g(h_2)$ for all $h_1,h_2\in \mathcal{H}$. From \eqref{eq:trivial_cocycle} and  \eqref{eq:gamma_equals_gamma_0} we have
\begin{align}
\label{eq:normality_proof_3}
\omega_0(h_1h_2) = \omega_0(h_1)+\omega_0(h_2) + \gamma(h_1,h_2).
\end{align}
From $\mathcal{H} \lhd \mathcal{G}$ we get $gh_1g^{-1}, gh_2g^{-1} \in \mathcal{H}$, and hence \eqref{eq:normality_proof_3} can be applied to these two elements:
\begin{align}
\label{eq:normality_proof_4}
\hspace{-5mm}
\omega_0(gh_1h_2g^{-1}) = \omega_0( (gh_1g^{-1}) (gh_2 g^{-1})) \stackrel{\eqref{eq:normality_proof_3}}{=} \omega_0(gh_1g^{-1})+\omega_0(gh_2g^{-1}) + \gamma(gh_1g^{-1},gh_2g^{-1}).
\end{align}
Now, using the definition \eqref{eq:Delta_definition} of $\Delta_g$ together with \eqref{eq:normality_proof_3} and \eqref{eq:normality_proof_4}, one observes
\begin{align*}
\Delta_g(h_1h_2) - \Delta_g(h_1) - \Delta_g(h_2) 
& \stackrel{\eqref{eq:Delta_definition}}{=} \omega_0(h_1h_2) + \gamma(g,h_1h_2) + \gamma(g h_1h_2, g^{-1}) - \omega_0(g h_1h_2g^{-1}) \\
& \quad - \omega_0(h_1) - \gamma(g,h_1) - \gamma(gh_1,g^{-1})   + \omega_0(gh_1g^{-1}) \\
& \quad - \omega_0(h_2) - \gamma(g,h_2) - \gamma(gh_2,g^{-1})    + \omega_0(gh_2g^{-1}) \\
& \stackrel{\eqref{eq:normality_proof_3}, \eqref{eq:normality_proof_4}}{=} \gamma(h_1,h_2) + \gamma(g,h_1h_2) - \gamma(g,h_1) - \gamma(g,h_2) + \gamma(gh_1h_2, g^{-1}) \\
& \quad  - \gamma(gh_1, g^{-1}) - \gamma(gh_2,g^{-1}) - \gamma(gh_1g^{-1}, gh_2g^{-1}).
\end{align*}
From the $2$-cocycle property of $\gamma$ (see \eqref{eq:gamma_2-cocycle_property}) we have
\begin{align}
\label{eq:normality_proof_5}
\gamma(h_1,h_2) + \gamma(g,h_1h_2) & = \gamma(g,h_1) + \gamma(gh_1, h_2), \\
\label{eq:normality_proof_6}
\gamma(gh_1,h_2) + \gamma(gh_1h_2,g^{-1}) & = \gamma(h_2,g^{-1}) + \gamma(gh_1, h_2g^{-1}), \\
\label{eq:normality_proof_7}
\gamma(gh_1,g^{-1}) + \gamma(gh_1g^{-1}, gh_2g^{-1}) & = \gamma(g^{-1}, gh_2g^{-1}) + \gamma(gh_1, g^{-1} g h_2 g^{-1}), \\
\label{eq:normality_proof_8}
\gamma(gh_2, g^{-1}) + \gamma(g^{-1}, gh_2g^{-1}) & = \gamma(g^{-1} gh_2, g^{-1}) + \gamma(g^{-1}, gh_2), \\
\label{eq:normality_proof_9}
\gamma(g,h_2) + \gamma(g^{-1}, gh_2) & = \gamma(g^{-1},g) + \gamma(g^{-1}g, h_2)
\end{align}
Using these, we try to prove $\Delta_g(h_1h_2) - \Delta_g(h_1) - \Delta_g(h_2) = 0$:
\begin{align*}
\Delta_g(h_1h_2) - \Delta_g(h_1) - \Delta_g(h_2)
& \stackrel{\eqref{eq:normality_proof_5}}{=}
\gamma(gh_1,h_2) - \gamma(g,h_2) + \gamma(gh_1h_2, g^{-1}) \\
& \quad  - \gamma(gh_1, g^{-1}) - \gamma(gh_2,g^{-1}) - \gamma(gh_1g^{-1}, gh_2g^{-1}) \\
& \stackrel{\eqref{eq:normality_proof_6}, \eqref{eq:normality_proof_7}}{=}
\gamma(h_2,g^{-1}) + \cancel{\gamma(gh_1, h_2g^{-1})} - \gamma(g,h_2)  \\
& \quad  - \gamma(g^{-1}, gh_2g^{-1}) - \cancel{\gamma(gh_1, h_2g^{-1})} - \gamma(gh_2,g^{-1}) \\
& \stackrel{\eqref{eq:normality_proof_8}}{=}
 - \gamma(g,h_2)  - \gamma(g^{-1}, gh_2)\stackrel{\eqref{eq:normality_proof_9}}{=}
-\gamma(g^{-1},g) - \gamma(\mathrm{id}_\mathcal{G},h_2)
\stackrel{\eqref{eq:gamma_normalization}, \eqref{eq:small_lemma}}{=} 0,
\end{align*}
as desired.
\end{proof}

\begin{remark}
Is there another proof of Prop.\ref{prop:embedding_normality} that does not rely on the perfectness of $\mathcal{H}$?
\end{remark}

\begin{remark}
\label{rem:extension_by_U1}
We could have constructed $\wh{\mathcal{G}}$ as a central extension of $\mathcal{G}$ by $\mathrm{U}(1) = \{ \alpha \in \mathbb{C} \, : \, |\alpha|=1\}$, with multiplication $(g_1,\alpha_1)\cdot(g_2,\alpha_2) = (g_1g_2, \alpha_1 \alpha_2 e^{{\rm i}\gamma(g_1,g_2)})$. Such is the case for loop groups, as certain proofs work only modulo $2\pi \mathbb{Z}$ there. As all proof in our case goes through in a genuine sense, not just modulo $2\pi\mathbb{Z}$, we use central extension by $\mathbb{R}$. If we want, we can translate all our results in terms of the central extension by $\mathrm{U}(1)$.
\end{remark}

\section{$2$d construction of $\mathbf{G} = {\rm Diff}_+(S^1)$ using diffeomorphisms of the closed disc $\ol{D^2}$}

In this section, we shall express the diffeomorphism group $\mathbf{G} = {\rm Diff}_+(S^1)$ of the one-dimensional manifold $S^1$ as the quotient $\mathcal{G}/\mathcal{H}$ of certain diffeomorphism groups of the two-dimensional manifold $D^2$ bounded by $S^1$. So, here we are going from $1$d geometry to $2$d geometry.

\subsection{Definition of our diffeomorphism groups $G$, $\mathcal{G}$ and  $\mathcal{H}$}
\label{sec:defin-our-diff}

\begin{definition}[relevant $2$-dimensional manifolds]
\label{def:relevant_2-dimensional_manifolds}
Let $\mathbb{C}$ be the complex plane, and let $\wh{\mathbb{C}} = \mathbb{C} \cup \{\infty\} = \mathbb{CP}^1$ be the usual Riemann sphere, or the {\em extended complex plane}. We sometimes identify this extended complex plane by $S^2$. Let $D^2$ be the open unit disc in $\mathbb{C} \subset \wh{\mathbb{C}}$, and $S^1 = \partial D^2$ the unit circle. Let $\ol{D^2} = D^2 \cup S^1$ be closed unit disc. Let the {\em outside unit disc} $(D^2)^\star$ be the complement in $\wh{\mathbb{C}} = S^2$ of $\ol{D^2}$, i.e.
$$
(D^2)^\star = \{ z \in \mathbb{C} : |z|>1 \} \cup \{\infty\} \subset \wh{\mathbb{C}}.
$$
Let $\ol{(D^2)^\star} = (D^2)^\star \cup S^1$ be the closure of the outside unit disc in $\wh{\mathbb{C}}$. 

\vs

We use $z$ for the usual complex coordinate for $\mathbb{C}$, while $x,y \in \mathbb{R}$ for the usual real and imaginary parts of $z$, that is, $z = x + iy$. We use $r \in [0,\infty)$ and $\theta\in \mathbb{R}/2\pi\mathbb{Z}$ for the polar coordinate, that is, $z = r \, e^{i\theta}$.
\end{definition}

The smooth structures and the orientations on the manifolds appearing in Def.\ref{def:relevant_2-dimensional_manifolds} are inherited from the usual ones on the Riemann sphere $S^2 = \wh{\mathbb{C}}$.

\begin{definition}[see e.g. \cite{B97}]
For a smooth oriented manifold $M$ possibly with boundary, denote by ${\rm Diff}_+(M)$ the group of all smooth orientation-preserving self-diffeomorphisms of $M$. An isotopy between two elements $f,g \in {\rm Diff}_+(M)$ is a smooth map $H : M \times [0,1] \to M$ such that for each $t\in [0,1]$, $H_t : M \to M$ defined by $H_t(x)=H(x,t)$ belongs to ${\rm Diff}_+(M)$, while $H_0 = f$ and $H_1=g$. We say $f,g$ are isotopic if there is an isotopy between them. For each subgroup $\mathscr{G}$ of ${\rm Diff}_+(M)$, denote by $\mathscr{G}_0$ the subset of all elements of $\mathscr{G}$ isotopic to the identify diffeomorphism.
\end{definition}

See \cite{B97} and references therein for a topology and a smooth structure on ${\rm Diff}_+(M)$.

\vspace{2mm}

Note that $\mathscr{G}_0$ is a subgroup of $\mathscr{G}$. In case $M=S^1$, it is well known that ${\rm Diff}_+(S^1)_0$ coincides with ${\rm Diff}_+(S^1)$, which we set to be our ${\bf G}$.

\begin{definition}
Let ${\bf G} := {\rm Diff}_+(S^1) = {\rm Diff}_+(S^1)_0$.
\end{definition}

We shall construct $\mathcal{G}$ and $\mathcal{H}$ as certain groups of diffeomorphisms of the unit disc $D^2$ bounded by $S^1$. We will apply the following definition to $M=\ol{D^2}$. 

\begin{definition}
\label{def:diffeomorphism_groups_of_bordered_manifolds}
Let $M$ be a smooth manifold with boundary. For any subgroup $\mathscr{G}$ of ${\rm Diff}_+(M)$, the {\em boundary-trivial subgroup} of $\mathscr{G}$ is defined as
\begin{align}
\label{eq:bt}
\mathscr{G}_{\rm bt} := \{ g \in \mathscr{G} \, : \, g|_{\partial M} = \mathrm{id}_{\partial M} \}.
\end{align}
\end{definition}
In the introduction, we denoted $\mathscr{G}_{\rm bt}$ by $\mathscr{G}_{\partial M}$.

\vs

The following notion is specific to the situation $M=\ol{D^2}$, but could also be applied to a general manifold with boundary if one chooses an identification of an open neighborhood of each boundary component, say $C$, with $C \times (0,1]$ where $C$ corresponds to $C\times \{1\}$.

\begin{definition}[asymptotically radial disc diffeomorphisms]
An element $g$ of ${\rm Diff}_+(\ol{D^2})$ is said to be \emph{asymptotically radial} if there exists $\epsilon \in (0,1)$ and $f\in {\rm Diff}_+(S^1)$, written in terms of the angular coordinate $f: \mathbb{R}/2\pi \mathbb{Z} \to \mathbb{R}/2\pi \mathbb{Z}$, such that
\begin{align}
  \nonumber
  g(r \, e^{{\rm i}\theta}) = r \, e^{{\rm i} f(\theta)}, \qquad \forall r \in (1-\epsilon,1], \quad \forall \theta \in \mathbb{R}/2\pi \mathbb{Z}.
\end{align}
Denote by ${\rm Diff}_{+{\rm ar}}(\ol{D^2})$ the set of all asymptotically radial elements of ${\rm Diff}_+(\ol{D^2})$.
\end{definition}
It is straightforward to see that ${\rm Diff}_{+{\rm ar}}(\ol{D^2})$ is a group.

\vs

We propose the following model of a diffeomorphism group of the closed disc $\ol{D^2}$ as our $\mathcal{G}$:
\begin{definition}[our model of the disc diffeomorphism group]
\label{def:our_model_of_the_disc_diffeomorphism_group}
Define
$$
\mathcal{G} := {\rm Diff}_{+{\rm ar}}(\ol{D^2})_0,
$$
i.e. the $C^\infty$-path component of the identity of ${\rm Diff}_{+{\rm ar}}(\ol{D^2})$. 

\vs

Define $\mathcal{H}$ to be its boundary-trivial subgroup:
\begin{align*}
  \mathcal{H} := \mathcal{G}_{\rm bt}.
\end{align*}
\end{definition}

\begin{proposition}
This $\mathcal{G}$ is a group. \qed
\end{proposition}

Now we must make sure that $\mathcal{G}/\mathcal{H} \cong \mathbf{G}$ and $[\mathcal{H},\mathcal{H}]=\mathcal{H}$; these are established in \S\ref{subsec:proof-requ-prop}.

\subsection{Proof of the required properties of our $\mathcal{G}$ and $\mathcal{H}$}
\label{subsec:proof-requ-prop}

In our proofs, we shall frequently use the following cutoff functions, whose existence is well known.

\begin{lemma}[cutoff functions]
\label{lem:xi}
For any positive real numbers $\epsilon_1,\epsilon_2$ such that $0  <\epsilon_1 < 1-\epsilon_2 < 1$, there exists a monotone increasing smooth function $\xi_{\epsilon_1,\epsilon_2} : [0,1] \to [0,1]$ satisfying $\xi_{\epsilon_1,\epsilon_2}\equiv 0$ on $[0,\epsilon_1]$ and $\xi_{\epsilon_1,\epsilon_2}\equiv 1$ on $[1-\epsilon_2,1]$. \qed
\end{lemma}

\begin{proposition}
\label{prop:quotient_map_is_surjection}
$\mathcal{H}$ is a closed normal subgroup of $\mathcal{G}$, and the natural map
$$
\mathcal{G} / \mathcal{H}  ~ \longrightarrow ~ {\rm Diff}_+(S^1) = \mathbf{G}, \qquad g\mathcal{H} \longmapsto g|_{S^1},
$$
given by `restriction to $S^1$', is an isomorphism.
\end{proposition}

\begin{proof}
`Restriction to $S^1$' gives the natural homomorphism $\mathcal{G} \subset {\rm Diff}_+(\ol{D^2}) \to {\rm Diff}_+(S^1) = {\bf G}$, and the kernel of this map is manifestly $\mathcal{H}$, telling us that $\mathcal{H}$ is a closed normal subgroup of $\mathcal{G}$. All we need to show is that this map $\mathcal{G} \to \mathbf{G}$ is surjective. Pick any $f\in {\rm Diff}_+(S^1)$. It suffices to construct $g\in \mathcal{G} \subset {\rm Diff}_+(\ol{D^2})$ whose restriction to $S^1$ is $f$. A key idea is the usage of a modification of the so-called `Alexander trick'; this modified version was used in \cite{Pf61} for the so-called `conformal welding', which is directly related to what we shall be doing later in the present paper.

\vs

Since ${\rm Diff}_+(S^1) = {\rm Diff}_+(S^1)_0$, there is a smooth isotopy $t\mapsto f_t$ in ${\rm Diff}_+(S^1)$ from $f_0 = {\rm id}_{S^1}$ to $f_1 = f$. Let each $f_t$ be given in terms of the angular coordinate $f_t : \mathbb{R}/2\pi\mathbb{Z} \to \mathbb{R}/2\pi\mathbb{Z}$. Lift it to $\til{f}_t : \mathbb{R} \to \mathbb{R}$, an orientation-preserving (i.e. $\til{f}_t'>0$) self-diffeomorphism of $\mathbb{R}$ satisfying $\til{f}_t(\theta+2n\pi) = \til{f}_t(\theta)+2n\pi$, $\forall n\in \mathbb{Z}$; for each $t$, the lift $f_t \leadsto \til{f}_t$ is unique up to the addition of an integer multiple of $2\pi$. Pick the lift of $f_0 = \mathrm{id}_{S^1}$ to be $\til{f}_0 = \mathrm{id}_\mathbb{R}$, i.e. $\til{f}_0(\theta)=\theta$, $\forall \theta \in \mathbb{R}$. Then for each $t \in (0,1]$ the lift $\til{f}_t$ of $f_t$ is uniquely determined by the smoothness of the isotopy $t\mapsto f_t$. One then observes that $t\mapsto \til{f}_t$ is a smooth isotopy in ${\rm Diff}_+(\mathbb{R})$.

\vs

Pick and fix any real number $\epsilon \in (0,1/2)$, and consider a smooth cutoff function $\xi := \xi_{\epsilon,\epsilon} : [0,1]\to [0,1]$ in Lem.\ref{lem:xi}. 
For each $t\in [0,1]$, define a map $g_t : \ol{D^2} \to \ol{D^2}$ as
\begin{align}
\label{eq:g_t_extension_of_til_f_t}
g_t (r \, e^{{\rm i}\theta}) :=  r \, e^{ {\rm i} \til{f}_{t \cdot \xi(r)} (\theta) }, \qquad \forall r \in [0,1], \quad \forall \theta\in \mathbb{R}/2\pi\mathbb{Z}.
\end{align}
For each $t$ and $r$, note that $\til{f}_{t\cdot \xi(r)} : \mathbb{R} \to \mathbb{R}$ descends to an orientation-preserving self-diffeomorphism of $\mathbb{R}/2\pi\mathbb{Z}$. For each $t$, it is easy to observe that $g_t$ is indeed a self-diffeomorphism of $\ol{D^2}$, and is asymptotically radial with the restriction to the boundary $S^1$ being $f_t \in {\rm Diff}_+(S^1)$. In particular, $g_t \in {\rm Diff}_{+{\rm ar}}(\ol{D^2})$ for each $t \in [0,1]$. For each fixed $r$ and $\theta$, the map $t\mapsto g_t(r \, e^{{\rm i}\theta})$ is $C^\infty$. Hence $t\mapsto g_t$ is a smooth isotopy in ${\rm Diff}_{+{\rm ar}}(\ol{D^2})$. From $\til{f}_0 = \mathrm{id}_\mathbb{R}$ one notes that $\til{f}_{0;r} = \mathrm{id}_\mathbb{R}$, $\forall r \in [0,1]$, hence $g_0=\mathrm{id}_{D^2}$. Therefore, the endpoint $g_1$ of the path $\{g_t\}_{t\in [0,1]}$ belongs to ${\rm Diff}_{+{\rm ar}}(\ol{D^2})_0 = \mathcal{G}$. Meanwhile, the boundary value of $g_1$ is $f_1=f$. Hence $g_1$ is an answer for the sought-for element $g\in\mathcal{G}$ satisfying $g|_{S^1}=f$.
\end{proof}

\begin{remark}
There is another famous extension of circle homeomorphisms to disc diffeomorphisms by Douady-Earle \cite{DE86}, which satisfies certain nice properties. However, the Douady-Earle extension is not asymptotically radial.
\end{remark}

It remains to show the perfectness of our $\mathcal{H}$. The perfectness and the simpleness of diffeomorphism groups are subjects of extensive study, one of the main investigators being William Thurston; we shall quote some known results that we need, instead of giving a proof here. What we should do is to identify our $\mathcal{H}$ with one of the diffeomorphism groups whose perfectness is proven.

\begin{definition}
For a self-homeomophism $f$ of a manifold $M$, the \emph{support} of $f$ is defined as the closure of the set $\{ x \in M \, | \, f(x) \neq x \}$ in $M$. We say $f$ is \emph{compactly supported} if the support of $f$ is a compact subset of $M$. Denote by ${\rm Diff}_{\rm c}(M)$ be the set of all compactly supported elements of ${\rm Diff}_+(M)$.
\end{definition}

\begin{lemma}
${\rm Diff}_{\rm c}(M)$ is a group. \qed
\end{lemma}

\begin{theorem}[Thurston \cite{T74}]
\label{thm:Thurston}
If $M$ is a smooth manifold, ${\rm Diff}_{\rm c}(M)_0$ is perfect.
\end{theorem}
A detailed proof of Thurston's theorem, as well as references to related results, can be found in Augustin Banyaga's book \cite{B97}; a fundamentally different proof is given by Haller-Rybicki-Teichmann \cite{HT03} \cite{HRT04}\footnote{\label{foot:Mann}as pointed out to the second author by Kathryn Mann}, which in turn is simplified by Kathryn Mann \cite{Ma}. One finds a self-contained proof for $M= \mathbb{R}^n$, $n\ge 2$, in a very short note of Mann \cite{Ma}. As we shall use the result for $M=D^2$ only, and since $D^2$ and $\mathbb{R}^2$ are diffeomorphic to each other\footnoteref{foot:Mann}, in principle we could just quote \cite{Ma} and nothing else. It remains to show:
\begin{proposition}
\label{prop:our_H_coincides_with_Diff_c_D2_0}
Our $\mathcal{H}$ in Def.\ref{def:our_model_of_the_disc_diffeomorphism_group} coincides with ${\rm Diff}_{\rm c}(D^2)_0$, as subgroups of ${\rm Diff}_+(D^2)$.
\end{proposition}
\begin{proof}
Note that ${\rm Diff}_{\rm c}(D^2)$ coincides with ${\rm Diff}_{+{\rm ar}}(\ol{D^2})_{\rm bt}$. Thus
$$
{\rm Diff}_{\rm c}(D^2)_0 = ({\rm Diff}_{+{\rm ar}}(\ol{D^2})_{\rm bt})_0 \subset ({\rm Diff}_{+{\rm ar}}(\ol{D^2})_0)_{\rm bt} = \mathcal{G}_{\rm bt} = \mathcal{H}.
$$
Conversely, let $h\in \mathcal{H} = \mathcal{G}_{\rm bt}$. Since $h \in \mathcal{G}= {\rm Diff}_{+{\rm ar}}(\ol{D^2})_0$, there is a smooth isotopy $t\mapsto h_t$ in ${\rm Diff}_{+{\rm ar}}(\ol{D^2})$ with $h_0 = {\rm id}_{\ol{D^2}}$ and $h_1 = h$. For each $t\in [0,1]$ let $f_t := (h_t)|_{S^1}$ be the boundary value of $h_t$, so that $f_0 = f_1 = {\rm id}_{S^1}$. For any chosen $\epsilon \in (0,1/2)$, apply the argument as in the proof of Prop.\ref{prop:quotient_map_is_surjection} to construct a smooth isotopy $t \mapsto g_t$ in ${\rm Diff}_{+{\rm ar}}(\ol{D^2})$ such that the boundary value of each $g_t$ is $f_t$, and that $g_0 = {\rm id}_{\ol{D^2}}$. For each $t\in [0,1]$, define $\ell_t := h_t \circ g_t^{-1}$, which belongs to ${\rm Diff}_{+{\rm ar}}(\ol{D}^2)$. Note that $(\ell_t)|_{S^1} = (h_t|_{S^1})\circ(g_t|_{S^1})^{-1} = f_t \circ f_t^{-1} = {\rm id}_{S^1}$ for each $t\in [0,1]$. Hence $\ell_t \in {\rm Diff}_{+{\rm ar}}(\ol{D^2})_{\rm bt} = {\rm Diff}_{\rm c}(D^2)$. So we constructed a smooth isotopy $t\mapsto \ell_t$ in ${\rm Diff}_{\rm c}(D^2)$ with $\ell_0 = h_0 \circ g_0^{-1} = {\rm id}_{D^2}$ and $\ell_1 = h_1 \circ g_1^{-1} = h \circ g_1^{-1}$; thus $h\circ g_1^{-1} \in {\rm Diff}_{\rm c}(D^2)_0$. Since ${\rm Diff}_{\rm c}(D^2)_0$ is a group, it suffices to show $g_1 \in {\rm Diff}_{\rm c}(D^2)_0$, in order to show $h\in {\rm Diff}_{\rm c}(D^2)_0$.

\vs

Note that $g_1$ is given by the formula \eqref{eq:g_t_extension_of_til_f_t} for $t=1$, with $\til{f}_1 : \mathbb{R} \to \mathbb{R}$ being a lift of $f_1 = {\rm id}_{S^1}$, hence there is some $n\in \mathbb{Z}$ s.t. $\til{f}_1(\theta) = \theta + 2n\pi$, $\forall \theta \in \mathbb{R}$. We build a smooth isotopy $t\mapsto \til{F}_t$ in ${\rm Diff}_+(\mathbb{R})$ by
$$
\til{F}_t(\theta) := \theta + (2n\pi) t, \qquad \forall t \in [0,1], ~ \forall \theta \in \mathbb{R},
$$
from $\til{F}_0 = {\rm id}_\mathbb{R}$ to $\til{F}_1 = \til{f}_1$. We then proceed as in the proof of Prop.\ref{prop:quotient_map_is_surjection}, to define $\xi = \xi_{\epsilon,\epsilon}$ by Lem.\ref{lem:xi} with the same $\epsilon\in (0,1/2)$ as above, and mimic \eqref{eq:g_t_extension_of_til_f_t} to define $k_t : \ol{D^2} \to \ol{D^2}$ as
\begin{align*}
k_t (r \, e^{{\rm i}\theta}) := r \, e^{{\rm i} \til{F}_{t \cdot \xi(r)} (\theta) }, \qquad \forall r \in [0,1], \quad \forall \theta\in \mathbb{R}/2\pi\mathbb{Z}.
\end{align*}
Then $t\mapsto k_t$ is a smooth isotopy in ${\rm Diff}_{+{\rm ar}}(\ol{D^2})$, $k_0 = \mathrm{id}_{\ol{D^2}}$, and $k_1 = g_1$. For any $\alpha \in \mathbb{R}$ (or $\alpha\in\mathbb{R}/2\pi\mathbb{Z}$), let $R_\alpha : \ol{D^2} \to \ol{D^2}$ be the counterclockwise rotation by angle $\alpha$, i.e.
\begin{align*}
  R_\alpha(r \, e^{{\rm i}\theta}) := r \, e^{{\rm i}(\theta+\alpha)}, \qquad \forall r \in [0,1], \quad \forall \theta \in \mathbb{R}/2\pi\mathbb{Z}.
\end{align*}
Note that $R_\alpha \in {\rm Diff}_{+{\rm ar}}(\ol{D^2})$; its smoothness on $D^2$ can be seen e.g. by recalling that it is a conformal automorphism of $D^2$. Note also that $k_t$ coincides with $R_{2n\pi t}$ on the annulus $1-\epsilon<r<1$, i.e. an open neighborhood (in $\ol{D^2}$) of the boundary $S^1$. Consider the smooth isotopy $t\mapsto R_{2n\pi t}$ in ${\rm Diff}_{+{\rm ar}}(\ol{D^2})$. For each $t\in [0,1]$ define $K_t := k_t \circ R_{2n\pi t}^{-1}$, which is an element of $\mathrm{Diff}_{+{\rm ar}}(\ol{D^2})$, and which is identity on the annulus $1-\epsilon<r<1$. Thus the support of $K_t$ is a compact subset of $D^2$, so $K_t \in \mathrm{Diff}_{\rm c}(D^2)$. Therefore, $t\mapsto K_t$ is a smooth isotopy in $\mathrm{Diff}_{\rm c} (D^2)$, where $K_0 = k_0 \circ R_0^{-1} = \mathrm{id}_{D^2}$, and $K_1 = k_1 \circ R_{2n\pi}^{-1} = g_1 \circ R_{2n \pi}^{-1} = g_1 \circ \mathrm{id}_{D^2}^{-1} = g_1$. This shows that $g_1 \in \mathrm{Diff}_{\rm c}(D^2)_0$, finishing the proof. \qed
\end{proof}

Now, Prop.\ref{prop:our_H_coincides_with_Diff_c_D2_0}, together with Thm.\ref{thm:Thurston} applied to $M=D^2$, yields the perfectness of our $\mathcal{H}$. 

\begin{corollary}
\label{cor:our_H_is_perfect}
Our $\mathcal{H}$ in Def.\ref{def:our_model_of_the_disc_diffeomorphism_group} is perfect. \qed
\end{corollary}

\section{The $2$-cocycle $\gamma$ of the disc diffeomorphism group $\mathcal{G}$}

We shall be more general, and construct a group $2$-cocycle of the diffeomorphism group of any smooth manifold with boundary. 
\begin{definition}
\label{def:G_M_theta_M}
Let $M$ be an oriented smooth $n$-dimensional manifold possibly with a smooth $(n-1)$-dimensional boundary, smoothly embedded as a compact subset in the standard Euclidean space $\mathbb{R}^n$ equipped with the coordinate system ${\bf x} = \left(\begin{smallmatrix} x_1 \\ \vdots \\ x_n\end{smallmatrix}\right)$. 

\vs

For $g\in {\rm Diff}_+(M)$, denote by $J_M(g)$ the Jacobian matrix of $g$ with respect to ${\bf x}$:
\begin{align}
\label{eq:J_M}
J_M(g)({\bf x}) := \left(\partial_{x_j} g_i \right)_{i,j}
\end{align}
That is, $J_M(g)$ is a ${\rm GL}(n,\mathbb{R})$-valued smooth function on $M$, with the $(i,j)$-component being $\partial_{x_j} g_i = \frac{\partial g_i}{\partial x_j}$, where $g_i$ is the $i$-th component function of $g$. Denote by $\theta_M(g)$ the following $\mathfrak{gl}(n,\mathbb{R})$-vaued $1$-form on $M$:
\begin{align}
\label{eq:theta_M}
\theta_M(g) := J_M(g)^{-1} dJ_M(g).
\end{align}
\end{definition}

Each $g\in {\rm Diff}_+(M)$ provides a smooth map $J_M(g) : M \to {\rm GL}(n,\mathbb{R})$, and $\theta_M(g)$ is just the pullback of the {\em Maurer-Cartan $1$-form} on the Lie group ${\rm GL}(n,\mathbb{R})$, which is ${\rm g}^{-1} d{\rm g}$ for the matrix-valued variable ${\rm g}$ running in ${\rm GL}(n,\mathbb{R})$. 

\vs

With this $1$-form we construct of a group $2$-cocycle of ${\rm Diff}_+(M)$:

\begin{proposition}[the $2$-cocycle $\gamma_M$]
\label{prop:gamma_M_2-cocycle}
Let $M$ be as in Def.\ref{def:G_M_theta_M}. The function $\gamma_M : {\rm Diff}_+(M) \times {\rm Diff}_+(M) \to \mathbb{R}$ defined by
\begin{align}
\label{eq:gamma_M_definition}
\gamma_M (g,h) := \int_M {\rm tr}_{\mathrm{mat}_n} \left( \theta_M(g) \wedge \theta_M(h^{-1}) \right), \qquad \forall g,h\in {\rm Diff}_+(M),
\end{align}
where ${\rm tr}_{\mathrm{mat}_n}$ means the ordinary matrix trace of $n$ by $n$ matrices,
is an $\mathbb{R}$-valued group $2$-cocycle of ${\rm Diff}_+(M)$, and satisfies the normalization condition \eqref{eq:gamma_normalization}.
\end{proposition}

\vs

The rest of the present subsection is dedicated to a proof of Prop.\ref{prop:gamma_M_2-cocycle}. Whenever $M$ is unambiguous from the context, we omit the subscript $M$. For convenience, for a $\frak{gl}(n,\mathbb{R})$-valued $p$-form $A$ on $M$ and $f\in {\rm Diff}_+(M)$, we often write $f^*A$ as $A\circ f$. We begin with a useful lemma.

\begin{lemma}[composition law for $\theta_M$]
\label{lem:theta_composition}
For $g,h\in {\rm Diff}_+(M)$, one has
\begin{align}
\label{eq:theta_composition}
\theta(g\circ h) = J(h)^{-1} \, (\theta(g) \circ h) \, J(h) + \theta(h).
\end{align}
\end{lemma}

\begin{proof}
Using the chain rule we have
\begin{align}
\label{eq:chain_rule}
J(g\circ h) = (J(g) \circ h) \, J(h),
\end{align}
and thus its differential is
\begin{align*}
dJ(g\circ h) = (d(J(g) \circ h) ) \, J(h) + (J(g) \circ h) \, dJ(h).
\end{align*}
Hence we can combine and get
\begin{align*}
\theta(g\circ h) & = J(g\circ h)^{-1} \, dJ(g\circ h) \\
& = J(h)^{-1} \, (J(g) \circ h)^{-1} \, d(J(g) \circ h) \, J(h) + J(h)^{-1} \, dJ(h).
\end{align*}
Since pullback commutes with $d$, we can simplify $(J(g) \circ h)^{-1} \, d(J(g) \circ h)$ as $\theta(g) \circ h$.
\end{proof}

\begin{lemma}
\label{lem:d_inverse}
Let $X$ be a ${\rm GL}(n,\mathbb{R})$-valued function on $M$, and $X^{-1}$ be the ${\rm GL}(n,\mathbb{R})$-valued function on $M$ whose value is the inverse matrix of the value of $X$. The differential of $X^{-1}$ is
\begin{align}
\label{eq:d_inverse}
d(X^{-1}) = -X^{-1} d X \, X^{-1}.
\end{align}
\end{lemma}

\begin{proof}
From $XX^{-1} = I$ where $I$ is the (constant) identity matrix, we have
\begin{align}
0 = (dX) X^{-1} + X (d(X^{-1})).
\end{align}
\end{proof}

\begin{lemma}[Jacobian matrix of inverse]
\label{lem:compose_h_inverse}
One can easily observe
\begin{align}
\label{eq:lem_compose_h_inverse}
J(h)\circ h^{-1} = J(h^{-1})^{-1}, \qquad \forall \, h \in {\rm Diff}_+(M).
\end{align}
By taking the exterior derivative of \eqref{eq:lem_compose_h_inverse} we get:
\begin{align}
\label{eq:d_J_compose_inverse}
dJ(h)\circ h^{-1} = - J(h^{-1})^{-1} \, dJ(h^{-1}) \, J(h^{-1})^{-1}, \qquad \forall \, h \in {\rm Diff}_+(M).
\end{align}
\end{lemma}

\begin{proof}
We have $id = h \circ h^{-1}$, so $I = J(h\circ h^{-1})$, where $I$ means the constant function with its value equal to the identity matrix. Then by the chain rule we have
\begin{align}
I = (J(h) \circ h^{-1}) \, J(h^{-1}). 
\end{align}
For \eqref{eq:d_J_compose_inverse}, using Lem.\ref{lem:d_inverse} and keeping in mind that the exterior derivative $d$ commutes with pullback, we can apply $d$ to both sides of \eqref{eq:lem_compose_h_inverse} and get \eqref{eq:d_J_compose_inverse}.
\end{proof}

Now we're ready to prove Prop.\ref{prop:gamma_M_2-cocycle}.

\begin{proof}[Proof of Prop.\ref{prop:gamma_M_2-cocycle}]
The well-defineness of the formula \eqref{eq:gamma_M_definition} is clear. Let's first prove the group $2$-cocycle property \eqref{eq:gamma_2-cocycle_property} of our $\gamma_M$ \eqref{eq:gamma_M_definition}:
\begin{align}
\label{eq:gamma_2-cocycle_property_to_prove}
\gamma_M(gh,k) + \gamma_M(g,h) = \gamma_M(g,hk) + \gamma_M(h,k),
\end{align}
for $g,h,k \in {\rm Diff}_+(M)$. Using the composition law \eqref{eq:theta_composition} for $\theta$ (Lem.\ref{lem:theta_composition}), we get
\begin{align*}
\theta(gh) \wedge \theta(k^{-1}) & = J(h)^{-1} \, (\theta(g) \circ h) \, J(h) \wedge \theta(k^{-1})
+ \theta(h) \wedge \theta(k^{-1}), \\
\theta(g) \wedge \theta(k^{-1} h^{-1}) & = \theta(g) \wedge J(h^{-1})^{-1} \,  (\theta(k^{-1})\circ h^{-1}) \, J(h^{-1}) + \theta(g) \wedge \theta(h^{-1}).
\end{align*}
In view of the definition \eqref{eq:gamma_M_definition} of $\gamma = \gamma_M$, the above observations about $\theta(gh) \wedge \theta(k^{-1})$ and $\theta(g) \wedge \theta(k^{-1} h^{-1})$ tell us that in order to prove \eqref{eq:gamma_2-cocycle_property_to_prove} it suffices to show:
\begin{align}
\label{eq:gamma_2-cocycle_proof_eq1}
& \int_M {\rm tr}_{\mathrm{mat}_n} \left( J(h)^{-1} \, (\theta(g)\circ h) \, J(h) \wedge \theta(k^{-1}) \right)  \\
\label{eq:gamma_2-cocycle_proof_eq2}
& = \int_M {\rm tr}_{\mathrm{mat}_n} \left( \theta(g) \wedge J(h^{-1})^{-1} \, (\theta(k^{-1}) \circ h^{-1}) \, J(h^{-1}) \right).
\end{align}
We will pull-back the integrand of \eqref{eq:gamma_2-cocycle_proof_eq2} by $h$. Recalling Lem.\ref{lem:compose_h_inverse} which says $J(h)\circ h^{-1} = J(h^{-1})^{-1}$, we get
\begin{align}
\hspace{-3mm}
\mbox{integral in } \eqref{eq:gamma_2-cocycle_proof_eq2}
= \int_{h^{-1}(M)} {\rm tr}_{\mathrm{mat}_n} \left( (\theta(g) \circ h) \wedge \, J(h) \, \theta(k^{-1}) \, J(h)^{-1} \right)
\end{align}
which is easily seen to be same as \eqref{eq:gamma_2-cocycle_proof_eq1}.

\vs

For the normalization condition \eqref{eq:gamma_normalization} of $\gamma_M$, observe for any $g,h\in{\rm Diff}_+(M)$ that
\begin{align*}
\gamma_M(g,h) \stackrel{\eqref{eq:gamma_M_definition}}{=} \int_M {\rm tr}_{\mathrm{mat}_n} (\theta(g) \wedge \theta(h^{-1}))
\,=\, - \int_M {\rm tr}_{\mathrm{mat}_n} (\theta(h^{-1}) \wedge \theta(g))
\stackrel{\eqref{eq:gamma_M_definition}}{=} - \gamma_M(h^{-1}, g^{-1}),
\end{align*}
which immediately implies $\gamma_M(g,g^{-1})=0$ (put $h=g^{-1}$). In fact, it is easy to show that the condition \eqref{eq:gamma_normalization} is equivalent to the condition $\gamma(g,h) = - \gamma(h^{-1},g^{-1})$, $\forall g,h \in {\rm Diff}_+(M)$.
\end{proof}

\begin{remark}
A priori, it is not clear if $\gamma_M$ defined in \eqref{eq:gamma_M_definition} is a non-trivial $2$-cocycle of ${\rm Diff}_+(M)$, i.e. if it is not the coboundary of a $1$-cochain.
\end{remark}

Applying this construction to the subgroup $\mathcal{G} = {\rm Diff}_{+{\rm ar}}(\ol{D^2})_0$ of ${\rm Diff}_+(\ol{D^2})$ defined in Def.\ref{def:our_model_of_the_disc_diffeomorphism_group}, we obtain a (possibly nontrivial) $\mathbb{R}$-valued group $2$-cocycle $\gamma$ of our disc diffeomorphism group $\mathcal{G}$.

\begin{definition}[group $2$-cocycle of our disc diffeomorphism group]
\label{def:our_gamma}
For a fixed scalar $c_0 \in \mathbb{R}$, we define the $\mathbb{R}$-valued $2$-cocycle $\gamma$ of our disc diffeomorphism group $\mathcal{G}$ in Def.\ref{def:our_model_of_the_disc_diffeomorphism_group} as
\begin{align}
\label{eq:gamma_definition}
\hspace{-3mm}
\gamma(g,h) := 3c_0 \, \gamma_{\ol{D^2}}(g,h) \stackrel{\eqref{eq:gamma_M_definition}}{=} 3c_0 \int_{D^2} {\rm tr}_{\mathrm{mat}_2} (\theta_{D^2}(g) \wedge \theta_{D^2}(h^{-1})),
\end{align}
for all $g,h\in \mathcal{G}$, where $\theta_{D^2}= \theta_{\ol{D^2}}$ is as in Def.\ref{def:G_M_theta_M} using the usual embedding of $D^2$ into the complex plane with the real coordinate system $\smallvectwo{x}{y}$, as described in Def.\ref{def:relevant_2-dimensional_manifolds}.
\end{definition}

\begin{remark}[suggestion for other definitions of $\mathcal{G}$]
The reason why we considered the asymptotically-radial and isotopic-to-identity condition for our subgroup $\mathcal{G}$ of ${\rm Diff}_+(\ol{D^2})$ in the present paper is mainly to guarantee that its boundary trivial subgroup $\mathcal{G}_{\rm bt} = \mathcal{H}$ is perfect (Cor.\ref{cor:our_H_is_perfect}). One may consider another subgroup of ${\rm Diff}_+(\ol{D^2})$ as $\mathcal{G}$ provided that $\mathcal{G}_{\rm bt} = \mathcal{H}$ is perfect.
\end{remark}

\begin{remark}
We are informed by Sam Nariman that the cocycle in eq.\eqref{eq:gamma_M_definition} has been considered in \cite{Billig} for the case when $M$ is a torus $\mathbb{T}^N$.
\end{remark}

\section{$3$d construction of a trivial $2$-cocycle $\gamma_0$ of $\mathcal{H}$}
\label{sec:gamma_0}

The next step is to construct a trivial $2$-cocycle $\gamma_0$ of $\mathcal{H}$, defined as the coboundary of some $1$-cochain (i.e. just a function on $\mathcal{H}$) $\omega_0$ as in \eqref{eq:trivial_cocycle}. In order to do this, we first `extend' each element $h$ of our boundary-trivial disc diffeomorphism group $\mathcal{H} = {\rm Diff}_{\rm c}(D^2)_0$ to a self-diffeomorphism of $S^2\approx \wh{\mathbb{C}}$ (recall that $\ol{D^2} \subset S^2$), and in turn, to a self-diffeomorphism of the open unit $3$-ball $B^3 \subset \mathbb{R}^3$ bounded by the unit sphere $S^2$ in $\mathbb{R}^3$. We shall then define $\omega_0(g)$ using this $3$d diffeomorphism obtained from $g$.

\subsection{$2$d$\to 3$d extension, and the Wess-Zumino terms}
\label{subsec:extension_from_S2_to_B3}

We first extend the boundary-trivial self-diffeomorphisms of the disc $D^2$ to self-diffeomorphisms of the Riemann sphere $S^2$.
\begin{definition}
\label{def:B_h}
For each $h\in \mathcal{H} = {\rm Diff}_{\rm c}(D^2)_0$ ($\because$Prop.\ref{prop:our_H_coincides_with_Diff_c_D2_0}), define $B_h:\wh{\mathbb{C}} \to \wh{\mathbb{C}}$ as
\begin{align}
\label{eq:B_on_subgroup}
B_h = \left\{ \begin{array}{ll} h & \mbox{on $D^2$}, \\ {\rm id} & \mbox{on $\ol{(D^2)^\star}$}, \end{array} \right.
\end{align}
where $\wh{\mathbb{C}}\approx S^2$, $D^2$, $(D^2)^\star$, $\ol{(D^2)^\star}$ are as in Def.\ref{def:relevant_2-dimensional_manifolds}.
\end{definition}

As the support of $h \in \mathcal{H}$ is a compact subset of the open unit disc $D^2$, we see that the self-diffeomorphism $h$ is identity on the complement in $D^2$ of its support; hence $B_h$ is $C^\infty$ on the whole $\wh{\mathbb{C}}$, and its support in $\wh{\mathbb{C}} \approx S^2$ is compact and does not contain the point $\infty$. 

\begin{definition}
An element of ${\rm Diff}_+(S^2)$ is said to be \emph{compactly supported from infinity} if its support in $\wh{\mathbb{C}} \approx S^2$ is compact and does not contain the point $\infty$. The set of all such elements is denoted by ${\rm Diff}_{\rm ci}(S^2)$.
\end{definition}

\begin{remark}
One can identify ${\rm Diff}_{\rm ci}(S^2)$ with ${\rm Diff}_{\rm c}(\mathbb{C})$.
\end{remark}

Then ${\rm Diff}_{\rm ci}(S^2)$ is a group. The following two observations are immediate.

\begin{lemma}
\label{lem:B_h_smoothness}
One has $B_h \in {\rm Diff}_{\rm ci}(S^2)_0 \subset {\rm Diff}_+(S^2)_0$ 
for all $h \in \mathcal{H}$. \qed
\end{lemma}

\begin{lemma}
\label{lem:composition_of_B_for_subgroup}
One has $B_{gh} = B_g \circ B_h$ for all $g,h\in \mathcal{H}$. \qed
\end{lemma}

In order to define the $1$-cochain $\omega_0$ of $\mathcal{H}$, we need yet another extension from $S^2$ to $B^3$. Here $B^3$ denotes the open unit ball in $\mathbb{R}^3$, while its boundary is $\partial B^3$ is identified with $S^2$, the Riemann sphere $\wh{\mathbb{C}}$. We still use the coordinate $\smallvectwo{x}{y}$ for $S^2$, or more precisely, for $S^2 \setminus\{\infty\}$, as described in Def.\ref{def:relevant_2-dimensional_manifolds}. We shall carefully avoid the subtlety that might possibly arise near the point $\infty$ of $S^2$. Thinking of the embedding $S^2 = \partial B^3 \subset B^3 \subset \mathbb{R}^3$, we choose to use a radial coordinate $\rho \in [0,1]$ so that $\smallvecthree{\rho}{x}{y}$ gives a coordinate system for the closed $3$-ball $\ol{B^3} = B^3 \cup S^2$; at $\rho=1$ we have the identification $\partial B^3 \leftrightarrow S^2$, $\smallvecthree{1}{x}{y} \leftrightarrow \smallvectwo{x}{y}$, and $\rho=0$ means the origin of $\mathbb{R}^3$, where one needs to be careful when dealing with the smoothness. Then we have a chain of inclusions inside $\mathbb{R}^3$:
\begin{align}
S^1 = \partial D^2 \subset \ol{D^2} = D^2 \cup S^1 \subset \wh{\mathbb{C}} \approx S^2 = \partial B^3 \subset \ol{B^3} = B^3 \cup S^2.
\end{align}

As $B^3$ has a radial coordinate, one can define the `asymptotically radial' self-diffeomorphisms of $B^3$. Let us focus on orientation-preserving ones.
\begin{definition}
An element $\mathbf{B}$ of ${\rm Diff}_+(\ol{B^3})$ is said to be \emph{asymptotically radial} if there exists a self-diffeomorphism $B$ of $S^2$ and $\epsilon \in (0,1)$ such that
\begin{align*}
  \mathbf{B} \vectwo{\rho}{\smallvectwo{x}{y}} = \vectwo{\rho}{B\smallvectwo{x}{y}}, \qquad \forall \rho \in (1-\epsilon,1], \quad \forall \smallvectwo{x}{y} \in S^2.
\end{align*}
Let ${\rm Diff}_{+{\rm ar}}(\ol{B^3})$ be the set of all asymptotically radial elements of ${\rm Diff}_+(\ol{B^3})$.
\end{definition}

\begin{definition}
\label{def:lci}
An element $\mathbf{B}$ of ${\rm Diff}_+(\ol{B^3})$ is said to be \emph{layered with compact-from-infinity supports} if it is given by
\begin{align*}
\mathbf{B}\vectwo{\rho}{\smallvectwo{x}{y}} = \vectwo{\rho}{B_\rho\smallvectwo{x}{y}}, \qquad \forall \rho \in [0,1], \quad \forall \smallvectwo{x}{y} \in S^2,
\end{align*}
for some elements $B_\rho$ of ${\rm Diff}_{\rm ci}(S^2)$ for $\rho\in [0,1]$. Let ${\rm Diff}_{\rm lci}(\ol{B^3})$ be the set of all such elements of ${\rm Diff}_+(\ol{B^3})$.
\end{definition}

Then ${\rm Diff}_{+{\rm ar}}(\ol{B^3})$ and ${\rm Diff}_{\rm lci}(\ol{B^3})$ are groups. Notice that we are requiring $\mathbf{B}|_{S^2} \in {\rm Diff}_{\rm ci}(S^2)$ for an element $\mathbf{B}$ of ${\rm Diff}_{\rm lci}(\ol{B^3})$.

\begin{definition}
\label{def:diff_B3_star}
Define
\begin{align}
\nonumber
{\rm Diff}_{+*}(\ol{B^3}) := {\rm Diff}_{+{\rm ar}}(\ol{B^3}) \cap {\rm Diff}_{\rm lci}(\ol{B^3}).
\end{align}
\end{definition}

\begin{proposition}
\label{prop:S2_extension_to_B3}
Any element $B \in {\rm Diff}_{\rm ci}(S^2)_0$ can be extended to an element $\til{B}$ of ${\rm Diff}_{+*}(\ol{B^3})_0$.
\end{proposition}
\begin{proof}
A $3$-dimensional analog of the modified Alexander trick used in the proof of Prop.\ref{prop:quotient_map_is_surjection} would work. Here we present yet another slight modification of that idea, which in fact could also have been adapted to the proof of Prop.\ref{prop:quotient_map_is_surjection}. Pick any $\epsilon \in (0,1/2)$, and let $\xi := \xi_{\epsilon,\epsilon} : [0,1] \to [0,1]$ be given as in Lem.\ref{lem:xi}. By assumption, there is a smooth isotopy $t\mapsto B_t$ in ${\rm Diff}_{\rm ci}(S^2)$ such that $B_0 = \mathrm{id}_{S^2}$ and $B_1 = B$; in particular, each $B_t$ is orientation-preserving. For each $t\in [0,1]$ define a map $\til{B}_t : \ol{B^3} \to \ol{B^3}$ as
\begin{align*}
    \til{B}_t \vectwo{\rho}{\smallvectwo{x}{y}} = \vectwo{\rho}{B_{t\cdot \xi(\rho)} \smallvectwo{x}{y}}, \qquad \forall \rho \in [0,1], \qquad \forall \smallvectwo{x}{y} \in S^2.
\end{align*}
From the properties of $\xi=\xi_{\epsilon,\epsilon}$ and the smooth isotopy $t\mapsto B_t$, it is easy to see for each $t\in [0,1]$ that $\til{B}_t$ is a self-homeomorphism of $\ol{B^3}$, is a diffeomorphism on $\ol{B^3} \setminus \{\smallvecthree{0}{0}{0}\}$, is identity near the origin, and is asymptotically radial with the boundary value $B_t \in {\rm Diff}_+(S^2)$; hence $\til{B}_t \in {\rm Diff}_{+{\rm ar}}(\ol{B^3}) \subset {\rm Diff}_+(\ol{B^3})$. It is manifest from the formula for $\til{B}_t$ and Def.\ref{def:lci} that $\til{B}_t \in {\rm Diff}_{\rm lci}(\ol{B^3})$; hence $\til{B}_t \in {\rm Diff}_{+*}(\ol{B^3})$. Moreover, $\til{B}_0 = \mathrm{id}_{\ol{B^3}}$, and $t\mapsto \til{B}_t$ is  a smooth isotopy in ${\rm Diff}_{+*}(\ol{B^3})$. Hence $\til{B}_1 \in {\rm Diff}_{+*}(\ol{B^3})_0$, with $\til{B}_1|_{S^2} = B$. Thus $\til{B} := \til{B}_1$ is a sought-for extension of $B$.
\end{proof}

\begin{remark}
One can use the higher dimensional Douady-Earle extension \cite[\S11]{DE86} to extend any $B\in {\rm Diff}_+(S^2)$ to an element of $\til{B}$ of ${\rm Diff}_+(B^3)$. However, for our purposes, it is not necessary to quote such a deep result.
\end{remark}

For any element $h \in \mathcal{H} = {\rm Diff}_{\rm c}(D^2)_0$, consider $B_h : S^2 \to S^2$ constructed in Def.\ref{def:B_h}, which is an element of ${\rm Diff}_{\rm ci}(S^2)_0$ (Lem.\ref{lem:B_h_smoothness}). Now, consider {\em any} extension of $B_h$ to an element of ${\rm Diff}_{+*}(\ol{B^3})_0$, i.e. choose any
\begin{align}
\label{eq:B_h_to_B3}
\til{B}_h \in {\rm Diff}_{+*}(\ol{B^3})_0 \quad \mbox{such that} \quad \til{B}_h|_{S^2} = B_h.
\end{align}
Prop.\ref{prop:S2_extension_to_B3} tells us that there exists at least one such extension. We shall define the function $\omega_0$ on $\mathcal{H}$ using any chosen extensions $h \in \mathcal{H} \leadsto \til{B}_h \in {\rm Diff}_{+*}(\ol{B^3})_0$.

\begin{definition}[the ``Wess-Zumino" term]
\label{def:WZ}
Define the function $W : {\rm Diff}_+(\ol{B^3})\to \mathbb{R}$ by
\begin{align}
\label{eq:W_definition}
W({\bf g}) := \int_{B^3} {\rm tr}_{{\rm mat}_3} \left( (\theta_{B^3}({\bf g}))^{\wedge 3} \right), \qquad \forall {\bf g}\in {\rm Diff}_+(\ol{B^3}),
\end{align}
where $\theta_{B^3} = \theta_{\ol{B^3}}$ is defined as in \eqref{eq:theta_M} for the coordinate system $\smallvecthree{\rho}{x}{y}$ for $B^3$, $c_0$ the same constant as in Def.\ref{def:our_gamma}, and ${\rm tr}_{\mathrm{mat}_3}$ is the trace of the $3$ by $3$ matrices. Define the function $\omega_0$ on $\mathcal{H}$ as
\begin{align}
\label{eq:omega_definition}
\omega_0(h) := c_0 \, W(\til{B}_h) = c_0 \int_{B^3} {\rm tr}_{{\rm mat}_3} \left( (\theta_{B^3}(\til{B}_h))^{\wedge 3}\right), \qquad \forall h\in \mathcal{H} = {\rm Diff}_{\rm c}(D^2)_0,
\end{align}
where $B_h$ is as defined in \eqref{eq:B_on_subgroup} and $\til{B}_h$ is any chosen extension \eqref{eq:B_h_to_B3} of $B_h$ to ${\rm Diff}_{+*}(\ol{B^3})_0$.
\end{definition}

We shall prove shortly:
\begin{proposition}
\label{prop:W_depends_only_on_S2}
$W({\bf g})$ for $\mathbf{g} \in {\rm Diff}_{+*}(\ol{B^3})_0$ depends only on the boundary value ${\bf g}|_{S^2}$.
\end{proposition}

This immediately implies the following.
\begin{corollary}
$\omega_0(h)$ \eqref{eq:omega_definition} does not depend on the choice of extension of $B_h$ to $\til{B}_h$ \eqref{eq:B_h_to_B3}. \qed
\end{corollary}

\begin{remark}
\label{rem:winding_number}
The three-dimensional integral $W({\bf g})$ for ${\bf g}|_{S^2} \in {\rm Diff}_+(S^2)$ resembles what is usually called the {\em Wess-Zumino term} in the Wess-Zumino-(Novikov)-Witten model of $2$-dimensional conformal field theory associated to a compact simple simply-connected Lie group $G$, where our $\theta_{B^3}({\bf g})$ is replaced by the pullback of the Maurer-Cartan $1$-form ${\rm g}^{-1} d{\rm g}$ on the Lie group $G$ along a smooth map $B^3 \to G$ which replaces our $J_{B^3} ({\bf g})$, and where ${\rm tr}_{{\rm mat}_3}$ is replaced by the suitably normalized Killing form of the Lie algebra of $G$.

\vs

In fact, this conformal field theory is a $2$-dimensional theory, so the initial input is a smooth map from a $2$-dimensional manifold, in our case $S^2$, to $G$, not something like $B^3 \to G$. So the usual Wess-Zumino story also involves the extension of a map $S^2\to G$ to a map $B^3 \to G$, and the Wess-Zumino term is well-defined only up to $2\pi \mathbb{Z}$. That is, if we choose a different extension to $B^3\to G$, then the Wess-Zumino term might change by an integer multiple of $2\pi$. For the two different choices of extensions, the integer resulting from the difference of Wess-Zumino terms is sometimes called the {\em winding number}. For us, the winding number is always zero, as we do not allow any maps $B^3 \to {\rm GL}_+(3,\mathbb{R}) = \{\mathrm{g} \in \mathrm{GL}(3,\mathbb{R}) \, | \, \det \mathrm{g} >0 \}$ other than the ones realized as the Jacobian of diffeomorphisms of $B^3$ isotopic to identity. Perhaps, if for example we replace ${\rm Diff}_+(S^2)$ by ${\rm Map}^\infty (S^2,S^2)$ (smooth maps from $S^2$ to itself) in our construction, we may get some winding numbers.
\end{remark}

In order to prove Prop.\ref{prop:W_depends_only_on_S2}, we shall show the following two propositions.

\begin{proposition}
\label{prop:W_coboundary}
The coboundary of $W$ can be computed to be
\begin{align}
\label{eq:W_coboundary}
W({\bf g}\circ {\bf h}) - W({\bf g}) - W({\bf h}) = 3 \int_{S^2} {\rm tr}_{\mathrm{mat}_2} (\theta_{S^2}({\bf g}|_{S^2}) \wedge \theta_{S^2} ({\bf h}^{-1}|_{S^2})),
\end{align}
for all ${\bf g},{\bf h}\in {\rm Diff}_+(\ol{B^3})$.
\end{proposition}

\begin{proposition}
\label{prop:W_zero_on_subgroup}
One has
\begin{align}
\label{eq:W_on_subgroup}
W({\bf h}) = 0, \qquad \forall \, {\bf h}\in ( \mathrm{Diff}_{+*}(\ol{B^3})_0)_{\rm bt} \stackrel{\eqref{eq:bt}}{=} \{ {\bf h}\in \mathrm{Diff}_{+*}(\ol{B^3})_0 \, : \, {\bf h}|_{S^2} = \mathrm{id}_{S^2} \}.
\end{align}
\end{proposition}

Once we have the above two Propositions, it is easy to prove Prop.\ref{prop:W_depends_only_on_S2}:

\begin{proof}[Proof of Prop.\ref{prop:W_depends_only_on_S2}, using Propositions \ref{prop:W_coboundary} and \ref{prop:W_zero_on_subgroup}]
Suppose that ${\bf g}_1, {\bf g}_2 \in {\rm Diff}_{+*}(\ol{B^3})_0$ satisfy ${\bf g}_1|_{S^2} = {\bf g}_2|_{S^2}$. Let ${\bf h} := {\bf g}_1^{-1} \circ {\bf g}_2 \in {\rm Diff}_{+*}(\ol{B^3})_0$. Then $\mathbf{h}|_{S^2} = (\mathbf{g}_1|_{S^2})^{-1} \circ (\mathbf{g}_2|_{S^2}) = \mathrm{id}_{S^2}$, so
$$
{\bf g}_2 = {\bf g}_1 \circ {\bf h}, \qquad \mbox{while} \quad {\bf h} \in (\mathrm{Diff}_{+*}(\ol{B^3})_0)_{\rm bt}.
$$
Put ${\bf g}_1$ into the place of ${\bf g}$ in \eqref{eq:W_coboundary} of Prop.\ref{prop:W_coboundary}, which is possible because ${\rm Diff}_{+*}(\ol{B^3})_0  \subset {\rm Diff}_+(\ol{B^3})$; since $W({\bf h})=0$ from Prop.\ref{prop:W_zero_on_subgroup}, we get
\begin{align*}
W({\bf g}_2) - W({\bf g}_1) = 3 \int_{S^2} {\rm tr}_{\mathrm{mat}_2}(\theta_{S^2}({\bf g}_1) \wedge \theta_{S^2}({\bf h}^{-1})).
\end{align*}
Since ${\bf h}|_{S^2} = \mathrm{id}_{S^2}$, we get $\theta_{S^2}({\bf h}^{-1}) \equiv 0$ on $S^2$, hence $W({\bf g}_2) = W({\bf g}_1)$.
\end{proof}

\vs

Using Prop.\ref{prop:W_coboundary} we can now compute the coboundary of $\omega_0$ and confirm that \eqref{eq:gamma_equals_gamma_0} holds:

\begin{lemma}[coboundary of $\omega_0$]
\label{lem:omega_coboundary_restriction}
One has
\begin{align}
\label{eq:omega_coboundary_restriction}
\omega_0(g\circ h) - \omega_0(g) - \omega_0(h) = \gamma(g,h), \qquad 
\forall \, g,\, h\in \mathcal{H}.
\end{align}
\end{lemma}

\begin{proof}
First, note that the restriction map ${\rm Diff}_+(\ol{B^3})_0 \to {\rm Diff}_+(S^2)_0$, ${\bf g} \mapsto {\bf g}|_{S^2}$, is a group homomorphism. For $g, h \in \mathcal{H}$, this restriction map sends $\til{B}_g$, $\til{B}_h$, $\til{B}_{g\circ h}\in$ ${\rm Diff}_{+*}(\ol{B^3})_0$ to $B_g$, $B_h$, $B_{g\circ h} \in$ ${\rm Diff}_{\rm ci}(S^2)_0 \subset {\rm Diff}_+(S^2)_0$ respectively, while we have $B_{g\circ h} = B_g \circ B_h$ from Lem.\ref{lem:composition_of_B_for_subgroup}. Therefore
$$
(\til{B}_{g\circ h})|_{S^2} = (\til{B}_g \circ \til{B}_h)|_{S^2}, \qquad \forall \, g,\, h \in \mathcal{H}.
$$
Thus, from Prop.\ref{prop:W_depends_only_on_S2} we have
\begin{align}
\label{eq:omega_coboundary_proof1}
W(\til{B}_{g\circ h}) = W(\til{B}_g \circ \til{B}_h), \qquad \forall \, g, \, h \in \mathcal{H},
\end{align}
and hence
\begin{align*}
\omega_0(g\circ h) - \omega_0(g) - \omega_0(h) & \stackrel{\eqref{eq:omega_definition}}{=} c_0\left( W(\til{B}_{g\circ h}) - W(\til{B}_g) - W(\til{B}_h) \right) \\
& \stackrel{\eqref{eq:omega_coboundary_proof1}}{=} c_0 \left( W(\til{B}_g \circ \til{B}_h) - W(\til{B}_g) - W(\til{B}_h) \right) \\
& \stackrel{{\rm Prop}.\ref{prop:W_coboundary}}{=} \, 3c_0 \int_{S^2} {\rm tr}_{\mathrm{mat}_2} (\theta_{S^2} (B_g) \wedge \theta_{S^2} (B_h^{-1})), \qquad \forall \, g, \, h \in \mathcal{H},
\end{align*}
where Prop.\ref{prop:W_coboundary} is applicable because ${\rm Diff}_{+*}(\ol{B^3})_0 \subset {\rm Diff}_+(\ol{B^3})$. From the definition \eqref{eq:B_on_subgroup}, for $g,h\in \mathcal{H}$ we have $\theta_{S^2}(B_h^{-1}) = \theta_{S^2}(\mathrm{id}_{S^2}) = 0$ and likewise $\theta_{S^2}(B_g)=0$ on $(D^2)^\star$, while $\theta_{S^2}(B_h^{-1}) = \theta_{S^2}(h^{-1}) = \theta_{D^2}(h^{-1})$ and likewise $\theta_{S^2}(B_g) = \theta_{S^2}(g) = \theta_{D^2}(g)$ on $D^2$. The result follows from the definition \eqref{eq:gamma_definition} of $\gamma(g,h)$.
\end{proof}

This allows us to apply the construction of \S\ref{sec:general_construction_of_a_central_extension_group}.

\vs

Propositions \ref{prop:W_coboundary} and \ref{prop:W_zero_on_subgroup} are proved in \S\ref{subsec:coboundary_of_W} and \S\ref{subsec:W_zero_on_sbugroup}, respectively.

\subsection{Proof of Prop.\ref{prop:W_coboundary}: computation of the coboundary of the Wess-Zumino term $W$}
\label{subsec:coboundary_of_W}

We first prove the following result about $\theta_M$, which is shown by standard computations about $J_M$, $\theta_M$, and ${\rm tr}_{\mathrm{mat}_n}$:
\begin{lemma}
\label{lem:theta_cubed_coboundary}
Let $M$ be an $n$-dimensional smooth manifold with boundary embedded in $\mathbb{R}^n$, as in Def.\ref{def:G_M_theta_M}. 
Then
\begin{align}
\label{eq:lem_theta_cubed_coboundary}
{\rm tr}_{{\rm mat}_n} \left( (\theta_M ({\bf g}\circ {\bf h}))^{\wedge 3} - (\theta_M ({\bf g}))^{\wedge 3} \circ {\bf h} - (\theta_M ({\bf h}))^{\wedge 3} \right) = 3  \, d\left({\rm tr}_{{\rm mat}_n} (\theta_M ({\bf g}) \wedge \theta_M ({\bf h}^{-1}))\right) \circ {\bf h}
\end{align}
holds as $3$-forms on $M$, for all ${\bf g},{\bf h} \in {\rm Diff}_+(M)$, where
$\theta_M$ is as in Def.\ref{def:G_M_theta_M}, ${\rm tr}_{{\rm mat}_n}$ is the usual trace of $n$ by $n$ matrices. 
\end{lemma}

{\it Proof of Lem.\ref{lem:theta_cubed_coboundary}.} For convenience, denote $\theta = \theta_M$, $J = J_M$ and ${\rm tr} = {\rm tr}_{{\rm mat}_n}$. Let ${\bf g},{\bf h}\in {\rm Diff}_+(M)$. We first compute ${\rm tr} (\theta({\bf g}\circ {\bf h})^{\wedge 3})$:
\begin{align}
\nonumber
{\rm tr} (\theta({\bf g}\circ {\bf h})^{\wedge 3})
& \, \stackrel{{\rm Lem}.\ref{lem:theta_composition}}{=} \, {\rm tr} \left( \left(J({\bf h})^{-1} \, [\theta({\bf g}) \circ {\bf h}] \, J({\bf h}) + \theta({\bf h}) \right)^{\wedge 3} \right) \\
\nonumber
& = {\rm tr} \left( \cancel{J({\bf h})^{-1}}^{\cone} (\theta({\bf g})\circ {\bf h})^{\wedge 3} \, \cancel{J({\bf h})}^{\cone} + \theta({\bf h})^{\wedge 3}    \right) \\
\label{eq:pf_eq1}
& \quad + 3 \, {\rm tr} \left(
J({\bf h})^{-1} [\theta({\bf g})\circ {\bf h}] \, \wedge\, [\theta({\bf g})\circ {\bf h}] \, \wedge\, dJ({\bf h}) \right) \\
\label{eq:pf_eq2}
& \quad + 3 \, {\rm tr}\left(
J({\bf h})^{-1} [\theta({\bf g})\circ {\bf h}] \,\wedge\, dJ({\bf h}) \,\wedge\, J({\bf h})^{-1} \, dJ({\bf h}) \right),
\end{align}
which can easily be checked using the definition $\theta({\bf h}) = J({\bf h})^{-1} \, dJ({\bf h})$ and the skew-symmetry of the wedge product, where the cancellation labeled by $\cone$ is due to the invariance of the trace under conjugation. We continue:
\begin{align*}
{\rm tr} ( \theta({\bf g} \circ {\bf h})^{\wedge 3} )
& = {\rm tr}\left(  (\theta({\bf g}) \circ {\bf h})^{\wedge 3} + (\theta({\bf h}))^{\wedge 3} \right) \\
& \quad + 3 \, {\rm tr} \left(
[J({\bf h})^{-1}\circ {\bf h}^{-1}] \, \theta({\bf g}) \,\wedge\, \theta({\bf g}) \,\wedge\, [dJ({\bf h})\circ {\bf h}^{-1}] \right) \circ {\bf h}\\
& \quad + 3 \, {\rm tr}\left(
[J({\bf h})^{-1}\circ {\bf h}^{-1}] \,  \theta({\bf g}) \,\wedge\, [dJ({\bf h})\circ {\bf h}^{-1}] \,\wedge\, [J({\bf h})^{-1} \circ {\bf h}^{-1}] [dJ({\bf h})\circ {\bf h}^{-1}] \right) \circ {\bf h}.
\end{align*}
By taking the inverse of both sides of \eqref{eq:lem_compose_h_inverse} of Lem.\ref{lem:compose_h_inverse} we have $J({\bf h})^{-1} \circ {\bf h}^{-1} = J({\bf h}^{-1})$. Using this, together with \eqref{eq:d_J_compose_inverse} of the same Lemma, we can now write the above as
\begin{align*}
& {\rm tr} ( \theta({\bf g} \circ {\bf h})^{\wedge 3} ) \\
& = {\rm tr}\left(  (\theta({\bf g}) \circ {\bf h})^{\wedge 3} +  (\theta({\bf h}))^{\wedge 3} \right) \\
& \quad - 3 \, {\rm tr} \left(
\cancel{J({\bf h}^{-1})}^{\ctwo} \, \theta({\bf g}) \, \wedge \, \theta({\bf g}) \, \wedge \, J({\bf h}^{-1})^{-1} \, dJ({\bf h}^{-1}) \, \cancel{J({\bf h}^{-1})^{-1}}^{\ctwo} \right) \circ {\bf h} \\
& \quad + 3 \, {\rm tr}\left(
\cancel{J({\bf h}^{-1})}^{\cthree} \,  \theta({\bf g}) \,\wedge\, J({\bf h}^{-1})^{-1} \, dJ({\bf h}^{-1}) \, J({\bf h}^{-1})^{-1} \,  \wedge\, dJ({\bf h}^{-1}) \, \cancel{J({\bf h}^{-1})^{-1}}^{\cthree} \right) \circ {\bf h} \\
& = {\rm tr}(\theta({\bf g})^{\wedge 3}) \circ {\bf h} + {\rm tr} (\theta({\bf h})^{\wedge 3} )
- 3 \, {\rm tr} \left(
\theta({\bf g}) \,\wedge\, \theta({\bf g}) \,\wedge\, \theta({\bf h}^{-1}) \, \right) \circ {\bf h}
+ 3 \, {\rm tr}\left(
\theta({\bf g}) \,\wedge\, \theta({\bf h}^{-1}) \,\wedge\, \theta({\bf h}^{-1})  \right) \circ {\bf h},
\end{align*}
where the cancellations labeled by $\ctwo$ and $\cthree$ are by the invariance of the trace under conjugation.

\vs

On the other hand, note that
\begin{align*}
& d(\theta({\bf g}) \,\wedge\, \theta({\bf h}^{-1})) = d\left( J({\bf g})^{-1} \, dJ({\bf g}) \,\wedge\, J({\bf h}^{-1})^{-1} \, dJ({\bf h}^{-1}) \right) \\
& \hspace{-2mm} = d(J({\bf g})^{-1}) \,\wedge\, dJ({\bf g}) \,\wedge\, J({\bf h}^{-1})^{-1} \, dJ({\bf h}^{-1})
- J({\bf g})^{-1}\,  dJ({\bf g}) \,\wedge\, d(J({\bf h}^{-1})^{-1}) \,\wedge\, dJ({\bf h}^{-1}) \\
& \hspace{-3mm} \stackrel{{\rm Lem.\ref{lem:d_inverse}}}{=}
\left( - J({\bf g})^{-1} \, dJ({\bf g}) \, J({\bf g})^{-1} \right) \,\wedge\, dJ({\bf g}) \,\wedge\, J({\bf h}^{-1})^{-1} \, dJ({\bf h}^{-1}) \\
& \qquad - J({\bf g})^{-1} \, dJ({\bf g}) \,\wedge\, \left( - J({\bf h}^{-1})^{-1} \, dJ({\bf h}^{-1}) \, J({\bf h}^{-1})^{-1} \right) \, \wedge\, dJ({\bf h}^{-1}) \\
& = - \theta({\bf g}) \,\wedge\, \theta({\bf g}) \,\wedge\, \theta({\bf h}^{-1}) + \theta({\bf g}) \,\wedge\, \theta({\bf h}^{-1}) \,\wedge\, \theta({\bf h}^{-1}).
\end{align*}
The trace of this, together with the above computation about ${\rm tr}(\theta({\bf g}\circ {\bf h})^{\wedge 3})$, yields the result. \qed

\vs

Prop.\ref{prop:W_coboundary} can now be shown by taking the integral of the result of Lem.\ref{lem:theta_cubed_coboundary}, and rewriting ${\rm tr}_{{\rm mat}_3}$ in the right-hand side of \eqref{eq:lem_theta_cubed_coboundary} in terms of ${\rm tr}_{{\rm mat}_2}$:

\begin{proof}[Proof of Prop.\ref{prop:W_coboundary}]
Take Lem.\ref{lem:theta_cubed_coboundary} for $M= \ol{B^3}$. Then the left-hand-side in \eqref{eq:lem_theta_cubed_coboundary} of Lem.\ref{lem:theta_cubed_coboundary} is a $3$-form on $B^3$, so let's take the integral of this over $B^3$: in view of \eqref{eq:W_definition}, this integral is
\begin{align*}
& W({\bf g}\circ {\bf h}) - \int_{B^3} {\rm tr}_{{\rm mat}_3}(\theta_{B^3}({\bf g})^{\wedge 3}) \circ {\bf h} - W({\bf h}) \\
& = W({\bf g}\circ {\bf h}) - \int_{{\bf h}(B^3)} {\rm tr}_{{\rm mat}_3}(\theta_{B^3}({\bf g})^{\wedge 3}) - W({\bf h}) = W({\bf g}\circ {\bf h}) - W({\bf g}) - W({\bf h}) \quad (\because {\bf h}(B^3)=B^3),
\end{align*}
which is the left-hand-side of \eqref{eq:W_coboundary}. Now, for any ${\bf g},{\bf h}\in {\rm Diff}_+(\ol{B^3})$, we shall confirm that the $2$-form 
\begin{align}
\label{eq:D3_2-form}
{\rm tr}_{{\rm mat}_3} (\theta_{B^3}({\bf g}) \wedge \theta_{B^3}({\bf h}^{-1}))
\end{align}
restricted (i.e. pulled back) to $S^2 = \partial B^3$ coincides with the $2$-form 
\begin{align}
\label{eq:S2_2-form}
{\rm tr}_{{\rm mat}_2} (\theta_{S^2}({\bf g}|_{S^2}) \wedge \theta_{S^2}({\bf h}^{-1}|_{S^2}))
\end{align}
on $S^2$, up to an exact $2$-form on $S^2$. Then, by Stokes' Theorem one can show that the integral of the right-hand-side of \eqref{eq:lem_theta_cubed_coboundary} on $B^3$ equals the right-hand-side of  \eqref{eq:W_coboundary}, finishing the proof of Prop.\ref{prop:W_coboundary}.

\vs

Of course, when proving that \eqref{eq:D3_2-form} equals  \eqref{eq:S2_2-form} as $2$-forms on $S^2$ (up to an exact $2$-form), we can replace ${\bf h}^{-1}$ by ${\bf h}$ for convenience. First, recall from \S\ref{subsec:extension_from_S2_to_B3} that we use the coordinate system $\smallvecthree{\rho}{x}{y}$ for $\ol{B^3}$ where $\smallvectwo{x}{y}$ parametrizes $S^2$ (or $S^2\setminus\{\infty\}$) and $\rho \in [0,1]$ is the radial coordinate, i.e. when $\rho=1$ we get $S^2 = \partial B^3$ parametrized by $\smallvectwo{x}{y}$. From now on, we fix ${\bf g},{\bf h}\in {\rm Diff}_+(\ol{B^3})$. We first write
\begin{align}
\label{eq:first_trace_as_sum_of_forms}
{\rm tr}_{{\rm mat}_3} (\theta_{B^3}({\bf g}) \wedge \theta_{B^3}({\bf h}))
= \square^{{\bf g},{\bf h}}_{\rho,x} \, d\rho\wedge dx + \square^{{\bf g},{\bf h}}_{\rho,y} \, d\rho \wedge dy + \square^{{\bf g},{\bf h}}_{x,y} \, dx \wedge dy
\end{align}
for some scalar-valued functions $\square^{{\bf g},{\bf h}}_{\rho,x}, \square^{{\bf g},{\bf h}}_{\rho,y}, \square^{{\bf g},{\bf h}}_{x,y}$ on $B^3$. The pullback of this $2$-form along the embedding $S^2 \hookrightarrow \ol{B^3}$ is then the following $2$-form on $S^2 = \partial B^3$:
\begin{align*}
\left. \square^{{\bf g},{\bf h}}_{x,y}  \right|_{\rho=1} \, dx\wedge dy.
\end{align*}
Now, let us write
\begin{align}
\label{eq:theta_as_sum}
\theta_{B^3}({\bf g}) = T^{\bf g}_\rho \, d\rho + T^{\bf g}_x\, dx + T^{\bf g}_y\, dy,
\end{align}
for some $\mathfrak{gl}(3,\mathbb{R})$-valued functions $T^{\bf g}_\rho, T^{\bf g}_x, T^{\bf g}_y$ on $B^3$, and likewise for ${\bf h}$ with the functions $T^{\bf h}_\rho$, $T^{\bf h}_x$, $T^{\bf h}_y$. Then, by plugging \eqref{eq:theta_as_sum} into \eqref{eq:first_trace_as_sum_of_forms} and using the skew-symmetry of the wedge product, we get
\begin{align}
\label{eq:square_as_commutator}
\square^{{\bf g},{\bf h}}_{x,y} = {\rm tr}_{{\rm mat}_3} (T^{\bf g}_x \, T^{\bf h}_y - T^{\bf h}_x \, T^{\bf g}_y).
\end{align}
In order to compute this, we first denote by ${\bf g}^{(1)},{\bf g}^{(2)},{\bf g}^{(3)}$ the component functions of ${\bf g}$, i.e.
\begin{align*}
{\bf g}: \vecthree{\rho}{x}{y} \mapsto \vecthree{ {\bf g}^{(1)}\smallvecthree{\rho}{x}{y} }{ {\bf g}^{(2)}\smallvecthree{\rho}{x}{y} }{ {\bf g}^{(3)}\smallvecthree{\rho}{x}{y} },
\end{align*}
with respect to the coordinate system $\smallvecthree{\rho}{x}{y}$ of $B^3$, and likewise for ${\bf h}$. So the Jacobian matrix for ${\bf g} \in {\rm Diff}_+(\ol{B^3})$ is
\begin{align*}
J_{B^3}({\bf g})\smallvecthree{\rho}{x}{y}  = \left({\renewcommand{\arraystretch}{1.2} \begin{array}{ccc}
(\partial_\rho \, {\bf g}^{(1)}) \smallvecthree{\rho}{x}{y} & (\partial_x {\bf g}^{(1)}) \smallvecthree{\rho}{x}{y} & (\partial_y {\bf g}^{(1)}) \smallvecthree{\rho}{x}{y} \\
(\partial_\rho \, {\bf g}^{(2)}) \smallvecthree{\rho}{x}{y} & (\partial_x {\bf g}^{(2)}) \smallvecthree{\rho}{x}{y} & (\partial_y {\bf g}^{(2)}) \smallvecthree{\rho}{x}{y} \\
(\partial_\rho \, {\bf g}^{(3)}) \smallvecthree{\rho}{x}{y} & (\partial_x {\bf g}^{(3)}) \smallvecthree{\rho}{x}{y} & (\partial_y {\bf g}^{(3)}) \smallvecthree{\rho}{x}{y}
\end{array}} \right)
\end{align*}
and hence
\begin{align*}
\theta_{B^3}({\bf g}) = J_{B^3}({\bf g})^{-1} \, dJ_{B^3}({\bf g}) = 
J_{B^3}({\bf g})^{-1} \left( \, \partial_\rho (J_{B^3}({\bf g})) \, d\rho + \partial_x (J_{B^3}({\bf g})) \, dx + \partial_y (J_{B^3}({\bf g})) \, dy\, \right).
\end{align*}
Matching with \eqref{eq:theta_as_sum} gives
\begin{align}
\label{eq:Tgr_Tgtau}
T^{\bf g}_x = J_{B^3}({\bf g})^{-1} \, \partial_x (J_{B^3}({\bf g})), \qquad
T^{\bf g}_y = J_{B^3}({\bf g})^{-1} \, \partial_y (J_{B^3}({\bf g})),
\end{align}
and likewise for ${\bf h}$. Eventually we want to compute the restriction of \eqref{eq:square_as_commutator} to $\rho=1$, therefore we'd like to know the restriction of \eqref{eq:Tgr_Tgtau} to $\rho=1$. Note that on $S^2$, i.e. if we fix $\rho=1$, then ${\bf g}^{(1)}$ is a constant function, so $(\partial_x {\bf g}^{(1)})|_{\rho=1} = \partial_x( {\bf g}^{(1)}|_{\rho=1})= 0 = (\partial_y {\bf g}^{(1)})|_{\rho=1}$. Also we note that the lower-right $2$ by $2$ part of $J_{B^3}({\bf g})$ restricted to $\rho=1$ is exactly $J_{S^2}({\bf g}|_{S^2})$, because
$$
{\bf g}|_{S^2} : \vectwo{x}{y} \to \left({\renewcommand{\arraystretch}{1.2} \begin{array}{c} {\bf g}^{(2)}\smallvecthree{1}{x}{y} \\ {\bf g}^{(3)}\smallvecthree{1}{x}{y} \end{array} }\right).
$$
Hence we can write the restriction of $J_{B^3}({\bf g})$ to $\rho=1$ as
\begin{align}
\label{eq:J_restriction}
(J_{B^3}({\bf g}))|_{\rho=1} \smallvectwo{x}{y} = \left( \begin{array}{c|c}
(\partial_\rho \, {\bf g}^{(1)}) \smallvecthree{1}{x}{y} & 0 \qquad 0 \\\hline
\begin{array}{c} * \\ * \end{array} & J_{S^2}({\bf g}|_{S^2}) \smallvectwo{x}{y}
\end{array} \right)
\end{align}
By a similar reasoning it's easy to see
\begin{align*}
\partial_x (J_{B^3}({\bf g})) |_{\rho=1} \smallvectwo{x}{y} & = \left( \begin{array}{c|c}
(\partial_{x \rho} \, {\bf g}^{(1)}) \smallvecthree{1}{x}{y} & 0 \qquad 0 \\\hline
\begin{array}{c} * \\ * \end{array} & \partial_x (J_{S^2}({\bf g}|_{S^2})) \smallvectwo{x}{y}
\end{array} \right), \\
\partial_y (J_{D^3}({\bf g}))|_{\rho=1} \smallvectwo{x}{y} & = \left( \begin{array}{c|c}
(\partial_{y\rho}\, {\bf g}^{(1)}) \smallvecthree{1}{x}{y} & 0 \qquad 0 \\\hline
\begin{array}{c} * \\ * \end{array} & \partial_y (J_{S^2}({\bf g}|_{S^2})) \smallvectwo{x}{y}
\end{array} \right),
\end{align*}
where the double subscript for $\partial$ means the second order partial derivatives, e.g. $\partial_{x\rho} = \partial_x \partial_\rho$. It is a simple exercise to show that the inverse of $(J_{B^3}({\bf g}))|_{\rho=1}$ is given by the similar block lower-triangular form as $(J_{B^3}({\bf g}))|_{\rho=1}$:
\begin{align*}
(J_{B^3}({\bf g}))^{-1}|_{\rho=1} \smallvectwo{x}{y} = \left( \begin{array}{c|c}
\left( (\partial_\rho \, {\bf g}^{(1)}) \smallvecthree{1}{x}{y} \right)^{-1} & 0 \qquad 0 \\\hline
\begin{array}{c} * \\ * \end{array} & \left( J_{S^2}({\bf g}|_{S^2})  \smallvectwo{x}{y} \right)^{-1}
\end{array} \right).
\end{align*}
Now, plugging these into \eqref{eq:Tgr_Tgtau} and \eqref{eq:square_as_commutator}, by a little bit of calculation we get
\begin{align*}
\square^{{\bf g},{\bf h}}_{x,y} & = \frac{\left. ( (\partial_{x\rho} \, {\bf g}^{(1)}) (\partial_{y\rho} \, {\bf h}^{(1)}) - (\partial_{x\rho} \, {\bf h}^{(1)}) ( \partial_{y\rho} \, {\bf g}^{(1)}) ) \right|_{\rho=1} }{ \left.( (\partial_\rho \, {\bf g}^{(1)}) (\partial_\rho \, {\bf h}^{(1)}) ) \right|_{\rho=1} } \\
& \quad
+ \underbrace{ {\rm tr}_{{\rm mat}_2} \left( \begin{array}{l}
(J_{S^2}({\bf g}|_{S^2}))^{-1} \, \partial_x (J_{S^2}({\bf g}|_{S^2})) \, (J_{S^2}({\bf h}|_{S^2}))^{-1} \, \partial_y (J_{S^2}({\bf h}|_{S^2})) \\
- (J_{S^2}({\bf h}|_{S^2}))^{-1} \, \partial_x (J_{S^2}({\bf h}|_{S^2})) \, (J_{S^2}({\bf g}|_{S^2}))^{-1} \, \partial_y (J_{S^2}({\bf g}|_{S^2})) 
\end{array} \right) }_{=: S\smallvectwo{x}{y}},
\end{align*}
and by a similar argument as \eqref{eq:first_trace_as_sum_of_forms} applied to $S^2$, we can prove that the underbraced second term, which is a scalar-valued function on $S^2$, denoted by $S\smallvectwo{x}{y}$, gives what we want to get:
$$
S\smallvectwo{x}{y} \, dx\wedge dy = {\rm tr}_{{\rm mat}_2} (\theta_{S^2}({\bf g}|_{S^2}) \wedge \theta_{S^2}({\bf h}|_{S^2})),
$$
as $2$-forms on $S^2$. Hence, what remains to be proved is that the first term of $\square^{{\bf g},{\bf h}}_{x,y}$ corresponds to an exact $2$-form on $S^2$, i.e. that
\begin{align}
\label{eq:to_be_exact}
\frac{\left. \left(\, (\partial_{x\rho} \, {\bf g}^{(1)}) (\partial_{y\rho} \, {\bf h}^{(1)}) - (\partial_{x\rho} \, {\bf h}^{(1)}) (\partial_{y\rho} \, {\bf g}^{(1)}) ) \, \right) \right|_{\rho=1} }{ \left.\left( (\partial_\rho \, {\bf g}^{(1)}) (\partial_\rho \, {\bf h}^{(1)}) \right) \right|_{\rho=1} } \,\, dx \wedge dy
\end{align}
is an exact $2$-form on $S^2$. In \eqref{eq:to_be_exact}, note for example that $(\partial_{x\rho} \, {\bf g}^{(1)})|_{\rho=1} = \partial_x ((\partial_\rho \, {\bf g}^{(1)})|_{\rho=1})$, so we can express \eqref{eq:to_be_exact} using the two functions $(\partial_\rho \, {\bf g}^{(1)})|_{\rho=1}$ and $(\partial_\rho \, {\bf h}^{(1)})|_{\rho=1}$ and their partial derivatives with respect to $x$ and $y$.

\vs

We can show that the function $(\partial_\rho \, {\bf g}^{(1)})|_{\rho=1}$ on $S^2$ is always positive. Note that the determinant of $(J_{B^3}({\bf g}))|_{\rho=1}$ is positive; moreover, by \eqref{eq:J_restriction} it equals $(\partial_\rho \, {\bf g}^{(1)})|_{\rho=1}$ times the determinant of $J_{S^2}({\bf g}|_{S^2})$, which is positive because ${\bf g}|_{S^2}$ is an orientation-preserving self-diffeomorphism of $S^2$. Thus $(\partial_\rho \, {\bf g}^{(1)})|_{\rho=1}$ must be positive.

\vs

So we can define the function $\phi^{\bf g}$ on $S^2$ as
$$
\phi^{\bf g} := \log \left( (\partial_\rho \,{\bf g}^{(1)})|_{\rho=1}\right),
$$
and likewise for ${\bf h}$, i.e. $\phi^{\bf h} := \log (\partial_\rho \, {\bf h}^{(1)}_\rho)|_{\rho=1}$. Then the $2$-form \eqref{eq:to_be_exact} becomes
\begin{align*}
& \frac{ \partial_x ((\partial_\rho \, {\bf g}^{(1)}_\rho)|_{\rho=1}) \cdot \partial_y ((\partial_\rho \, {\bf h}^{(1)})|_{\rho=1}) - \partial_x ((\partial_\rho \, {\bf h}^{(1)})|_{\rho=1}) \cdot \partial_y((\partial_\rho \, {\bf g}^{(1)})|_{\rho=1}) }{ (\partial_\rho \, {\bf g}^{(1)})|_{\rho=1}  \cdot (\partial_\rho \, {\bf h}^{(1)})|_{\rho=1} } \,\, dx\wedge dy \\
& = \left(\partial_x (\phi^{\bf g}) \cdot \partial_y (\phi^{\bf h}) - \partial_x (\phi^{\bf h}) \cdot \partial_y (\phi^{\bf g}) \right) dx \wedge dy,
\end{align*}
which is indeed an exact $2$-form on $S^2$, because we have the following, as $2$-forms on $S^2$:
\begin{align*}
d\left( (\phi^{\bf g}) \, d (\phi^{\bf h}) \right) 
= d(\phi^{\bf g}) \wedge d(\phi^{\bf h})
& = (\partial_x(\phi^{\bf g}) \, dx + \partial_y (\phi^{\bf g})\, dy) \wedge (\partial_x(\phi^{\bf h}) \, dx + \partial_y (\phi^{\bf h})\, dy)  \\
& = \left(\partial_x (\phi^{\bf g}) \cdot \partial_y (\phi^{\bf h}) - \partial_x (\phi^{\bf h}) \cdot \partial_y (\phi^{\bf g}) \right) dx \wedge dy.
\end{align*}

\end{proof}

\subsection{Proof of Prop.\ref{prop:W_zero_on_subgroup}: $W$ is zero on the boundary trivial $3$d diffeomorphisms}
\label{subsec:W_zero_on_sbugroup}

\begin{lemma}
\label{lem:W_is_zero_of_bt_p0}
$W$ defined in \eqref{eq:W_definition} is zero on $({\rm Diff}_+(\ol{B^3})_{\rm bt})_0$.
\end{lemma}
\begin{proof}
Consider any smooth isotopy $t\mapsto {\bf k}_t$, $t\in (-\epsilon,\epsilon)$, in ${\rm Diff}_+(\ol{B^3})$ passing through $\mathrm{id}_{B^3}$ at $t=0$. So there is a smooth function $v : \ol{B^3} \to \mathbb{R}^3$ such that $\mathbf{k}_t (\mathbf{x}) = \mathbf{x} + t \, v(\mathbf{x}) + o(t)$, $\forall \mathbf{x}\in B^3$. Then $J_{B^3}(\mathbf{k}_t) = J_{B^3}(\mathrm{id}_{B^3}) + t \, J_{B^3}(v) + o(t) = I + t J_{B^3}(v) + o(t)$, where $I$ is the $3\times 3$ identity matrix. Thus
\begin{align*}
  (J_{B^3}(\mathbf{k}_t))^{-1} & = I - t \, J_{B^3}(v) + o(t), \\
  dJ_{B^3}(\mathbf{k}_t) & = t\, dJ_{B^3}(v) + o(t), \\
  \theta_{B^3}(\mathbf{k}_t) & = (J_{B^3}(\mathbf{h}_t))^{-1} dJ_{B^3}(\mathbf{h}_t) = t \, dJ_{B^3}(v) + o(t), \\
  \theta_{B^3}(\mathbf{k}_t)^{\wedge 3} & = t^3 (dJ_{B^3}(v))^{\wedge 3} + o(t^3),
\end{align*}
where $dJ_{B^3}(v)$ is a well-defined $\mathfrak{gl}(3,\mathbb{R})$-valued $1$-form on $\ol{B^3}$. Thus $\int_{B^3} \mathrm{tr}_{\mathrm{mat}_3}(\theta_{B^3}(\mathbf{k}_t)^{\wedge 3})$ is $o(t)$, and hence $\frac{d}{dt} W(\mathbf{k}_t)|_{t=0} = 0$.

\vs

Now let $\mathbf{g} \in {\rm Diff}_+(\ol{B^3})_{\rm bt}$, and let $t\mapsto {\bf g}_t$, $t\in (-\epsilon,\epsilon)$, be any smooth isotopy in ${\rm Diff}_+(\ol{B^3})$ passing through $\mathbf{g}$ at $t=0$. For each $t$ define $\mathbf{k}_t := \mathbf{g}^{-1} \circ \mathbf{g}_t$; then $t\mapsto {\bf k}_t$, $t\in (-\epsilon,\epsilon)$, is a smooth isotopy in ${\rm Diff}_+(\ol{B^3})$ passing through $\mathbf{g}^{-1} \circ \mathbf{g} = \mathrm{id}_{\ol{B^3}}$ at $t=0$. So, by what we just proved, we have $\frac{d}{dt} W(\mathbf{k}_t)|_{t=0}=0$. Meanwhile, we have $W(\mathbf{g}_t) = W( \mathbf{g} \circ \mathbf{k}_t) = W(\mathbf{g}) + W(\mathbf{k}_t)$ from \eqref{eq:W_coboundary} of Prop.\ref{prop:W_coboundary}, because the integral in the the right-hand-side of \eqref{eq:W_coboundary} is zero in this case because $\theta_{S^2}(\mathbf{g}|_{S^2}) = \theta_{S^2}(\mathrm{id}_{S^2})=0$, for our $\mathbf{g}$ is in ${\rm Diff}_+(\ol{B^3})_{\rm bt}$.

\vs

So, for any smooth isotopy $t\mapsto \ell_t$, $t\in [0,1]$, in ${\rm Diff}_+(\ol{B^3})_{\rm bt}$, the function $[0,1]\to \mathbb{R}$ defined by $t \mapsto W(\ell_t)$ has zero first derivative, hence is a constant function. This, together with $W(\mathrm{id}_{\ol{B^3}})=0$ which one can directly verify, shows that $W$ is the zero function on $({\rm Diff}_+(\ol{B^3})_{\rm bt})_0$.
\end{proof}

In order to prove Prop.\ref{prop:W_zero_on_subgroup}, consider any ${\bf h} \in ({\rm Diff}_{+*}(\ol{B^3})_0)_{\rm bt}$; we shall eventually prove $W(\mathbf{h})=0$. By definition, there is a smooth isotopy $t \mapsto {\bf h}_t$, $t\in [0,1]$, in ${\rm Diff}_{+*}(\ol{B^3})$, with $\mathbf{h}_0 = \mathrm{id}_{\ol{B^3}}$, $\mathbf{h}_1 = \mathbf{h}$, while $\mathbf{h}|_{S^2} = \mathrm{id}_{S^2}$. For each $t\in [0,1]$ let $h_t := \mathbf{h}_t|_{S^2} \in {\rm Diff}_{\rm ci}(S^2)$. Then $t\mapsto h_t$ is a smooth isotopy in ${\rm Diff}_{\rm ci}(S^2)$ with $h_0=\mathrm{id}_{S^2}=h_1$, so we have $h_t \in {\rm Diff}_{\rm ci}(S^2)_0$ for each $t\in [0,1]$. Namely, for each $t\in [0,1]$ and $u \in [0,1]$, let $h^{(t)}_u := h_{ut}$; then, for each $t\in [0,1]$, $u\mapsto h^{(t)}_u$ is a smooth isotopy in ${\rm Diff}_{\rm ci}(S^2)$ with $h^{(t)}_0 = \mathrm{id}_{S^2}$ and $h^{(t)}_1 = h_t$. Now, for each $t\in [0,1]$ we lift $h_t \in {\rm Diff}_{\rm ci}(S^2)_0$ to $\til{h}_t \in {\rm Diff}_{+*}(\ol{B^3})_0$ as in the proof of Prop.\ref{prop:S2_extension_to_B3}.

\vs

Let us be more precise. Choose and fix any $\epsilon \in (0,1/2)$ and let $\xi := \xi_{\epsilon,\epsilon}$ (Lem.\ref{lem:xi}). For each $t\in [0,1]$ define $\til{h}_t : \ol{B^3} \to \ol{B^3}$ as
\begin{align}
\label{eq:til_h_t}
  \til{h}_t \vectwo{\rho}{\smallvectwo{x}{y}} = \vectwo{\rho}{h_{t\cdot\xi(\rho)}\smallvectwo{x}{y}}, \qquad \forall \rho\in [0,1], \qquad \forall \smallvectwo{x}{y} \in S^2.
\end{align}
Then $t\mapsto \til{h}_t$ is a smooth isotopy in ${\rm Diff}_{+*}(\ol{B^3})$ with $\til{h}_0 = \mathrm{id}_{\ol{B^3}}$; hence it is a smooth isotopy in ${\rm Diff}_{+*}(\ol{B^3})_0$. Meanwhile, for each $t \in [0,1]$ we have $\til{h}_t|_{S^2} = h_t$, so $\ell_t := \mathbf{h}_t \circ \til{h}_t^{-1}$ is an element of ${\rm Diff}_+(\ol{B^3})$ whose boundary value is $\ell_t|_{S^2} = (\mathbf{h}_t|_{S^2}) \circ (\til{h}_t|_{S^2})^{-1} = h_t \circ h_t^{-1} = \mathrm{id}_{S^2}$. So $t\mapsto \ell_t$ is a smooth isotopy in ${\rm Diff}_+(\ol{B^3})_{\rm bt}$ with $\ell_0 = \mathbf{h}_0 \circ \til{h}_0^{-1} = \mathrm{id}_{\ol{B^3}}$; in particular, each $\ell_t$ is in $({\rm Diff}_+(\ol{B^3})_{\rm bt})_0$, so Lem.\ref{lem:W_is_zero_of_bt_p0} tells us $W(\ell_t)=0$. Observe now that $W(\mathbf{h}_t) = W(\ell_t \circ \til{h}_t) \stackrel{\eqref{eq:W_coboundary}}{=} W(\ell_t) + W(\til{h}_t) = W(\til{h}_t)$ for each $t\in [0,1]$; when applying \eqref{eq:W_coboundary} of Prop.\ref{prop:W_coboundary} here, note that the right-hand-side of \eqref{eq:W_coboundary} is zero in this case because $\theta_{S^2}(\ell_t|_{S^2}) = \theta_{S^2}(\mathrm{id}_{S^2})=0$. So $W(\mathbf{h}) = W(\mathbf{h}_1) = W(\til{h}_1)$. Thus, establishing
\begin{align}
\label{eq:to_establish}
W(\til{h}_1)=0
\end{align}
would finish our proof of Prop.\ref{prop:W_zero_on_subgroup}.

\vs

For each $\rho\in [0,1]$, define $B_\rho \in {\rm Diff}_+(S^2)$ as $B_\rho := h_{\xi(\rho)}$; then $\rho\mapsto B_\rho$ is smooth. From \eqref{eq:til_h_t},
\begin{align*}
  J_{B^3}(\til{h}_1) = \mattwo{1}{0}{*_1}{J_{S^2}(B_\rho)}, \qquad\mbox{so}\qquad dJ_{B^3}(\til{h}_1) = \mattwo{0}{0}{*_2}{dJ_{S^2}(B_\rho)},
\end{align*}
where the two $d$'s are exterior derivative for $B^3$. Thus
\begin{align*}
  \theta_{B^3}(\til{h}_1) = J_{B^3}(\til{h}_1)^{-1} \, dJ_{B^3}(\til{h}_1) = \mattwo{0}{0}{*_3}{J_{S^2}(B_\rho)^{-1} \, dJ_{S^2}(B_\rho)},
\end{align*}
and therefore one observes that
\begin{align*}
  W(\til{h}_1) \stackrel{\eqref{eq:W_definition}}{=} \int_{B^3} \mathrm{tr}_{\mathrm{mat}_3} \left( (\theta_{B^3}(\til{h}_1))^{\wedge 3} \right)
= \int_{B^3} \mathrm{tr}_{\mathrm{mat}_2} \left( ( J_{S^2}(B_\rho)^{-1} \, dJ_{S^2}(B_\rho) )^{\wedge 3} \right).
\end{align*}
For each $\rho\in [0,1]$, $J_{S^2}(B_\rho)$ is a $\mathrm{GL}_+(2,\mathbb{R})$-valued smooth function on $S^2$, where
\begin{align*}
  \mathrm{GL}_+(n,\mathbb{R}) := \{ \mathrm{g} \in \mathrm{GL}(n,\mathbb{R}) \, | \, \det \mathrm{g} >0 \} \qquad (\mbox{for any positive integer $n$}),
\end{align*}
for $B_\rho$ is an orientation-preserving self-diffeomorphism of $S^2$. As $\rho\mapsto B_\rho$ is smooth, one can regard the symbol $J_{S^2}(B_\rho)$ as a $\mathrm{GL}_+(2,\mathbb{R})$-valued smooth function on $\ol{B^3}$, i.e. one can write
\begin{align*}
  J_{S^2}(B_\rho) : \ol{B^3} \to \mathrm{GL}_+(2,\mathbb{R}) \qquad (\mbox{of class $C^\infty$}).
\end{align*}
When $\rho=1$ we have $B_1 = h_{\xi(1)} = h_1 = \mathrm{id}_{S^2}$, hence
\begin{align}
  \label{eq:J_S2_of_B1}
  \begin{array}{l}
  \mbox{for each $\rho>1-\epsilon$, the map $J_{S^2}(B_\rho) : S^2 \to \mathrm{GL}_+(2,\mathbb{R})$}\\
  \mbox{is the constant function with value $\mathrm{id}_{2\times 2}$},
  \end{array}
\end{align}
where $\mathrm{id}_{2\times 2}$ stands for the $2\times 2$ identity matrix.

\begin{definition}[$3$-form $\eta$ on linear Lie groups]
\label{def:eta}
Let $n$ be a positive integer. Consider the following real-valued differential $3$-form on the Lie group $\mathrm{GL}(n,\mathbb{R})$:
\begin{align}
\label{eq:eta}
\eta_n := \mathrm{tr}_{{\rm mat}_n} \left( (\mathrm{g}^{-1} d\mathrm{g})^{\wedge 3} \right),
\end{align}
where $\mathrm{g}$ in \eqref{eq:eta} denotes a matrix variable parametrizing $\mathrm{GL}(n,\mathbb{R})$, and $\mathrm{tr}_{\mathrm{mat}_n}$ is the usual trace of $n\times n$ matrices.
\end{definition}

Note that the identity component of $\mathrm{GL}(n,\mathbb{R})$ is $\mathrm{GL}_+(n,\mathbb{R})$, which forms a Lie group; we denote the pullback of $\eta_n$ along the embedding $\mathrm{GL}_+(n,\mathbb{R}) \hookrightarrow \mathrm{GL}(n,\mathbb{R})$ also by $\eta_n$.

\vs

So one can now write
\begin{align}
\label{eq:W_til_h_1_integral}
  W(\til{h}_1) = \int_{B^3} (J_{S^2}(B_\rho))^* \, \eta_2,
\end{align}
where $\eta_2$ is regarded as a differential $2$-form on $\mathrm{GL}_+(2,\mathbb{R})$. In order to take advantage of some results in algebraic topology, we should first establish the following fact, which must be well known; we provide a quick proof for completeness:

\begin{lemma}
\label{lem:closedform}
$\eta_n$ is a closed $3$-form on $\mathrm{GL}(n,\mathbb{R})$, hence also on $\mathrm{GL}_+(n,\mathbb{R})$, i.e. $d\eta_n =0$.
\end{lemma}

\begin{proof}
Let $\mathrm{g} \in \mathrm{GL}(n,\mathbb{R})$ denote the matrix variable parametrizing $\mathrm{GL}(n,\mathbb{R})$; thus $\mathrm{g}$ can be thought of as the identity map $\mathrm{g}\mapsto \mathrm{g}$ from $\mathrm{GL}(n,\mathbb{R})$ to $\mathrm{GL}(n,\mathbb{R})$. From Lem.\ref{lem:d_inverse} we have
\begin{align}
\label{eq:d_g_inverse}
d(\mathrm{g}^{-1}) = - \mathrm{g}^{-1} (d \mathrm{g}) \, \mathrm{g}^{-1}.
\end{align}
Now, using \eqref{eq:d_g_inverse} and $d^2 \mathrm{g} = 0$, we compute
\begin{align*}
& d ( \, \mathrm{g}^{-1} (d\mathrm{g}) \wedge \mathrm{g}^{-1} (d\mathrm{g}) \wedge \mathrm{g}^{-1} (d\mathrm{g})) \\
& = (\ul{-\mathrm{g}^{-1} \, (d\mathrm{g}) \, \mathrm{g}^{-1}})  \wedge d\mathrm{g} \wedge (\mathrm{g}^{-1} d\mathrm{g}) \wedge (\mathrm{g}^{-1} d\mathrm{g})
-  (\mathrm{g}^{-1} d\mathrm{g}) \wedge (\ul{-\mathrm{g}^{-1}  (d\mathrm{g}) \mathrm{g}^{-1}}) \wedge d\mathrm{g} \wedge (\mathrm{g}^{-1} d\mathrm{g}) \\
& \quad + (\mathrm{g}^{-1} d\mathrm{g}) \wedge (\mathrm{g}^{-1}  d\mathrm{g}) \wedge (\ul{- \mathrm{g}^{-1} (d\mathrm{g}) \mathrm{g}^{-1}}) \wedge d\mathrm{g} \\
& = - (\mathrm{g}^{-1} d\mathrm{g})^{\wedge 4},
\end{align*}
where the underlined parts come from \eqref{eq:d_g_inverse}. Thus we have
\begin{align}
\label{eq:closedform_proof1}
d \left( \mathrm{tr}_{{\rm mat}_n} (\mathrm{g}^{-1} d\mathrm{g})^{\wedge 3} \right) = - \mathrm{tr}_{{\rm mat}_n} (\mathrm{g}^{-1} d\mathrm{g})^{\wedge 4}.
\end{align}
We're trying to assert that this $4$-form is $0$. Inspired by the usual proof for this, we proceed as follows.

\vs

Using the skew-symmetry we have
\begin{align}
\label{eq:closedform_proof2}
\begin{array}{rl}
\mathrm{tr}_{{\rm mat}_n} (\mathrm{g}^{-1} d\mathrm{g})^{\wedge 3}
& = \mathrm{tr}_{{\rm mat}_n} (\ul{ \mathrm{g}^{-1} } \, \,\ul{ d\mathrm{g} \wedge \mathrm{g}^{-1} d\mathrm{g} \wedge \mathrm{g}^{-1} d\mathrm{g}}) \\
& = \mathrm{tr}_{{\rm mat}_n} (\ul{d\mathrm{g} \wedge \mathrm{g}^{-1} d\mathrm{g} \wedge \mathrm{g}^{-1} d\mathrm{g}} \,\, \ul{\mathrm{g}^{-1}})
= \mathrm{tr}_{{\rm mat}_n} (d\mathrm{g}\, \mathrm{g}^{-1})^{\wedge 3}.
\end{array}
\end{align}
Now again using \eqref{eq:d_g_inverse} and $d^2 g = 0$, we observe
\begin{align*}
d( (d\mathrm{g} \, \mathrm{g}^{-1})^{\wedge 3}) & = d( \, d\mathrm{g} \, \mathrm{g}^{-1} \wedge d\mathrm{g} \, \mathrm{g}^{-1} \wedge d\mathrm{g} \, \mathrm{g}^{-1} ) \\
& = - d\mathrm{g} \wedge (\ul{-\mathrm{g}^{-1} (d\mathrm{g})  \mathrm{g}^{-1}}) \wedge d\mathrm{g} \, \mathrm{g}^{-1} \wedge d\mathrm{g} \, \mathrm{g}^{-1}
+ d\mathrm{g} \, \mathrm{g}^{-1} \wedge d\mathrm{g} \wedge (\ul{ -\mathrm{g}^{-1} (d\mathrm{g}) \mathrm{g}^{-1}}) \wedge d\mathrm{g} \, \mathrm{g}^{-1} \\
& \quad
- d\mathrm{g} \, \mathrm{g}^{-1} \wedge d\mathrm{g} \, \mathrm{g}^{-1} \wedge d\mathrm{g} \wedge (\ul{ -\mathrm{g}^{-1} (d\mathrm{g}) \mathrm{g}^{-1} }) \\
& = (d\mathrm{g} \, \mathrm{g}^{-1})^{\wedge 4},
\end{align*}
and therefore
\begin{align}
\label{eq:closedform_proof3}
d \left( \mathrm{tr}_{{\rm mat}_n} (d\mathrm{g} \, \mathrm{g}^{-1} )^{\wedge 3} \right) = \mathrm{tr}_{{\rm mat}_n} (d\mathrm{g}\, \mathrm{g}^{-1})^{\wedge 4}.
\end{align}
Again using the skew-symmetry we have
\begin{align}
\label{eq:closedform_proof4}
\mathrm{tr}_{{\rm mat}_n} ( d\mathrm{g} \, \mathrm{g}^{-1} )^{\wedge 4}
= \mathrm{tr}_{{\rm mat}_n} (\mathrm{g}^{-1} d\mathrm{g})^{\wedge 4}.
\end{align}
Hence we get
\begin{align*}
\mathrm{tr}_{{\rm mat}_n} (\mathrm{g}^{-1} d\mathrm{g})^{\wedge 4}
\stackrel{\eqref{eq:closedform_proof4}}{=}
\mathrm{tr}_{{\rm mat}_n} (d\mathrm{g} \, \mathrm{g}^{-1})^{\wedge 4}
& \stackrel{\eqref{eq:closedform_proof3}}{=}
d \left( \mathrm{tr}_{{\rm mat}_n} (d\mathrm{g} \, \mathrm{g}^{-1} )^{\wedge 3} \right) \\
& \stackrel{\eqref{eq:closedform_proof2}}{=}
d \left( \mathrm{tr}_{{\rm mat}_n} (\mathrm{g}^{-1} d\mathrm{g})^{\wedge 3} \right)
\stackrel{\eqref{eq:closedform_proof1}}{=} - \mathrm{tr}_{{\rm mat}_n} (\mathrm{g}^{-1} d\mathrm{g})^{\wedge 4},
\end{align*}
and therefore can conclude that this $4$-form \eqref{eq:closedform_proof1} is $0$, as desired.
\end{proof}

Therefore $\eta_2$ can be thought of as representing an element $[\eta_2]$ of the third De Rham cohomology group $H^3_{\rm dR} (\mathrm{GL}_+(2,\mathbb{R}))$, which is canonically isomorphic to the third singular cohomology group $H^3(\mathrm{GL}_+(2,\mathbb{R});\mathbb{R})$ by De Rham's Theorem (see e.g. \cite[Appendix A]{M91}). We would like to identify the integral \eqref{eq:W_til_h_1_integral} as the evaluation of this third cohomology class on a third singular homology class. To do this carefully, we follow \cite{M91} for basic algebraic topology. 

\vs

Observe that the closed unit $3$-ball $\ol{B^3}$ is homeomorphic to the \emph{unit $3$-cube} $[0,1]^3 \subset \mathbb{R}^3$. Fix one such homeomorphism and call it
\begin{align}
\label{eq:sigma_3d_homeomorphism}
\sigma : [0,1]^3 \to \ol{B^3};
\end{align}
one can find a $\sigma$ that restricts to a $C^\infty$ diffeomorphism of the interiors. Let
\begin{align}
\label{eq:sigma_1}
\sigma_1 := J_{S^2}(B_\rho) \circ \sigma: [0,1]^3 \to \mathrm{GL}_+(2,\mathbb{R}).
\end{align}
So $\sigma_1$ represents a \emph{cubical singular $3$-chain} in $\mathrm{GL}_+(2,\mathbb{R})$ in the sense as in \cite[Chap.VII]{M91}, i.e. $\sigma_1 \in C_3(\mathrm{GL}_+(2,\mathbb{R}))$. In fact, $\sigma_1$ is a singular $3$-cycle, i.e. $\sigma_1$ is in the kernel of the `boundary map' $\partial_3 : C_3(\mathrm{GL}_+(2,\mathbb{R})) \to C_2(\mathrm{GL}_+(2,\mathbb{R}))$, for all the boundary points of $[0,1]^3$ are mapped by $\sigma_1$ to a single point in $\mathrm{GL}_+(2,\mathbb{R})$, namely the identity matrix, as we saw in \eqref{eq:J_S2_of_B1}, hence $\partial_3 \sigma_1$ is a degenerate $2$-chain. From \eqref{eq:J_S2_of_B1} one can also show that $\sigma_1$ is smooth. So we can write the integral \eqref{eq:W_til_h_1_integral} as $\int_{[0,1]^3} (\sigma_1)^* \eta_2$, which is precisely the value of the natural pairing between the De Rham coholomogy class $[\eta_2]$ in  $H^3_{\rm  dR}(\mathrm{GL}_+(2,\mathbb{R}))$ and the singular homology class $[\sigma_1]$ in $H_3(\mathrm{GL}_+(2,\mathbb{R}))$.

\vs

Meanwhile, the Gram-Schmidt process yields ${\rm GL}_+(n,\mathbb{R})\to {\rm SO}(n)$, which is a deformation retraction. Hence
\begin{align*}
  H_3(\mathrm{GL}_+(2,\mathbb{R})) \cong H_3({\rm SO}(2)) \cong H_3 (S^1) = 0.
\end{align*}

\vs

So $[\sigma_1]=0 \in H_3(\mathrm{GL}_+(2,\mathbb{R}))$, hence the integral \eqref{eq:W_til_h_1_integral} is zero, for it is the value of the pairing between $[\eta_2] \in H^3_{\rm dR}(\mathrm{GL}_+(2,\mathbb{R}))$ and $[\sigma_1] \in H_3(\mathrm{GL}_+(2,\mathbb{R}))$. Hence $W(\til{h}_1)=0$, as we wanted in \eqref{eq:to_establish}. This is the end of our proof of Prop.\ref{prop:W_zero_on_subgroup}.

\section{Central extension of $\mathrm{Vect}(S^1) \ltimes \mathrm{Vect}(S^1)$}

\subsection{Lie algebra central extensions, in general}
\label{subsec:Lie_algebra_central_extensions_in_general}

Let $\mathcal{G}$ be any Lie group, and let $\mathrm{Lie}(\mathcal{G})$ be the Lie algebra of $\mathcal{G}$, regarded as the right invariant vector fields on $\mathcal{G}$. Recall that, for our $\mathcal{G}$ we studied a central extension $\wh{\mathcal{G}}$ of $\mathcal{G}$ corresponding to the $\mathbb{R}$-valued group $2$-cocycle $\gamma : \mathcal{G} \times \mathcal{G} \to \mathbb{R}$. Any group $2$-cocycle $\gamma$ of $\mathcal{G}$ induces a map $\beta: \mathrm{Lie}(\mathcal{G}) \times \mathrm{Lie}(\mathcal{G}) \to \mathbb{R}$ via `differentiation':
\begin{align}
\label{eq:beta_definition}
\beta(V,W) := \left. \frac{\partial^2}{\partial t \partial s} \gamma(h_t, k_s) - \gamma(k_s, h_t) \right|_{t,s=0}, \qquad V,W\in \mathrm{Lie}(\mathcal{G}),
\end{align}
where $\{h_t\}_{t\in(-\epsilon,\epsilon)}$, $\{k_s\}_{s\in(-\epsilon,\epsilon)}$ are $C^1$ curves in $\mathcal{G}$ passing through the identity element, corresponding respectively to $V$, $W$. That is, $h_0=k_0=\mathrm{id}_\mathcal{G}$, and the `derivatives' of $h_t$ and $k_s$ at $t=0$ and $s=0$ are $V$ and $W$ respectively (one writes $\frac{d}{dt} h_t |_{t=0}=V$, $\frac{d}{ds} k_s |_{s=0} = W$). A more precise way to deal with these Lie algebra elements is to regard them as derivations on $C^\infty(\mathcal{G};\mathbb{R})$. Namely, $V$ acts (from the left) as derivation on $C^\infty(\mathcal{G};\mathbb{R})$ by 
$$
(V.F)(g) = \frac{d}{dt} F(h_tg)|_{t=0}.
$$ 
For $V,W \in \mathrm{Lie}(\mathcal{G})$, we define the Lie bracket $[V,W] \in \mathrm{Lie}(\mathcal{G})$ as the unique right invariant vector field on $\mathcal{G}$ acting as derivation on $C^\infty(\mathcal{G};\mathbb{R})$ as $[V,W].F = V.(W.F) - W.(V.F)$.

\vs

Then one can check that the map $\beta$ \eqref{eq:beta_definition} is an {\em $\mathbb{R}$-valued Lie algebra $2$-cocycle} $\beta$, i.e. a bilinear map $\mathrm{Lie}(\mathcal{G}) \times \mathrm{Lie}(\mathcal{G}) \to \mathbb{R}$ that is skew-symmetric
\begin{align}
\label{eq:general_beta_skew-symmetry}
\beta(V,W) = - \beta(W,V), \qquad \forall \, V,W\in \mathrm{Lie}(\mathcal{G}),
\end{align}
and satisfies the {\em Lie algebra $2$-cocycle property}:
\begin{align}
\label{eq:general_beta_2-cocycle_property}
\beta([V,W],U) + \beta([W,U],V) + \beta([U,V],W) = 0, \qquad \forall \, V,W,U\in \mathrm{Lie}(\mathcal{G}).
\end{align}

\vs

The central extension $\wh{\mathrm{Lie}(\mathcal{G})}$ of the Lie algebra $\mathrm{Lie}(\mathcal{G})$ is constructed as
$$
\wh{\mathrm{Lie}(\mathcal{G})} = \mathrm{Lie}(\mathcal{G}) \oplus \mathbb{R} = \{ (V, a) \, | \, V \in \mathrm{Lie}(\mathcal{G}), ~ a \in \mathbb{R} \}
$$
as a vector space, with the Lie bracket
\begin{align}
\label{eq:Lie_algebra_central_extension_definition}
[(V_1, a_1), \, (V_2, a_2)] := ([V_1,V_2], \, \beta(V_1,V_2)), \qquad V_1,V_2 \in \mathrm{Lie}(\mathcal{G}), \quad a_1,a_2 \in \mathbb{R}.
\end{align}
One observes that the above constructed $\wh{\mathrm{Lie}(\mathcal{G})}$ can be naturally identified with $\mathrm{Lie}(\wh{\mathcal{G}})$. Namely, let $(V,a) \in \wh{\mathrm{Lie}(\mathcal{G})}$. Let $\{h_t\}_{t\in(-\epsilon,\epsilon)}$ and $\{\alpha_t\}_{t\in(-\epsilon,\epsilon)}$ be $C^1$ curves passing through the identity respectively in $\mathcal{G}$ and in $\mathbb{R}$ corresponding to $V$ and $a$; so $\alpha_t = t a + o(t)$. Then $\{(h_t,\alpha_t)\}_{t\in(-\epsilon,\epsilon)}$ is a $C^1$ curve in $\wh{\mathcal{G}}$ passing through the identity, hence represents a tangent vector to $\wh{\mathcal{G}}$ at $\mathrm{id}_{\wh{\mathcal{G}}}$, yielding to a unique right invariant vector field on $\wh{\mathcal{G}}$, i.e. an element of $\mathrm{Lie}(\mathcal{G})$. One can then verify that thus constructed map $\wh{\mathrm{Lie}(\mathcal{G})} \to \mathrm{Lie}(\wh{\mathcal{G}})$ is a well-defined Lie algebra isomorphism. From now on, we will identify $\wh{\mathrm{Lie}(\mathcal{G})}$ and $\mathrm{Lie}(\wh{\mathcal{G}})$ via this isomorphism.

\vs

\subsection{Disc vector fields}
\label{subsec:disc_vector_fields}

In order to apply what we discussed in the previous subsection to our group $\mathcal{G}$, we must first investigate what we mean by vector fields on the circle and on the disc.

\begin{definition}
We denote by ${\rm Vect}(S^1)$ the set of all smooth real vector fields on $S^1 \approx \mathbb{R}/2\pi\mathbb{Z}$. In other words,
\begin{align}
{\rm Vect}(S^1) := \left\{ \left. v = v(\theta) \frac{\partial}{\partial \theta} \, \right| \, v(\theta) \in C^\infty(\mathbb{R}/2\pi\mathbb{Z}; \mathbb{R}) \right\}.
\end{align}
\end{definition}
\begin{lemma}
${\rm Vect}(S^1)$ is a real Lie algebra, with the Lie bracket being the commutator of vector fields on $S^1$, i.e. 
\begin{align}
\label{eq:Vect_S1_Lie_bracket}
\left[v(\theta)\frac{\partial}{\partial\theta}, ~ w(\theta)\frac{\partial}{\partial\theta}\right] = (v(\theta)w'(\theta) - v'(\theta)w(\theta))\frac{\partial}{\partial\theta},
\end{align}
where the prime ${}'$ means the derivative with respect to $\theta$. \qed
\end{lemma}
\begin{lemma}
\label{lem:Lie_G_is_Vect_S1}
The Lie algebra ${\rm Lie}({\bf G})$ of ${\bf G} = {\rm Diff}_+(S^1)$ coincides with ${\rm Vect}(S^1)$. \qed
\end{lemma}

\begin{definition}
A \emph{generalized smooth real disc vector field} $V$ is an expression
\begin{align}
\label{eq:disc_vector_field_V}
  V = V_1\smallvectwo{x}{y} \frac{\partial}{\partial x} + V_2\smallvectwo{x}{y} \frac{\partial}{\partial y}, 
\end{align}
for some $V_1,V_2 \in C^\infty(D^2;\mathbb{R})$. Denote by ${\rm GVect}(D^2)$ the set of all generalized smooth real disc vector fields.
\end{definition}
\begin{lemma}
$\mathrm{GVect}(D^2)$ is a real Lie algebra, with the Lie bracket being the commutator of vector fields on $\mathbb{R}^2$. \qed
\end{lemma}

In order to study the asymptotic behavior of generalized smooth real disc vector fields, we write $V$ \eqref{eq:disc_vector_field_V} in terms of polar coordinates $r,\theta$ of $D^2$. From $x = r \cos\theta$, $y=r \sin\theta$ we have $dx = \cos\theta \, dr - r \sin\theta \, d\theta$, $dy = \sin\theta \, dr + r \cos\theta \, d\theta$ and
\begin{align}
\label{eq:vector_field_change_formula_for_polar}
\frac{\partial}{\partial x} = \cos\theta \frac{\partial}{\partial r} - r^{-1} \sin\theta \frac{\partial}{\partial \theta} \quad\mbox{and}\quad
\frac{\partial}{\partial y} = \sin\theta \frac{\partial}{\partial r} + r^{-1} \cos\theta \frac{\partial}{\partial \theta} \quad\mbox{on}\quad D^2\setminus\{0\},
\end{align}
and therefore
\begin{align}
\label{eq:V_in_polar}
V = (V_1 \cdot \cos\theta + V_2 \cdot \sin\theta) \frac{\partial}{\partial r} + r^{-1} ( -V_1\cdot \sin\theta + V_2 \cdot \cos\theta) \frac{\partial}{\partial \theta} \quad\mbox{on}\quad D^2\setminus\{0\},
\end{align}
where the arguments of $V_1$, $V_2$ here are $\smallvectwo{r\cos\theta}{r\sin\theta}$.

\begin{definition}
A \emph{genuine smooth real disc vector field} is a generalized smooth real disc vector field $V$ that satisfies the \emph{tangent condition}: $V_1,V_2$ \eqref{eq:disc_vector_field_V} extend continuously to $\ol{D^2}$ and satisfy
\begin{align}
\label{eq:tangent_condition}
V_1\smallvectwo{\cos\theta}{\sin\theta} \cdot \cos\theta + V_2\smallvectwo{\cos\theta}{\sin\theta} \cdot \sin\theta = 0, \qquad \forall \theta \in \mathbb{R}/2\pi\mathbb{Z}.
\end{align}
Denote by ${\rm Vect}(D^2) \subset {\rm GVect}(D^2)$ the set of all genuine smooth real disc vector fields.
\end{definition}
So genuine smooth real disc vector fields are what exponentiate to flows of $D^2$; they `preserve' the disc. For all discussion from now on, it is much more convenient to work solely with the polar coordinates. Write any $V\in {\rm GVect}(D^2)$ as
$$
V = V_3\smallvectwo{r}{\theta} \frac{\partial}{\partial r} + V_4\smallvectwo{r}{\theta} \frac{\partial}{\partial \theta} \quad\mbox{on}\quad D^2\setminus\{0\};
$$
thus
\begin{align}
\label{eq:V3_and_V4_in_V1_and_V2}
\left\{
\begin{array}{l}
V_3\smallvectwo{r}{\theta} = V_1\smallvectwo{r\cos\theta}{r\sin\theta} \cdot \cos\theta + V_2\smallvectwo{r\cos\theta}{r\sin\theta} \cdot \sin\theta, \\
V_4\smallvectwo{r}{\theta} = r^{-1}\left( -V_1\smallvectwo{r\cos\theta}{r\sin\theta} \cdot \sin\theta + V_2\smallvectwo{r\cos\theta}{r\sin\theta} \cdot \cos\theta \right)
\end{array} \right.
\end{align}
as we already saw in \eqref{eq:V_in_polar}. 
\begin{lemma}
For $V\in {\rm GVect}(D^2)$, the condition $V\in {\rm Vect}(D^2)$ is equivalent to $\lim_{r\nearrow 1} V_3\smallvectwo{r}{\theta} = 0$, $\forall \theta \in \mathbb{R}/2\pi\mathbb{Z}$, which we write as $V_3\smallvectwo{1}{\theta}=0$, $\forall \theta \in \mathbb{R}/2\pi\mathbb{Z}$, or as $V_3|_{r=1}\equiv 0$. \qed
\end{lemma}
Suppose $V\in {\rm Vect}(D^2)$ is the tangent vector to a smooth curve in ${\rm Diff}_{+{\rm ar}}(\ol{D^2})$ (or in $\mathcal{G} = {\rm Diff}_{+{\rm ar}}(\ol{D^2})_0$) passing through the identity. We then can see that $V_3\equiv 0$ and $\partial_r V_4 \equiv 0$ on some neighborhood of $S^1$ in $\ol{D^2}$. When $V$ is tangent to a curve in ${\rm Diff}_{+{\rm ar}}(\ol{D^2})_{\rm bt}$ (or in $\mathcal{H} = \mathcal{G}_{\rm bt} = {\rm Diff}_{\rm c}(D^2)_0$), we can see $V_3\equiv 0 \equiv V_4$ on some neighborhood of $S^1$ in $\ol{D^2}$. We would like to generalize these notions to elements of ${\rm GVect}(D^2)$.
\begin{definition}
A generalized smooth real disc vector field $V$ is said to be \emph{asymptotically radial} if $V_1,V_2$ \eqref{eq:disc_vector_field_V} extend continuously to $\ol{D^2}$ and there exists $\epsilon \in (0,1)$ such that
\begin{align}
\label{eq:ar_condition_for_V}
\partial_r V_3 \equiv 0 \equiv \partial_r V_4 \quad\mbox{on some neighborhood of $S^1$ in $\ol{D^2}$.}
\end{align}
We call such $V$ \emph{asymptotically zero} if
\begin{align}
\label{eq:az_condition_for_V}
V_3 \equiv 0 \equiv V_4 \quad\mbox{on some neighborhood of $S^1$ in $\ol{D^2}$.}
\end{align}

Denote by ${\rm GVect}_{\rm ar}(D^2)$ (resp. ${\rm GVect}_{\rm az}(D^2)$) the set of all asymptotically radial (resp. asymptotically zero) generalized smooth real disc vector fields, and by ${\rm Vect}_{\rm ar}(D^2)$ (resp. ${\rm Vect}_{\rm az}(D^2)$) the set of all asymptotically radial (resp. asymptotically zero) genuine smooth real disc vector fields.
\end{definition}

\begin{lemma}
${\rm GVect}_{\rm az}(D^2)$ coincides with ${\rm Vect}_{\rm az}(D^2)$. \qed
\end{lemma}

\begin{lemma}
$\mathrm{Vect}(D^2)$, $\mathrm{GVect}_{\rm ar}(D^2)$, $\mathrm{Vect}_{\rm ar}(D^2)$, $\mathrm{GVect}_{\rm az}(D^2)$, $\mathrm{Vect}_{\rm az}(D^2)$ are real Lie subalgebras of $\mathrm{GVect}(D^2)$. 
\end{lemma}
{\it Proof.} We must check that each subset of $\mathrm{GVect}(D^2)$ in the statement of the lemma is closed under the Lie bracket; to see that each subset is a real vector space is easy. For this purpose and also for later use, we investigate the Lie bracket of disc vector fields, in polar coordinates. Let $V^{(a)}, V^{(b)} \in \mathrm{GVect}(D^2)$, and let 
$$
V^{(c)} := [V^{(a)}, V^{(b)}] \in \mathrm{GVect}(D^2).
$$
From the usual vector field commutator formula we get
\begin{align}
\label{eq:V_c_polar_components}
V^{(c)}_j = V^{(a)}_3 \cdot (\partial_r V^{(b)}_j) + V^{(a)}_4 \cdot (\partial_\theta V^{(b)}_j) - V^{(b)}_3 \cdot (\partial_r V^{(a)}_j) - V^{(b)}_4 \cdot (\partial_\theta V^{(a)}_j), \quad j=3,4,
\end{align}
where $\partial_r$ and $\partial_\theta$ stand for the corresponding partial derivatives. 

\vs

First, consider the case $\mathrm{Vect}(D^2)$. Let $V^{(a)},V^{(b)}\in \mathrm{Vect}(D^2)$ and let $V^{(c)} = [V^{(a)},V^{(b)}]$. Then from \eqref{eq:V_c_polar_components} and $V^{(a)}_3|_{r=1}=V^{(b)}_3|_{r=1}=0$ we can deduce $V^{(c)}_3|_{r=1}=0$, meaning $V^{(c)}\in \mathrm{Vect}(D^2)$, proving that $\mathrm{Vect}(D^2)$ is closed under the Lie bracket.

\vs

Let us now check the case $\mathrm{GVect}_{\rm ar}(D^2)$. So, assume $V^{(a)},V^{(b)} \in \mathrm{GVect}_{\rm ar}(D^2)$, and let $V^{(c)}=[V^{(a)}, V^{(b)}]$. From the asymptotically radial conditions on $V^{(a)},V^{(b)}$, on some neighborhood of $S^1$ in $\ol{D^2}$ we have $V^{(c)}_j = V^{(a)}_4 \, (\partial_\theta V^{(b)}_j) - V^{(b)}_4 \, (\partial_\theta V^{(a)}_j)$, $j=3,4$. One then easily checks that $\partial_r V^{(c)}_j$ too is constantly zero on this same neighborhood, $j=3,4$. Hence $V^{(c)} \in \mathrm{GVect}_{\rm ar}(D^2)$, showing that $\mathrm{GVect}_{\rm ar}(D^2)$ is indeed closed under the Lie bracket.

\vs

We omit the proof for the asymptotically zero cases. \qed

\vs

The following two lemmas are straightforward observations.
\begin{lemma}
\label{lem:az_is_ideal}
$\mathrm{Vect}_{\rm az}(D^2)$ is an ideal in $\mathrm{GVect}_{\rm ar}(D^2)$, hence also in $\mathrm{Vect}_{\rm ar}(D^2)$. \qed
\end{lemma}

\begin{lemma}
The Lie algebra $\mathrm{Lie}(\mathcal{G})$ of our disc diffeomorphism group $\mathcal{G} = {\rm Diff}_{+{\rm ar}}(\ol{D^2})_0$ naturally coincides with $\mathrm{Vect}_{\rm ar}(D^2)$. The Lie algebra $\mathrm{Lie}(\mathcal{H})$ of the boundary trivial subgroup $\mathcal{H} = \mathcal{G}_{\rm bt} = {\rm Diff}_{\rm c}(D^2)_0$ naturally coincides with $\mathrm{Vect}_{\rm az}(D^2)$. \qed
\end{lemma}

From $\mathcal{H} \lhd \mathcal{G}$ with $\mathbf{G} \cong \mathcal{G} / \mathcal{H}$, we see that $\mathrm{Lie}(\mathcal{H})$ is an ideal of $\mathrm{Lie}(\mathcal{G})$, with
$$
\mathrm{Lie}(\mathbf{G}) \cong \mathrm{Lie}(\mathcal{G})/\mathrm{Lie}(\mathcal{H}).
$$

\begin{lemma}
\label{lem:Vect_S1_as_quotient}
One has the natural isomorphism of Lie algebras
\begin{align}
\label{eq:Vects_as_quotient}
  \mathrm{Vect}_{\rm ar}(D^2)/\mathrm{Vect}_{\rm az}(D^2)\cong \mathrm{Vect}(S^1),
\end{align}
induced by the `restriction' map
\begin{align}
\label{eq:vector_field_restriction}
\mathrm{Vect}_{\rm ar}(D^2) \to \mathrm{Vect}(S^1) ~ : ~ V \mapsto V|_{S^1} := V_4\smallvectwo{1}{\theta} \, \frac{\partial}{\partial \theta}. \qed
\end{align}
\end{lemma}

The Lie algebra version of our construction provides a $1$-dimensional central extension of the Lie algebra $\mathrm{Vect}_{\rm ar}(D^2) = \mathrm{Lie}(\mathcal{G})$, whose factor by $\mathrm{Vect}_{\rm az}(D^2) = \mathrm{Lie}(\mathcal{H})$ would give a $1$-dimensional central extension of the Lie algebra $\mathrm{Vect}(S^1) = \mathrm{Lie}(\mathbf{G})$, whose appropriate complexification is what is referred to as the \emph{Virasoro algebra}. We take a step further at this moment. Namely, in the left-hand-side of \eqref{eq:Vects_as_quotient} we replace $\mathrm{Vect}_{\rm ar}(D^2)$ by $\mathrm{GVect}_{\rm ar}(D^2)$. Thanks to Lem.\ref{lem:az_is_ideal}, the quotient $\mathrm{GVect}_{\rm az}(D^2) = \mathrm{Vect}_{\rm az}(D^2)$ is a Lie algebra; we shall prove that this quotient Lie algebra is isomorphic to the semidirect product of two copies of $\mathrm{Vect}(S^1)$, which in \S\ref{subsec:central_extension_of_two_copies} will eventually be centrally extended via a natural generalization of our construction in the present paper.

\begin{definition}[semidirect product of Lie algebras]
\label{def:semidirect_product_of_Lie_algebras}
Let $\frak{g}$ and $\frak{h}$ be Lie algebras with brackets $[ \, , \, ]_\frak{g}$ and $[ \, , \, ]_\frak{h}$, respectively, and suppose that there is a Lie algebra map $\pi : \frak{g} \to \mathrm{Der}(\frak{h})$, where $\mathrm{Der}(\frak{h})$ is the Lie algebra of all \emph{derivations} of $\frak{h}$; this means
$$
\pi(v) ([w_1,w_2]_\frak{h}) = [\pi(v) (w_1), w_2]_\frak{h} + [w_1, \pi(v)(w_2)]_\frak{h}, \qquad \forall v\in \frak{g}, \quad \forall w_1,w_2 \in \frak{h}.
$$
Then the \emph{semidirect product} $\frak{g} \ltimes \frak{h}$ is defined as the Lie algebra whose underlying vector space is $\frak{g} \times \frak{h}$, with the Lie bracket $[ \, , \,]$ defined as 
$$
[(v_1,w_1), (v_2,w_2)] := ([v_1,v_2]_\frak{g}, \, [w_1,w_2]_\frak{h} + \pi(v_1)(w_2) - \pi(v_2)(w_1)), \qquad \forall v_1,v_2 \in \frak{g}, \quad \forall w_1,w_2 \in \frak{h}.
$$
\end{definition}
\begin{remark}
Sometimes this is called the `semidirect sum'.
\end{remark}
\begin{lemma}
$\frak{g} \ltimes \frak{h}$ is a well-defined Lie algebra.\qed
\end{lemma}
In particular, when $\frak{h}$ is an abelian Lie algebra, $\mathrm{Der}(\frak{h})$ coincides with the Lie algebra of all vector space endomorphisms of $\frak{h}$.

\begin{proposition}
\label{prop:semidirect_Vect_S1}
One has a canonical Lie algebra isomorphism
\begin{align*}
\mathrm{GVect}_{\rm ar}(D^2)/\mathrm{Vect}_{\rm az}(D^2) \cong \mathrm{Vect}(S^1) \ltimes (\mathrm{Vect}(S^1))_{\rm ab},
\end{align*}
where $(\mathrm{Vect}(S^1))_{\rm ab}$ denotes the abelian Lie algebra with the underlying vector space being $\mathrm{Vect}(S^1)$, with the action of $\mathrm{Vect}(S^1)$ on $(\mathrm{Vect}(S^1))_{\rm ab}$ given as follows:
\begin{align}
\label{eq:Vect_S1_action_on_itself}
  \mathrm{Vect}(S^1) \,\, \rotatebox[origin=c]{-90}{$\circlearrowright$}\ (\mathrm{Vect}(S^1))_{\rm ab} ~ : ~ \left(v(\theta)\frac{\partial}{\partial \theta}\right).\left(w(\theta)\frac{\partial}{\partial \theta}\right) := v(\theta) w'(\theta) \frac{\partial}{\partial \theta},
\end{align}
for each $v(\theta)\frac{\partial}{\partial \theta} \in \mathrm{Vect}(S^1)$ and $w(\theta)\frac{\partial}{\partial \theta} \in (\mathrm{Vect}(S^1))_{\rm ab}$, where $w'(\theta) = \frac{\partial w(\theta)}{\partial \theta}$.
\end{proposition}

{\it Proof of Prop.\ref{prop:semidirect_Vect_S1}.} Let us first construct a non-canonical vector space isomorphism
\begin{align}
\label{eq:non-canonical_vector_space_isomorphism}
\mathrm{GVect}_{\rm ar}(D^2) \stackrel{\sim}{\longrightarrow} \mathrm{Vect}_{\rm ar}(D^2) \times (\mathrm{Vect}(S^1))_{\rm ab}.
\end{align}
Let $\xi:[0,1]\to [0,1]$ be a function as in Lem.\ref{lem:xi}; we choose one and fix it. For each $f \in C^\infty(\mathbb{R}/2\pi\mathbb{Z};\mathbb{R})$ define
\begin{align*}
  W_f := (W_f)_1\smallvectwo{x}{y} \frac{\partial}{\partial x} + (W_f)_2\smallvectwo{x}{y} \frac{\partial}{\partial y}
\end{align*}
where
$$
(W_f)_1\smallvectwo{r\cos\theta}{r\sin\theta} := \xi(r) \cdot (\cos\theta) \cdot f(\theta), \quad (W_f)_2\smallvectwo{r\cos\theta}{r\sin\theta} := \xi(r) \cdot (\sin\theta)\cdot f(\theta), \quad \left\{ \begin{array}{l} \forall r \in [0,1), \\ \forall \theta\in \mathbb{R}/2\pi\mathbb{Z}.\end{array} \right.
$$
We find that $W_f \in \mathrm{GVect}_{\rm ar}(D^2)$. The polar components are easier to understand:
\begin{align}
\label{eq:W_f_polar_components}
(W_f)_3\smallvectwo{r}{\theta} = \xi(r) f(\theta), \qquad (W_f)_4\smallvectwo{r}{\theta}=0, \qquad \left\{ \begin{array}{l} \forall r \in (0,1), \\ \forall \theta\in \mathbb{R}/2\pi\mathbb{Z}.\end{array} \right. 
\end{align}

\vs

For each $V \in \mathrm{GVect}_{\rm ar}(D^2)$, define a $C^\infty$ function $f_V: \mathbb{R}/2\pi\mathbb{Z} \to \mathbb{R}$ as
\begin{align}
\label{eq:f_V}
f_V(\theta) := V_3|_{r=1}(\theta) = V_3\smallvectwo{1}{\theta} = V_1\smallvectwo{\cos\theta}{\sin\theta} \cdot \cos\theta + V_2\smallvectwo{\cos\theta}{\sin\theta} \cdot \sin\theta, \qquad \forall \theta \in \mathbb{R}/2\pi\mathbb{Z}.
\end{align}
Denote
\begin{align}
\label{eq:V_ar}
  V_{\rm ar} := V - W_{f_V}, \qquad \forall V \in \mathrm{GVect}_{\rm ar}(D^2).
\end{align}
It is a straightforward exercise to verify that $V_{\rm ar} \in \mathrm{Vect}_{\rm ar}(D^2)$. We claim that
$$
V \mapsto \left( V_{\rm ar}, ~ f_V(\theta)\frac{\partial}{\partial \theta} \right)
$$
is a sought-for isomorphism \eqref{eq:non-canonical_vector_space_isomorphism}. Indeed, the $\mathbb{R}$-linearity, the injectivity, and the surjectivity, are obvious. 

\vs

Adding an element of $\mathrm{Vect}_{\rm az}(D^2)$ to $V \in \mathrm{GVect}_{\rm ar}(D^2)$ does not alter $f_V$, and results in adding that element to $V_{\rm ar}$. Hence the above map induces a well-defined vector space isomorphism
\begin{align}
\label{eq:canonical_vector_space_isomorphism}
\mathrm{GVect}_{\rm ar}(D^2)/\mathrm{Vect}_{\rm az}(D^2) \stackrel{\sim}{\longrightarrow} (\mathrm{Vect}_{\rm ar}(D^2)/\mathrm{Vect}_{\rm az}(D^2)) \ltimes (\mathrm{Vect}(S^1))_{\rm ab}.
\end{align}
Lem.\ref{lem:Vect_S1_as_quotient} says that $\mathrm{Vect}_{\rm ar}(D^2)/\mathrm{Vect}_{\rm az}(D^2)$ can be identified with $\mathrm{Vect}(S^1)$ as Lie algebras, with the class in $\mathrm{Vect}_{\rm ar}(D^2)/\mathrm{Vect}_{\rm az}(D^2)$ for $V_{\rm ar} \in \mathrm{Vect}_{\rm ar}(D^2)$ corresponding to $V_{\rm ar}|_{S^1} = (V_{\rm ar})_4\smallvectwo{1}{\theta} \frac{\partial}{\partial \theta} \in \mathrm{Vect}(S^1)$. From this, we define the Lie algebra structure on the right-hand-side of \eqref{eq:canonical_vector_space_isomorphism} using Def.\ref{def:semidirect_product_of_Lie_algebras} and the action \eqref{eq:Vect_S1_action_on_itself}. From \eqref{eq:W_f_polar_components} and \eqref{eq:V_ar}, one can observe that for each $V\in \mathrm{GVect}_{\rm ar}(D^2)$ we have
\begin{align}
\label{eq:V_ar_polar_components}
(V_{\rm ar})_3\smallvectwo{r}{\theta} = V_3\smallvectwo{r}{\theta} - \xi(r) \, f_V(\theta), \qquad (V_{\rm ar})_4\smallvectwo{r}{\theta} = V_4\smallvectwo{r}{\theta},\qquad \left\{ \begin{array}{l} \forall r \in (0,1), \\ \forall \theta\in \mathbb{R}/2\pi\mathbb{Z}.\end{array} \right. 
\end{align}
In particular, $V_{\rm ar}|_{S^1} = V_4\smallvectwo{1}{\theta}\frac{\partial}{\partial\theta} = V|_{S^1}$; let us write the map \eqref{eq:canonical_vector_space_isomorphism} more clearly:
\begin{align}
\label{eq:map_more_clearly}
{\renewcommand{\arraystretch}{1.4} \begin{array}{l}
\mbox{class of $V\in \mathrm{GVect}_{\rm ar}(D^2) ~=$}~~~ V + \mathrm{Vect}_{\rm az}(D^2)  \\
\displaystyle \quad \longmapsto ~ \left( V|_{S^1}, \frac{\partial}{\partial \theta}, ~ f_V(\theta) \frac{\partial}{\partial \theta} \right) = \left( V_4\smallvectwo{1}{\theta} \frac{\partial}{\partial \theta}, ~ V_3\smallvectwo{1}{\theta} \frac{\partial}{\partial \theta} \right).
\end{array}}
\end{align}
So, we now see that this map is canonical, for adding an element of $\mathrm{Vect}_{\rm az}(D^2)$ to $V\in \mathrm{GVect}_{\rm ar}(D^2)$ does not alter the functions $V_4\smallvectwo{1}{\theta}, f_V(\theta) \in C^\infty(\mathbb{R}/2\pi\mathbb{Z};\mathbb{R})$, which are completely determined by $V$, without any external choice such as the auxiliary function $\xi$. It only remains to show that this map is a Lie algebra homomorphism.

\vs

Let $V^{(a)}, V^{(b)} \in \mathrm{GVect}_{\rm ar}(D^2)$ and let
$$
V^{(c)} := [V^{(a)}, V^{(b)}] \in \mathrm{GVect}_{\rm ar}(D^2),
$$
each of the three elements $V^{(a)},V^{(b)},V^{(c)}$ being considered as representing the corresponding class in $\mathrm{GVect}_{\rm ar}(D^2)/\mathrm{Vect}_{\rm az}(D^2)$. The formula \eqref{eq:V_c_polar_components} for $j=4$, restricted to $r=1$, tells us
\begin{align*}
V^{(c)}_4|_{r=1} & = V^{(a)}_3|_{r=1} \, \underbrace{(\partial_r V^{(b)}_4)|_{r=1}} + V^{(a)}_4|_{r=1} \, (\partial_\theta (V^{(b)}_4|_{r=1})) \\
& \quad  - V^{(b)}_3|_{r=1} \, \underbrace{(\partial_r V^{(a)}_4)|_{r=1}} - V^{(b)}_4|_{r=1} \, (\partial_\theta (V^{(a)}_4|_{r=1})),
\end{align*}
where the two underbraced parts are zero because of the condition \eqref{eq:ar_condition_for_V} for $V^{(a)}$ and $V^{(b)}$. Recalling the definition of the `restriction' $V|_{S^1} = V|_{r=1}(\theta) \frac{\partial}{\partial \theta}$ and, we thus obtain
\begin{align}
\label{eq:Lie_algebra_homomorphism_eq1}
V^{(c)}|_{S^1} = \left[ V^{(a)}|_{S^1}, ~ V^{(b)}|_{S^1} \right],
\end{align}
where the Lie bracket of the right-hand-side is that of $\mathrm{Vect}(S^1)$ as written in \eqref{eq:Vect_S1_Lie_bracket}.

\vs

In the meantime, observe
\begin{align*}
  f_{V^{(c)}} \stackrel{\eqref{eq:f_V}}{=} V^{(c)}_3 |_{r=1} & \stackrel{\eqref{eq:V_c_polar_components}}{=} V^{(a)}_4|_{r=1} \cdot (\partial_\theta (V^{(b)}_3|_{r=1})) - V^{(b)}_4|_{r=1} \cdot(\partial_\theta(V^{(a)}_3|_{r=1})) \\
& \stackrel{\eqref{eq:f_V}}{=} V^{(a)}_4|_{r=1} \cdot f'_{V^{(b)}} - V^{(b)}_4|_{r=1} \cdot f'_{V^{(a)}}.
\end{align*}
Using the action described in \eqref{eq:Vect_S1_action_on_itself}, one can write
\begin{align}
\label{eq:Lie_algebra_homomorphism_eq2}
f_{V^{(c)}}(\theta) \frac{\partial}{\partial \theta} = \left( V^{(a)}_4\smallvectwo{1}{\theta} \frac{\partial}{\partial\theta} \right).\left( f_{V^{(b)}}(\theta) \frac{\partial}{\partial \theta} \right) - \left( V^{(b)}_4\smallvectwo{1}{\theta} \frac{\partial}{\partial\theta} \right).\left( f_{V^{(a)}}(\theta) \frac{\partial}{\partial \theta} \right).
\end{align}
So, in view of Def.\ref{def:semidirect_product_of_Lie_algebras}, the equalities \eqref{eq:Lie_algebra_homomorphism_eq1} and \eqref{eq:Lie_algebra_homomorphism_eq2} tell us that the map \eqref{eq:map_more_clearly} is indeed a Lie algebra homomorphism, finishing the proof. \qed

\vs

An interesting remark is that the action \eqref{eq:Vect_S1_action_on_itself} is induced by only a `half' of the usual adjoint action of $\mathrm{Vect}(S^1)$ on itself, instead of the full adjoint action. However, one can still check that \eqref{eq:Vect_S1_action_on_itself} is a well-defined Lie algebra action.

\subsection{Lie algebra $2$-cocycle of $\mathrm{Vect}(S^1) \ltimes (\mathrm{Vect}(S^1))_{\rm ab}$}
\label{subsec:central_extension_of_two_copies}

We shall first apply the construction of \S\ref{subsec:Lie_algebra_central_extensions_in_general} to our group $\mathcal{G}$, to obtain a Lie algebra $2$-cocycle of $\mathrm{Lie}(\mathcal{G})$. In order to `differentiate' our group $2$-cocycle $\gamma$ of the group $\mathcal{G} = {\rm Diff}_{+{\rm ar}}(\ol{D^2})_0$, we need the following.
\begin{lemma}[variation of $\theta$]
\label{lem:variation_of_theta}
As in Def.\ref{def:G_M_theta_M}, let $M$ be an $n$-dimensional manifold with boundary embedded in $\mathbb{R}^n$, $\mathscr{G} := {\rm Diff}_{+*}(M)\subset {\rm Diff}_+(M)$ (with $*$ representing a certain analytic condition), and $\theta_M(g) := J_M(g)^{-1} dJ_M(g)$ for each $g\in \mathscr{G}$ where $J_M(g)$ is the Jacobian matrix of $g$. Let $V\in \mathrm{Lie}(\mathscr{G}) \subset \mathrm{Vect}(M)$, and let $\{h_s\}_{s\in (-\epsilon,\epsilon)}$ be a smooth curve in $\mathscr{G}$ at the identity corresponding to $V$. For $g\in \mathscr{G}$ we have the following variational formula for $\theta_M$ at $g$ in the direction $V$:
\begin{align}
\label{eq:variation_theta1}
\left. \frac{d}{ds} \theta_M(g\circ h_s) \right|_{s=0} = 
- J_M(g)^{-1} \, J_M(J_M(g) V) \, \theta_M(g)
+ J_M(g)^{-1} dJ_M(J_M(g) V),
\end{align}
where $V$ is thought of as an $\mathbb{R}^n$-column-vector-valued function on $M$. In particular,
\begin{align}
\label{eq:variation_theta2}
\left. \frac{d}{ds} \theta_M(h_s) \right|_{s=0} = d J_M(V).
\end{align}
\end{lemma}

\begin{proof}
First, write
\begin{align}
\label{eq:h_t_in_Diff_M}
h_s = \mathrm{id}_M + s V + o(s),
\end{align}
using the pointwise addition, where the addition is taken with respect to the vector space structure of the ambient $\mathbb{R}^n$. 
Then
\begin{align*}
g\circ h_s = g + s \, J_M(g) V + o(s).
\end{align*} 
We can extend the definition of $J_M$ (Def. \ref{def:G_M_theta_M}) to any $\mathbb{R}^n$-valued differentiable functions on $M$. One observes from its formula \eqref{eq:J_M} that such assignment $J_M$ is $\mathbb{R}$-linear, hence
$$
J_M(g\circ h_s) = J_M(g) + s \, J_M (J_M(g) V) + o(s).
$$
One can then easily check
$$
J_M(g\circ h_s)^{-1} = J_M(g)^{-1} - s \, J_M(g)^{-1} \, J_M (J_M(g) V) \, J_M(g)^{-1} + o(s),
$$
from which it follows that
\begin{align*}
\theta_M(g\circ h_s) & = J_M(g\circ h_s)^{-1} \, dJ_M(g\circ h_s) \\
& = \left( J_M(g)^{-1} - s \, J_M(g)^{-1} \, J_M (J_M(g) V) \, J_M(g)^{-1} + o(s) \right) \\
& \quad \cdot (dJ_M(g) + s \, d J_M(J_M(g) V) + o(s) ) \\
& = \theta_M(g) + s \left(
- J_M(g)^{-1} \, J_M(J_M(g) V) \, \theta_M(g)
+ J_M(g)^{-1} dJ_M(J_M(g) V)
\right) + o(s),
\end{align*}
as desired. 

\vs

When $g=\mathrm{id}_M$ we get $J_M(g) = I_n$ (the constant function with value being the identity matrix) and so $\theta_M(g) = 0$, and therefore we get
\begin{align}
\nonumber
\theta_M(h_s) = s \,  dJ_M(V)  + o(s).
\end{align}

\end{proof}

The following is a straightforward exercise using Lem.\ref{lem:variation_of_theta}, and the skew-symmetry:
\begin{lemma}
The Lie algebra $2$-cocycle $\beta$ of $\mathrm{Lie}(\mathcal{G}) = \mathrm{Vect}_{\rm ar}(D^2)$ corresponding via \eqref{eq:beta_definition} to the group $2$-cocycle $\gamma$ (Def.\ref{def:our_gamma}) of $\mathcal{G} = {\rm Diff}_{+{\rm ar}}(\ol{D^2})_0$ (Def.\ref{def:our_model_of_the_disc_diffeomorphism_group}) is given by
\begin{align}
\label{eq:our_beta_formula}
  \beta(V,W) = - 6c_0 \int_{D^2} \mathrm{tr}_{\mathrm{mat}_2} (dJ_{D^2}(V) \wedge dJ_{D^2}(W)),
\end{align}
for each $V,W \in \mathrm{Lie}(\mathcal{G}) = \mathrm{Vect}_{\rm ar}(D^2)$. \qed
\end{lemma}
In particular, if we take \eqref{eq:our_beta_formula} as the definition of a map $\beta : \mathrm{Vect}_{\rm ar}(D^2) \times \mathrm{Vect}_{\rm ar}(D^2) \to \mathbb{R}$, it satisfies the skew-symmetry \eqref{eq:general_beta_skew-symmetry} and the Lie algebra $2$-cocycle property \eqref{eq:general_beta_2-cocycle_property}. Denote by $\wh{\mathrm{Vect}_{\rm ar}}(D^2) = \wh{\mathrm{Lie}(\mathcal{G})}$ be the corresponding Lie algebra central extension.

\vs

We extend $\beta$ to $\mathrm{GVect}_{\rm ar}(D^2)$ by the same formula \eqref{eq:our_beta_formula}, and denote it by $\mathrm{G}\beta$.
\begin{lemma}
\label{lem:Gbeta_is_well-defined}
For every $V,W\in\mathrm{GVect}_{\rm ar}(D^2)$, define $(\mathrm{G}\beta)(V,W)$ as the right-hand-side of \eqref{eq:our_beta_formula}. This yields a well-defined $\mathbb{R}$-bilinear skew-symmetric function $\mathrm{G}\beta : \mathrm{GVect}_{\rm ar}(D^2) \times \mathrm{GVect}_{\rm ar}(D^2) \to \mathbb{R}$.
\end{lemma}

\begin{proof}
Note that if we write $V \in \mathrm{GVect}_{\rm ar}(D^2)$ as in \eqref{eq:disc_vector_field_V}, then the symbol $J_{D^2}(V)$ means
\begin{align}
\label{eq:J_D2_V}
J_{D^2}(V) = \mattwo{\partial_x V_1}{\partial_y V_1}{\partial_x V_2}{\partial_y V_2}.
\end{align}
As an example, take $\partial_x V_1$. From \eqref{eq:V3_and_V4_in_V1_and_V2} we have
\begin{align}
\label{eq:V1_V2_in_V3_V4}
\left\{ 
{\renewcommand{\arraystretch}{1.4}
\begin{array}{l}
V_1\smallvectwo{x}{y} = V_3\smallvectwo{r}{\theta} \cdot \cos\theta - V_4 \smallvectwo{r}{\theta} \cdot r \sin\theta \\
V_2 \smallvectwo{x}{y} = V_3\smallvectwo{r}{\theta} \cdot \sin\theta + V_4\smallvectwo{r}{\theta} \cdot r \cos\theta
\end{array}}
\right.
\qquad \mbox{on $D^2\setminus\{0\}$},
\end{align}
and from \eqref{eq:vector_field_change_formula_for_polar} we have
\begin{align}
\label{eq:partial_x_y_of_V_1_2_in_polar_partials}
\left\{ 
{\renewcommand{\arraystretch}{1.4}
\begin{array}{l}
\partial_x V_j = \cos\theta \, \partial_r V_j - r^{-1} \sin\theta \, \partial_\theta V_j, \\
\partial_y V_j = \sin\theta \, \partial_r V_j + r^{-1} \cos\theta \, \partial_\theta V_j,
\end{array}}
\right.
\qquad \mbox{on $D^2\setminus\{0\}$, \qquad for $j=1,2$}.
\end{align}
Meanwhile, the asymptotically radial condition \eqref{eq:ar_condition_for_V} of $V\in \mathrm{GVect}_{\rm ar}(D^2)$ lets us smoothly extend $V_3$ and $V_4$ to a neighborhood of $\ol{D^2}$ in the plane, e.g. by letting $V_3\smallvectwo{r}{\theta} := V_3\smallvectwo{1}{\theta}$ and $V_4\smallvectwo{r}{\theta} := V_4\smallvectwo{1}{\theta}$ for all $r \in [1,1+\epsilon)$ and all $\theta \in \mathbb{R}/2\pi\mathbb{Z}$. Thus each entry of $J_{D^2}(V)$ extends to a smooth function on this neighborhood of $\ol{D^2}$. Hence the integrand of the integral in \eqref{eq:our_beta_formula} is a differential $2$-form on $D^2$ whose coefficient functions smoothly extend to a neighborhood of $\ol{D^2}$ in the plane, hence are bounded functions; thus the integral converges.

\vs

The skew-symmetry is immediate, and the $\mathbb{R}$-bilinearity is obvious by inspection.
\end{proof}

The following computation will find itself crucial.

\begin{lemma}[computation of $\mathrm{G}\beta$]
For each $V, W\in \mathrm{GVect}_{\rm ar}(D^2)$, define the functions $v_3(\theta),v_4(\theta),w_3(\theta),w_4(\theta) \in C^\infty(\mathbb{R}/2\pi\mathbb{Z};\mathbb{R})$ as
\begin{align}
\label{eq:v3_v4_w3_w4}
  v_j(\theta) := V_j|_{r=1}(\theta) = V_j\smallvectwo{1}{\theta}, \quad
  w_j(\theta) := W_j|_{r=1}(\theta) = W_j\smallvectwo{1}{\theta}, \qquad
\mbox{for each ~ $j=3,4$}.
\end{align}
Then
\begin{align}
\label{eq:Gbeta_computation}
\left\{ {\renewcommand{\arraystretch}{1.6} \begin{array}{rcl}
(\mathrm{G}\beta)(V,W) & = & \displaystyle -6 c_0 \int_0^{2\pi} \mathbf{I}_{V,W}(\theta) \, d\theta, \qquad\mbox{where  } \\
 \mathbf{I}_{V,W} & = & (v_4' w_4'' - 2v_4 w_4') + (3v_3 w_3' - v_3' w_4' + v_4' w_3') ~ \in ~ C^\infty(\mathbb{R}/2\pi\mathbb{Z};\mathbb{R}),
\end{array}} \right.
\end{align}
where the prime ${}'$ means the derivative with respect to the $\theta$ variable.
\end{lemma}

\begin{proof}
As in the proof of Lem.\ref{lem:Gbeta_is_well-defined}, one can smoothly extend $J_{D^2}(V)$ and $J_{D^2}(W)$ to an open neighborhood $N$ of $\ol{D^2}$ in the plane $\mathbb{C} \approx \mathbb{R}^2$. Observe first that
\begin{align*}
  \mathrm{tr}_{\mathrm{mat}_2} (dJ_{D^2}(V) \wedge dJ_{D^2}(W))
= \mathrm{tr}_{\mathrm{mat}_2} d( J_{D^2}(V) \, dJ_{D^2}(W)) = d\left( \mathrm{tr}_{\mathrm{mat}_2}(J_{D^2}(V)\, dJ_{D^2}(W))\right)
\end{align*}
holds as smooth differential $2$-forms on $N$, where $\mathrm{tr}_{\mathrm{mat}_2}(J_{D^2}(V)\, dJ_{D^2}(W))$ is a smooth differential $1$-form on $N$. So Stokes' Theorem applies, yielding
$$
\beta(V,W) = - 6c_0 \int_{S^1} \iota_{S^1\to N}^* \left( \mathrm{tr}_{\mathrm{mat}_2}(J_{D^2}(V)\, dJ_{D^2}(W)) \right),
$$
where $\iota_{S^1\to N} : S^1 \hookrightarrow N$ is the usual embedding, as described in Def.\ref{def:relevant_2-dimensional_manifolds}. So it remains to  compute $\iota_{S^1\to N}^* \left( \mathrm{tr}_{\mathrm{mat}_2}(J_{D^2}(V)\, dJ_{D^2}(W)) \right)$; one way to do so is to first write the $1$-form $\mathrm{tr}_{\mathrm{mat}_2}(J_{D^2}(V)\, dJ_{D^2}(W))$ on $N$ as a $C^\infty(N;\mathbb{R})$-linear combination of $dr$ and $d\theta$, and then take only the $d\theta$ term, with its coefficient function restricted to $S^1 \subset N$.

\vs

From \eqref{eq:J_D2_V} we have
\begin{align*}
& \mathrm{tr}_{\mathrm{mat}_2} (J_{D^2}(V) \, dJ_{D^2}(W)) \\
& = \mathrm{tr}_{\mathrm{mat}_2} \mattwo{\partial_x V_1}{\partial_y V_1}{\partial_x V_2}{\partial_y V_2} \left( \mattwo{\partial_r (\partial_x W_1)}{\partial_r(\partial_y W_1)}{\partial_r(\partial_x W_2)}{\partial_r(\partial_y W_2)} dr + \mattwo{\partial_\theta(\partial_x W_1)}{\partial_\theta(\partial_y W_1)}{\partial_\theta(\partial_x W_2)}{\partial_\theta(\partial_y W_2)} d\theta \right) \\
& = (\sim) \, dr + \left( \, (\partial_x V_1) \, \partial_\theta(\partial_x W_1) + (\partial_y V_1) \, \partial_\theta(\partial_x W_2) + (\partial_x V_2) \, \partial_\theta(\partial_y W_1) + (\partial_y V_2) \, \partial_\theta(\partial_y W_2) \right) d\theta
\end{align*}
on $N$, for some expression $(\sim)$. Therefore
\begin{align}
\nonumber
  & \iota_{S^1\to N}^* \left( \mathrm{tr}_{\mathrm{mat}_2}(J_{D^2}(V)\, dJ_{D^2}(W)) \right) \\
\label{eq:V_W_integrand_on_circle_to_compute}
& = \underbrace{ \left(
\begin{array}{l}
(\partial_x V_1)|_{r=1} \, \partial_\theta \left( (\partial_x W_1)|_{r=1} \right)  + (\partial_y V_1)|_{r=1}\, \partial_\theta \left( (\partial_x W_2) |_{r=1} \right)  \\
+ (\partial_x V_2)|_{r=1} \, \partial_\theta \left( (\partial_y W_1)|_{r=1}\right) + (\partial_y V_2)|_{r=1} \, \partial_\theta \left( (\partial_y W_2)|_{r=1} \right)
\end{array}
\right) } d\theta.
\end{align}
Denote the underbraced coefficient function by $\til{\mathbf{I}}_{V,W}(\theta) \in C^\infty(\mathbb{R}/2\pi\mathbb{Z};\mathbb{R})$, so that
$$
(\mathrm{G}\beta)(V,W) = - 6c_0 \int_0^{2\pi} \til{\mathbf{I}}_{V,W}(\theta) \, d\theta.
$$
For convenience, write
\begin{align*}
  c(\theta) := \cos\theta, \qquad s(\theta) = \sin\theta,
\end{align*}
Whenever clear, we will omit the argument $\theta$ of functions in $\theta$. From \eqref{eq:V1_V2_in_V3_V4}, we get
\begin{align}
\label{eq:V1_V2_and_rho_1}
  V_1|_{r=1} = V_3|_{r=1} \cdot \cos\theta - V_4|_{r=1} \cdot \sin\theta = c \, v_3 - s \, v_4, \qquad
  V_2|_{r=1} = s \, v_3 + c \, v_4,
\end{align}
as well as
\begin{align*}
  (\partial_r V_1)|_{r=1} & = (\partial_r V_3)|_{r=1} \cdot c - (\partial_r V_4)|_{r=1} \cdot  s - (V_4|_{r=1}) \cdot s = - s \, v_4, \\
  (\partial_r V_2)|_{r=1} & = (\partial_r V_3)|_{r=1} \cdot s + (\partial_r V_4)|_{r=1} \cdot c + (V_4|_{r=1}) \cdot c = c \, v_4,
\end{align*}
where we used $(\partial_r V_3)|_{r=1} = 0 = (\partial_r V_4)|_{r=1}$ which we get from the condition \eqref{eq:ar_condition_for_V} of $V$. Applying $\partial_\theta$ to \eqref{eq:V1_V2_and_rho_1}, we get
\begin{align*}
  \partial_\theta (V_1|_{r=1}) = - s \, v_3 + c \, v_3' - c \, v_4 - s \, v_4', \qquad
  \partial_\theta (V_2|_{r=1}) = c \, v_3 + s \, v_3' - s \, v_4 + c \, v_4',
\end{align*}
where the prime ${}'$ stands for the $\partial_\theta$-partial derivative. Observe now
\begin{align*}
(\partial_x V_1)|_{r=1} & \stackrel{\eqref{eq:partial_x_y_of_V_1_2_in_polar_partials}}{=} c\, (\partial_r V_1)|_{r=1} - s \, \partial_\theta (V_1|_{r=1})
= - cs \, v_4 - s\, (- s \, v_3 + c \, v_3' - c \, v_4 - s \, v_4') \\
& = s^2 \, v_3 - c s \, v_3' + s^2 \, v_4', \\
(\partial_x V_2)|_{r=1} & \stackrel{\eqref{eq:partial_x_y_of_V_1_2_in_polar_partials}}{=} c\, (\partial_r V_2)|_{r=1} - s \, \partial_\theta (V_2|_{r=1}) = c^2 \, v_4 - s (c \, v_3 + s \, v_3' - s \, v_4 + c \, v_4') \\
& = -cs \, v_3 - s^2 \, v_3' + v_4 - cs \, v_4', \\
(\partial_y V_1)|_{r=1} & \stackrel{\eqref{eq:partial_x_y_of_V_1_2_in_polar_partials}}{=} s\, (\partial_r V_1)|_{r=1} + c \, \partial_\theta(V_1|_{r=1}) = -s^2 \, v_4 + c (- s \, v_3 + c \, v_3' - c \, v_4 - s \, v_4') \\
& = - cs \, v_3 + c^2 \, v_3' - v_4 - cs \, v_4', \\
(\partial_y V_2)|_{r=1} & \stackrel{\eqref{eq:partial_x_y_of_V_1_2_in_polar_partials}}{=} s\, (\partial_r V_2)|_{r=1} + c \, \partial_\theta(V_2|_{r=1}) = cs\, v_4 + c (c \, v_3 + s \, v_3' - s \, v_4 + c \, v_4') \\
& = c^2 \, v_3 + cs\, v_3' + c^2 \, v_4',
\end{align*}
where we used $c^2 + s^2 =1$. Likewise for $W$; taking the $\partial_\theta$-partial derivatives, using
$$
(s^2)' = 2cs, \qquad (c^2)' = - 2cs, \qquad (cs)' = c^2 - s^2
$$
we get
\begin{align*}
  \partial_\theta ( (\partial_x W_1)|_{r=1}) & = 2cs \, w_3 + (-c^2+2s^2) \, w_3' - cs\, w_3'' + 2cs \, w_4' + s^2 \, w_4'', \\
  \partial_\theta ( (\partial_x W_2)|_{r=1}) & = (-c^2+s^2) \, w_3 - 3 cs \, w_3' - s^2 \, w_3'' + (1-c^2+s^2) \, w_4' - cs \, w_4'', \\
  \partial_\theta ( (\partial_y W_1)|_{r=1}) & = (-c^2+s^2) \, w_3 - 3 cs \, w_3' + c^2 \, w_3'' + (-1-c^2+s^2) \, w_4' - cs \, w_4'', \\
  \partial_\theta ( (\partial_y W_2)|_{r=1}) & = -2cs\, w_3 + (2c^2-s^2) \, w_3' + cs \, w_3'' - 2cs \, w_4' + c^2 \, w_4''.
\end{align*}
It is now a straightforward task to compute $\til{\mathbf{I}}_{V,W}$, i.e. the underbraced part in \eqref{eq:V_W_integrand_on_circle_to_compute}. We collect the coefficient of each monomial in $v_j,v_j',w_j,w_j',w_j''$ (for $j=3,4$): (note there is no $w_4$)
\begin{align*}
  v_3 w_3 & : s^2(2cs) -cs(-c^2+s^2) -cs(-c^2+s^2) + c^2 (-2cs)
= 0, \\
  v_3 w_3' & : s^2(-c^2+2s^2)-cs(-3cs) -cs(-3cs) + c^2(2c^2-s^2)
= 2(s^2+c^2)^2 = 2, \\
  v_3 w_3'' & : s^2(-cs) - cs(-s^2) - cs(c^2) +c^2(cs) = 0, \\
  v_3 w_4' & : s^2(2cs) - cs(1-c^2+s^2) - cs(-1-c^2+s^2)+c^2(-2cs)
= 0, \\
  v_3 w_4'' & : s^2(s^2) - cs(-cs) - cs(-cs) + c^2(c^2) = (c^2+s^2)^2 = 1, \\
  v_3' w_3 & : -cs(2cs) +c^2(-c^2+s^2) - s^2(-c^2+s^2) + cs(-2cs)
= -(c^2+s^2)^2 = -1, \\
  v_3' w_3' & : -cs(-c^2+2s^2) + c^2(-3cs) - s^2(-3cs) + cs(2c^2-s^2)
= 0, \\
  v_3' w_3'' & : -cs(-cs) + c^2(-s^2) - s^2(c^2) + cs(cs) = 0, \\
  v_3' w_4' & : -cs(2cs) + c^2(1-c^2+s^2) - s^2(-1-c^2+s^2) + cs(-2cs)
= (c^2+s^2) - (c^2+s^2)^2 = 0, \\
  v_3' w_4'' & : -cs(s^2) + c^2(-cs) - s^2(-cs) + cs(c^2) =0, \\
  v_4 w_3 & : -(-c^2+s^2) + (-c^2+s^2)=0, \\
  v_4 w_3' & : - (-3cs) + (-3cs) = 0, \\
  v_4 w_3'' & : - (-s^2) + (c^2) = 1, \\
  v_4 w_4' & : - (1-c^2+s^2) + (-1-c^2+s^2) = -2, \\
  v_4 w_4'' & : - (-cs) + (-cs) = 0, \\
  v_4' w_3 & : s^2(2cs) - cs(-c^2+s^2) - cs(-c^2+s^2) + c^2(-2cs) = 0, \\
  v_4' w_3' & : s^2(-c^2+2s^2) - cs(-3cs) - cs(-3cs) + c^2(2c^2-s^2)
= 2(c^2+s^2)^2 = 2, \\
  v_4' w_3'' & : s^2(-cs) - cs(-s^2) - cs(c^2) + c^2(cs) = 0, \\
  v_4' w_4' & : s^2(2cs) - cs(1-c^2+s^2) - cs(-1-c^2+s^2) + c^2(-2cs) =0,\\
  v_4' w_4'' & : s^2(s^2) - cs(-cs) - cs(-cs) + c^2(c^2) = (c^2+s^2)^2 = 1.
\end{align*}
So we have
$$
\til{\mathbf{I}}_{V,W}
= (v_4' w_4'' - 2 v_4 w_4') + (2v_3 w_3' + v_3 w_4'' - v_3' w_3 + v_4 w_3'' + 2 v_4' w_3'),
$$
which is slightly different from the sought-for $\mathbf{I}_{V,W}(\theta)$. To finish the proof, one uses integration by parts when integrating the $1$-form $\til{\mathbf{I}}_{V,W}(\theta) \, d\theta$ over $S^1 \approx \mathbb{R}/2\pi\mathbb{Z}$, e.g. $\int_0^{2\pi} v_3 w_4'' d\theta = - \int_0^{2\pi} v_3' w_4' d\theta$; this is possible because derivatives of any order of $v_3,v_4,w_3,w_4$ are well-defined $\mathbb{R}$-valued continuous functions on $\mathbb{R}/2\pi\mathbb{Z}$, i.e. $2\pi$-periodic.  
\end{proof}

\begin{corollary}
\label{cor:Gbeta_zero_when_one_is_restricted}
One has
$$
(\mathrm{G}\beta)(V,W) = 0, \qquad \forall V \in \mathrm{GVect}_{\rm ar}(D^2), \quad \forall W \in \mathrm{Vect}_{\rm az}(D^2).
$$
\end{corollary}

\begin{proof}
From \eqref{eq:az_condition_for_V} and \eqref{eq:v3_v4_w3_w4}, we see that $W\in \mathrm{Vect}_{\rm az}(D^2)$ means $w_3 = w_4 = 0$.
\end{proof}

By the $\mathbb{R}$-bilinearity and the skew-symmetry, this implies:
\begin{corollary}
\label{cor:Gbeta_descends}
$\mathrm{G}\beta$ descends to 
$$
\ol{\mathrm{G}\beta} ~ : ~ \mathrm{GVect}_{\rm ar}(D^2)/\mathrm{Vect}_{\rm az}(D^2) \times \mathrm{GVect}_{\rm ar}(D^2)/\mathrm{Vect}_{\rm az}(D^2) ~ \to ~ \mathbb{R},
$$
defined by
$$
(\ol{\mathrm{G}\beta})\left( V + \mathrm{Vect}_{\rm az}(D^2), ~ W + \mathrm{Vect}_{\rm az}(D^2) \right) := (\mathrm{G}\beta)(V,W), \qquad \forall V,W \in \mathrm{GVect}_{\rm ar}(D^2). \qed
$$
\end{corollary}

Consider the restriction of $\ol{\mathrm{G}\beta}$ to $\mathrm{Vect}_{\rm ar}(D^2)/\mathrm{Vect}_{\rm az}(D^2) \times \mathrm{Vect}_{\rm ar}(D^2)/\mathrm{Vect}_{\rm az}(D^2)$:
$$
(\ol{\mathrm{G}\beta}) \left( V + \mathrm{Vect}_{\rm az}(D^2), ~ W + \mathrm{Vect}_{\rm az}(D^2) \right) = - 6c_0 \int_0^{2\pi} v_4'(\theta) \, w_4''(\theta) \, d\theta, \quad
\forall V,W \in \mathrm{Vect}_{\rm ar}(D^2).
$$
One finds that this is a Lie algebra $2$-cocycle of $\mathrm{Vect}_{\rm ar}(D^2)/\mathrm{Vect}_{\rm az}(D^2) \stackrel{\eqref{eq:Vects_as_quotient}}{\cong} \mathrm{Vect}(S^1)$, being cohomologous to $-144\pi {\rm i} c_0$ times the famous \emph{Gelfand-Fuchs cocycle} $\mathbf{GF} : \mathrm{Vect}(S^1) \times \mathrm{Vect}(S^1) \to \mathbb{R}$, which is given by
\begin{align}
\label{eq:GF}
{\bf GF}\left(v(\theta) \frac{\partial}{\partial \theta}, \, w(\theta)\frac{\partial}{\partial\theta} \right) = \frac{1}{24\pi {\rm i}} \int_0^{2\pi} v'(\theta) \, w''(\theta) \, d\theta,
\end{align}
where the prime ${}'$ means the derivative with respect to the angle variable $\theta$.

\vs

We finally show:

\begin{proposition}
\label{prop:Gbeta_is_Lie_algebra_2-cocycle}
$\mathrm{G}\beta$ is a Lie algebra $2$-cocycle of $\mathrm{GVect}_{\rm ar}(D^2)$.
\end{proposition}

\begin{proof}
Let $V^{(0)},V^{(1)},V^{(2)} \in \mathrm{GVect}_{\rm ar}(D^2)$. For each $j \in \mathbb{Z}/3\mathbb{Z} = \{0,1,2\}$, denote by $V^{(j \,j+1)}:=[V^{(j)},V^{(j+1)}]$. Like in \eqref{eq:v3_v4_w3_w4}, we define the following elements in $C^\infty(\mathbb{R}/2\pi\mathbb{Z};\mathbb{R})$:
$$
v^{(j)}_k := V^{(j)}_k|_{r=1}, \qquad
v^{(j \, j+1)}_k := V^{(j\, j+1)}_k|_{r=1}, \qquad
\forall j\in \mathbb{Z}/3\mathbb{Z}, \quad \forall k=3,4.
$$
To express $v^{(j \, j+1)}_k$ we use Prop.\ref{prop:semidirect_Vect_S1}, which tells us
$$
v^{(j \, j+1)}_4 = v^{(j)}_4 \, (v^{(j+1)}_4)' - (v^{(j)}_4)' \, v^{(j+1)}_4, \qquad
v^{(j\, j+1)}_3 = v^{(j)}_4 \, (v^{(j+1)}_3)' - v^{(j+1)}_4 \, (v^{(j)}_3)',
$$
for each $j\in \mathbb{Z}/3\mathbb{Z}$. 

\vs

Using integration by parts, one can modify $\mathbf{I}_{V,W}$ in \eqref{eq:Gbeta_computation} to
$$
(v_4'w_4'' - 2v_4 w_4') + \wh{\mathbf{I}}_{V,W}, \qquad\mbox{where} \qquad \wh{\mathbf{I}}_{V,W} := 3v_3 w_3' + v_3 w_4'' + v_4' w_3'.
$$
Observe that
\begin{align*}
& \wh{\mathbf{I}}_{[V^{(j)},V^{(j+1)}], \, V^{(j+2)}} = \wh{\mathbf{I}}_{V^{(j\, j+1)},V^{(j+2)}} \\
& = 
3 v^{(j\, j+1)}_3 \, (v^{(j+2)}_3)' + v^{(j\, j+1)}_3 \, (v^{(j+2)}_4)'' + (v^{(j\, j+1)}_4)' \, (v^{(j+2)}_3)' \\
& = 3 v^{(j)}_4 \, (v^{(j+1)}_3)' \, (v^{(j+2)}_3)' - 3 v^{(j+1)}_4 \, (v^{(j)}_3)' \, (v^{(j+2)}_3)' 
+ v^{(j)}_4 \, (v^{(j+1)}_3)' \, (v^{(j+2)}_4)'' - v^{(j+1)}_4 \, (v^{(j)}_3)' \, (v^{(j+2)}_4)'' \\
& \quad + \cancel{ (v^{(j)}_4)' \, (v^{(j+1)}_4)' \, (v^{(j+2)}_3)'} + v^{(j)}_4 \, (v^{(j+1)}_4)'' \, (v^{(j+2)}_3)' - (v^{(j)}_4)'' \, v^{(j+1)}_4 \, (v^{(j+2)}_3)' - \cancel{ (v^{(j)}_4)' \, (v^{(j+1)}_4)' \, (v^{(j+2)}_3)' } \\
& = \underbrace{ 3 v^{(j)}_4 \, (v^{(j+1)}_3)' \, (v^{(j+2)}_3)'}_{\cone} 
- \underbrace{ 3 v^{(j+1)}_4 \, (v^{(j+2)}_3)'\, (v^{(j)}_3)' }_{\ctwo} 
+ \underbrace{ v^{(j)}_4 \, (v^{(j+2)}_4)'' \, (v^{(j+1)}_3)' }_{\cthree}
- \underbrace{ v^{(j+1)}_4 \, (v^{(j+2)}_4)'' \, (v^{(j)}_3)' }_{\cfour} \\
& \quad  + \underbrace{ v^{(j)}_4 \, (v^{(j+1)}_4)'' \, (v^{(j+2)}_3)' }_{\cfive} 
- \underbrace{ v^{(j+1)}_4 \, (v^{(j)}_4)'' \, (v^{(j+2)}_3)' }_{\csix}.
\end{align*}
We shall take the cyclic sum
\begin{align}
  \label{eq:cyclic_sum_wh_bf_I}
  \sum_{j\in \mathbb{Z}/3\mathbb{Z}} \wh{\mathbf{I}}_{[V^{(j)},V^{(j+1)}],\, V^{(j+2)}}.
\end{align}
One can see that the cyclic sum of $\cone$ cancels with that of $\ctwo$, the cyclic sum of $\cthree$ cancels with that of $\csix$, and the cyclic sum of $\cfour$ cancels with that of $\cfive$. So the sum \eqref{eq:cyclic_sum_wh_bf_I} is zero.

\vs

Note that we did not deal with the part $v_4' w_4'' - 2v_4 w_4'$ of $\mathbf{I}_{V,W}$ \eqref{eq:Gbeta_computation}. One can prove, by a similar computation as just done, that the corresponding cyclic sum for this part also vanishes; we skipped this because it is well known that $(v_4(\theta) \frac{\partial}{\partial \theta},w_4(\theta) \frac{\partial}{\partial\theta}) \mapsto \int_0^{2\pi} (v_4' w_4'' - 2v_4 w_4')d\theta$ is a Lie algebra $2$-cocycle of $\mathrm{Vect}(S^1)$. Hence
$
\sum_{j\in \mathbb{Z}/3\mathbb{Z}} \mathbf{I}_{[V^{(j)},V^{(j+1)}],\, V^{(j+2)}} = 0,
$
so 
$$
\sum_{j\in \mathbb{Z}/3\mathbb{Z}} (\mathrm{G}\beta)([V^{(j)},V^{(j+1)}],\, V^{(j+2)}) = 0. 
$$
\end{proof}

\vs

Denote by $\wh{\mathrm{GVect}_{\rm ar}}(D^2)$ be the corresponding Lie algebra ($1$-dimensional) central extension of $\mathrm{GVect}_{\rm ar}(D^2)$. Corollaries \ref{cor:Gbeta_zero_when_one_is_restricted} and \ref{cor:Gbeta_descends} lets us `factor out' $\mathrm{Vect}_{\rm az}(D^2)$ from this central extension. 

\vs

So, in view of Prop.\ref{prop:semidirect_Vect_S1}, the factor $\wh{\mathrm{GVect}_{\rm ar}}(D^2)/\mathrm{Vect}_{\rm az}(D^2)$ effectively realizes a central extension of $\mathrm{Vect}(S^1) \ltimes (\mathrm{Vect}(S^1))_{\rm ab}$ by $\mathbb{R}$. We note that this extended central $\mathbb{R}$ is spread into the two factors $\mathrm{Vect}(S^1)$ and $(\mathrm{Vect}(S^1))_{\rm ab}$, instead of being associated to only one factor. So the central extension we obtained is neither
$$
\mathrm{Vect}(S^1)^\wedge \ltimes (\mathrm{Vect}(S^1))_{\rm ab} \quad\mbox{nor}\quad
\mathrm{Vect}(S^1) \ltimes (\mathrm{Vect}(S^1))_{\rm ab}^\wedge,
$$
where $\mathrm{Vect}(S^1)^\wedge = \wh{\mathrm{Lie}(\mathbf{G})} = \mathrm{Lie}(\wh{\mathbf{G}})$, a central extension of $\mathrm{Vect}(S^1)$ associated to a $2$-cocycle cohomologous to certain multiple of the Gelfand-Fuchs cocycle, is identified with the Lie algebra of the Virasoro-Bott group $\wh{\mathbf{G}}$ constructed in the present paper, i.e. can be considered as the `real' \emph{Virasoro algebra}, while $(\mathrm{Vect}(S^1))_{\rm ab}^\wedge$ is a central extension of the abelian Lie algebra $(\mathrm{Vect}(S^1))_{\rm ab}$ associated to a $2$-cocycle given by some multiple of
\begin{align}
\label{eq:Heisenberg_bracket}
(v(\theta) \frac{\partial}{\partial \theta}, \, w(\theta) \frac{\partial}{\partial \theta}) \mapsto \int_0^{2\pi} v'(\theta) w(\theta) d\theta,
\end{align}
hence can be thought of as a `real' \emph{Heisenberg algebra}, giving us a motivation to denote it by
$$
\mathrm{Heis}(S^1) := (\mathrm{Vect}(S^1))_{\rm ab}^\wedge.
$$

\vs

\subsection{Identification of the generalized Lie algebra of vector fields}

Recall that $\mathrm{G}\beta$ is an $\mathbb{R}$-valued Lie algebra $2$-cocycle of the (infinite dimensional) real Lie algebra $\mathrm{GVect}_{\rm ar}(D^2)$, where each element $V$ of $\mathrm{GVect}_{\rm ar}(D^2)$ is of the form $V = V_1\smallvectwo{x}{y} \frac{\partial}{\partial x} + V_2\smallvectwo{x}{y} \frac{\partial}{\partial y}$ for $V_1,V_2 \in C^\infty(D^2;\mathbb{R})$, or equivalently $V = V_3\smallvectwo{r}{\theta} \frac{\partial}{\partial r} + V_4\smallvectwo{r}{\theta} \frac{\partial}{\partial \theta}$ for $V_3,V_4 \in C^\infty( (0,1) \times \mathbb{R}/2\pi\mathbb{Z}; \mathbb{C})$ related to $V_1,V_2$ by the formula \eqref{eq:V3_and_V4_in_V1_and_V2}, satisfying the `asymptotically radial' condition \eqref{eq:ar_condition_for_V}, namely $\partial_r V_3 \equiv 0 \equiv \partial_r V_4$ on some neighborhood of $S^1$ in $\ol{D^2}$. We now allow $V_1$ and $V_2$ to be complex-valued, still requiring the same asymptotically radial condition.
\begin{definition}
A \emph{complexified generalized smooth real disc vector field $V$} is an expression \eqref{eq:disc_vector_field_V} with $V_1,V_2 \in C^\infty(D^2;\mathbb{C})$. Denote by $\mathrm{GVect}(D^2)_\mathbb{C}$ be the set of all complexified generalized smooth real disc vector fields. For each $V$, define $V_3,V_4 \in C^\infty( (0,1) \times \mathbb{R}/2\pi\mathbb{Z}; \mathbb{C})$ using the same formula \eqref{eq:V3_and_V4_in_V1_and_V2} as before. Define the notions of asymptotically radial and asymptotically zero by using the same conditions \eqref{eq:ar_condition_for_V} and \eqref{eq:az_condition_for_V} as before, and denote the collection of such elements of $\mathrm{GVect}(D^2)_\mathbb{C}$ by $\mathrm{GVect}_{\rm ar}(D^2)_\mathbb{C}$ and $\mathrm{GVect}_{\rm az}(D^2)$ respectively.

\vs

We say that a complexified generalized smooth real disc vector field $V$ \eqref{eq:disc_vector_field_V} is \emph{genuine} if the component functions $V_1$ and $V_2$ satisfy the condition \eqref{eq:tangent_condition}. Denote the set of all genuine elements of $\mathrm{GVect}(D^2)_\mathbb{C}$ by $\mathrm{Vect}(D^2)_\mathbb{C}$. The set of all elements of $\mathrm{Vect}(D^2)_\mathbb{C}$ that are asymptotically radial (resp. asymptotically zero) is denoted by $\mathrm{Vect}_{\rm ar}(D^2)_\mathbb{C}$ (resp. by $\mathrm{Vect}_{\rm az}(D^2)_\mathbb{C})$.

\vs

Endow Lie bracket structures to these sets by the formula \eqref{eq:V_c_polar_components}.
\end{definition}

\begin{definition}
Define $\mathrm{Vect}(S^1)_\mathbb{C}$ as the set of all complexified smooth real vector fields $v = v(\theta) \frac{\partial}{\partial \theta}$ with $v(\theta) \in C^\infty(\mathbb{R}/2\pi\mathbb{Z};\mathbb{C})$, endowed with the Lie bracket \eqref{eq:Vect_S1_Lie_bracket}.
\end{definition}

\begin{lemma}
The sets just defined are well-defined (infinite dimensional) Lie algebras over $\mathbb{C}$. The set $\mathrm{Vect}_{\rm az}(D^2)_\mathbb{C}$ is an ideal in $\mathrm{GVect}_{\rm ar}(D^2)_\mathbb{C}$, hence also in $\mathrm{Vect}_{\rm az}(D^2)_\mathbb{C}$. One has the natural isomorphism of complex Lie algebras
\begin{align}
\label{qe:Gar_mod_Vaz_isomorphism_complex}
\mathrm{Vect}_{\rm ar}(D^2)_\mathbb{C} / \mathrm{Vect}_{\rm az}(D^2)_\mathbb{C} \cong \mathrm{Vect}(S^1)_\mathbb{C}
\end{align}
given by the restriction map whose formula is in \eqref{eq:vector_field_restriction}. \qed
\end{lemma}

\begin{proposition}
One has a canonical isomorphism of complex Lie algebras
\begin{align*}
  \mathrm{GVect}_{\rm ar}(D^2)_\mathbb{C} / \mathrm{Vect}_{\rm az}(D^2)_\mathbb{C} \cong \mathrm{Vect}(S^1)_\mathbb{C} \ltimes (\mathrm{Vect}(S^1)_\mathbb{C})_{\rm ab},
\end{align*}
where $(\mathrm{Vect}(S^1))_{\rm ab}$ denotes the abelian complex Lie algebra with the underlying $\mathbb{C}$-vector space being $\mathrm{Vect}(S^1)_\mathbb{C}$, with the action of $\mathrm{Vect}(S^1)_\mathbb{C}$ on $((\mathrm{Vect}(S^1)_\mathbb{C})_{\rm ab}$ given by the same formula as in \eqref{eq:Vect_S1_action_on_itself}. \qed
\end{proposition}

\begin{proposition}
For each $V,W\in \mathrm{GVect}_{\rm ar}(D^2)_\mathbb{C}$ define the number $(\mathrm{G}\beta_\mathbb{C})(V,W) \in \mathbb{C}$ as the right-hand-side of \eqref{eq:our_beta_formula}, with the help of \eqref{eq:J_D2_V}. Then $\mathrm{G}\beta_\mathbb{C}$ is a (complex) Lie algebra $2$-cocycle of $\mathrm{GVect}_{\rm ar}(D^2)_\mathbb{C}$. \qed
\end{proposition}

\begin{lemma}
One has
$$
(\mathrm{G}\beta_\mathbb{C})(V,W)= 0, \qquad \forall V\in \mathrm{GVect}_{\rm ar}(D^2)_\mathbb{C}, \quad \forall W \in \mathrm{Vect}_{\rm az}(D^2)_\mathbb{C}.
$$
So $\mathrm{G}\beta_\mathbb{C}$ descends to a complex Lie algebra $2$-cocycle on $\mathrm{GVect}_{\rm ar}(D^2)_\mathbb{C}$
$$
\ol{\mathrm{G}\beta_\mathbb{C}} : \mathrm{GVect}_{\rm ar}(D^2)_\mathbb{C} / \mathrm{Vect}_{\rm az}(D^2)_\mathbb{C} \times \mathrm{GVect}_{\rm ar}(D^2)_\mathbb{C} / \mathrm{Vect}_{\rm az}(D^2)_\mathbb{C} \to \mathbb{C},
$$
defined by
$$
(\ol{\mathrm{G}\beta_\mathbb{C}})(V+\mathrm{Vect}_{\rm az}(D^2)_\mathbb{C}, \, W + \mathrm{Vect}_{\rm az}(D^2)_\mathbb{C}) := (\mathrm{G}\beta_\mathbb{C})(V,W), ~~ \forall V,W\in \mathrm{GVect}_{\rm ar}(D^2)_\mathbb{C}.
$$
The restriction of $\ol{\mathrm{G}\beta_\mathbb{C}}$ to $\mathrm{Vect}_{\rm ar}(D^2)_\mathbb{C} / \mathrm{Vect}_{\rm az}(D^2)_\mathbb{C} \times \mathrm{Vect}_{\rm ar}(D^2)_\mathbb{C} / \mathrm{Vect}_{\rm az}(D^2)_\mathbb{C}$ is
$$
(\ol{\mathrm{G}\beta_\mathbb{C}})(V+\mathrm{Vect}_{\rm az}(D^2)_\mathbb{C}, \, W+\mathrm{Vect}_{\rm az}(D^2)_\mathbb{C}) = - 6c_0 \int_0^{2\pi} v_4'(\theta) \, w_4''(\theta) d\theta
$$
for each $V,W\in \mathrm{Vect}_{\rm ar}(D^2)_\mathbb{C}$, where $v_4(\theta)$ and $w_4(\theta)$ are defined by \eqref{eq:v3_v4_w3_w4}. \qed
\end{lemma}

\begin{definition}
Denote by $\wh{\mathrm{GVect}_{\rm ar}}(D^2)_\mathbb{C}$ the central extension of $\mathrm{GVect}_{\rm ar}(D^2)_\mathbb{C}$ by $\mathbb{C}$ corresponding to the $2$-cocycle $\mathrm{G}\beta_\mathbb{C}$. Denote by $(\mathrm{Vect}(S^1)_\mathbb{C})^\wedge$ the central extension of $\mathrm{Vect}(S^1)_\mathbb{C}$ by $\mathbb{C}$ corresponding to the $2$-cocycle induced by a restriction of $\ol{\mathrm{G}\beta_\mathbb{C}}$.
\end{definition}
\begin{remark}
This complex Lie algebra extension $(\mathrm{Vect}(S^1)_\mathbb{C})^\wedge$ is (`equivalent to') what is usually referred to as the \emph{Virasoro algebra}.
\end{remark}

Consider the central extension by $\mathbb{C}$ of the complex Lie algebra $\mathrm{GVect}_{\rm ar}(D^2)_\mathbb{C}/\mathrm{Vect}_{\rm az}(D^2)_\mathbb{C} \stackrel{\eqref{qe:Gar_mod_Vaz_isomorphism_complex}}{\cong} \mathrm{Vect}(S^1)_\mathbb{C} \ltimes (\mathrm{Vect}(S^1)_\mathbb{C})_{\rm ab}$ corresponding to the $2$-cocycle $\ol{\mathrm{G}\beta_\mathbb{C}}$. As in the real case, this central extension is neither 
$$
(\mathrm{Vect}(S^1)_\mathbb{C})^\wedge \ltimes (\mathrm{Vect}(S^1)_\mathbb{C})_{\rm ab} \quad\mbox{nor}\quad
\mathrm{Vect}(S^1)_\mathbb{C} \ltimes \mathrm{Heis}(S^1)_\mathbb{C},
$$
where $\mathrm{Heis}(S^1)_\mathbb{C}$ is the central extension by $\mathbb{C}$ of $(\mathrm{Vect}(S^1)_\mathbb{C})_{\rm ab}$ associated to the $2$-cocycle defined as the natural complexification of the formula \eqref{eq:Heisenberg_bracket}. In order to investigate this central extension we just constructed, it is convenient to present it as generators and relations.
\begin{definition}
Write the above mentioned central extension as
$$
\left( \mathrm{Vect}(S^1)_\mathbb{C} \ltimes (\mathrm{Vect}(S^1)_\mathbb{C})_{\rm ab} \right) \, \times \, \mathbb{C}
$$
as a $\mathbb{C}$-vector space, with the Lie bracket being
$$
[(X_1, a_1), \, (X_2,a_2)] = ([X_1,X_2], \, \ol{\mathrm{G}\beta_\mathbb{C}}(X_1,X_2))
$$
for all $X_1,X_2\in \mathrm{Vect}(S^1)_\mathbb{C} \ltimes (\mathrm{Vect}(S^1)_\mathbb{C})_{\rm ab}$ and $a_1,a_2 \in \mathbb{C}$, where $\ol{\mathrm{G}\beta_\mathbb{C}}$ denotes, by a slight abuse of notation, the $2$-cocycle on $\mathrm{Vect}(S^1)_\mathbb{C} \ltimes (\mathrm{Vect}(S^1)_\mathbb{C})_{\rm ab}$ induced via the isomorphism \eqref{qe:Gar_mod_Vaz_isomorphism_complex} by the $2$-cocycle $\ol{\mathrm{G}\beta_\mathbb{C}}$ on $\mathrm{GVect}_{\rm ar}(D^2)_\mathbb{C}/\mathrm{Vect}_{\rm az}(D^2)_\mathbb{C}$.

\vs

Denote by $\mathbf{c}$ the central generator $(\mathbf{0},1)$, where $\mathbf{0}$ is the zero element of $\mathrm{Vect}(S^1)_\mathbb{C} \ltimes (\mathrm{Vect}(S^1)_\mathbb{C})_{\rm ab}$.

\vs

For each $n\in \mathbb{Z}$ define elements $\wh{L}_n$ and $\wh{J}_n$ of this central extension as
$$
\wh{L}_n := \left( ({\rm i} e^{{\rm i}n\theta} \frac{\partial}{\partial \theta}, \, \mathbf{0}) , ~ 0 \right), \qquad
\wh{J}_n := \left( (\mathbf{0}, \, {\rm i} e^{{\rm i}n\theta} \frac{\partial}{\partial \theta}) , ~ 0 \right),
$$
where each of the two $\mathbf{0}$'s denote the zero element of $\mathrm{Vect}(S^1)_\mathbb{C}$ or $(\mathrm{Vect}(S^1)_\mathbb{C})_{\rm ab}$.
\end{definition}

\begin{lemma}
\label{lem:our_Lie_algebra_presentation}
For each $n,m\in \mathbb{Z}$ one has
\begin{align*}
[\wh{L}_n, \, \wh{L}_m] & = (n-m) \, \wh{L}_{n+m} - 12\pi c_0 \, {\rm i} \, (n^3-2n) \, \delta_{n+m,0} \, \mathbf{c}, \\
[\wh{L}_n, \, \wh{J}_m] & = -m \, \wh{J}_{n+m} + 12\pi c_0 \,{\rm i} \, n^2 \, \delta_{n+m,0} \, \mathbf{c},  \\
[\wh{J}_n, \, \wh{J}_m] & = - 12\pi c_0 \, {\rm i} \, (3n) \, \delta_{n+m,0} \, \mathbf{c}.
\end{align*}
\end{lemma}
\begin{proof}
We first compute the Lie bracket of the `polynomial vector fields' $e^{in\theta} \frac{\partial}{\partial\theta}$ in $\mathrm{Vect}(S^1)_\mathbb{C}$
\begin{align*}
\left[{\rm i} \,e^{{\rm i}n\theta} \frac{\partial}{\partial \theta}, \,\,\, {\rm i}\, e^{{\rm i}m\theta} \frac{\partial}{\partial \theta}\right]_{\mathrm{Vect}(S^1)_\mathbb{C}} \stackrel{\eqref{eq:Vect_S1_Lie_bracket}}{=} - ( \,  e^{{\rm i}n\theta} ({\rm i}m) e^{{\rm i}m\theta} - ({\rm i}n) e^{{\rm i}n\theta} e^{{\rm i}m\theta}) \frac{\partial}{\partial\theta} = (n-m) \, {\rm i} \, e^{{\rm i}(n+m)\theta} \frac{\partial}{\partial\theta}.
\end{align*}
The action of $\mathrm{Vect}(S^1)_\mathbb{C}$ on $(\mathrm{Vect}(S^1))_{\rm ab}$ used in the construction of the semidirect product $\mathrm{Vect}(S^1)_\mathbb{C} \ltimes (\mathrm{Vect}(S^1)_\mathbb{C})_{\rm ab}$ is, from the formula \eqref{eq:Vect_S1_action_on_itself},
$$
\left( {\rm i} \,e^{{\rm i} n\theta} \frac{\partial}{\partial \theta} \right). \left( {\rm i} \,e^{{\rm i} m\theta} \frac{\partial}{\partial \theta} \right) = - m\, {\rm i} \, e^{{\rm i}(n+m)\theta} \frac{\partial}{\partial \theta}.
$$
From Def.\ref{def:semidirect_product_of_Lie_algebras} we compute the following Lie brackets in $\mathrm{Vect}(S^1)_\mathbb{C} \ltimes (\mathrm{Vect}(S^1)_\mathbb{C})_{\rm ab}$:
\begin{align*}
\left[ ({\rm i} e^{{\rm i} n\theta} \frac{\partial}{\partial \theta}, \, \mathbf{0}), \,\,\, ({\rm i} e^{{\rm i} m\theta} \frac{\partial}{\partial \theta}, \, \mathbf{0}) \right]_{\mathrm{Vect}(S^1)_\mathbb{C} \ltimes (\mathrm{Vect}(S^1)_\mathbb{C})_{\rm ab}} & = ( (n-m) \, {\rm i} \, e^{{\rm i} (n+m)\theta} \frac{\partial}{\partial \theta}, \, \mathbf{0}), \\
\left[ ({\rm i} e^{{\rm i} n\theta} \frac{\partial}{\partial \theta}, \, \mathbf{0}), \,\,\, (\mathbf{0}, \, {\rm i} e^{{\rm i} m\theta} \frac{\partial}{\partial \theta}) \right]_{\mathrm{Vect}(S^1)_\mathbb{C} \ltimes (\mathrm{Vect}(S^1)_\mathbb{C})_{\rm ab}} & = ( \mathbf{0}, \, - m \, {\rm i} \, e^{{\rm i} (n+m)\theta} \frac{\partial}{\partial\theta}  ), \\
\left[ (\mathbf{0}, \, {\rm i} e^{{\rm i} n\theta} \frac{\partial}{\partial \theta}), \,\,\, (\mathbf{0}, \, {\rm i} e^{{\rm i} m\theta} \frac{\partial}{\partial \theta}) \right]_{\mathrm{Vect}(S^1)_\mathbb{C} \ltimes (\mathrm{Vect}(S^1)_\mathbb{C})_{\rm ab}} & = ( \mathbf{0}, \, \mathbf{0}  ).
\end{align*}

\vs

Using the formula \eqref{eq:Gbeta_computation} and
$$
\int_0^{2\pi} e^{{\rm i} n\theta} d\theta = 2\pi \, \delta_{n,0},
$$
we now compute $\ol{\mathrm{G}\beta_\mathbb{C}}$:
\begin{align*}
\ol{\mathrm{G}\beta_\mathbb{C}}\left( ({\rm i} e^{{\rm i}n\theta} \frac{\partial}{\partial \theta}, \, \mathbf{0}), \, ({\rm i} e^{{\rm i}m\theta} \frac{\partial}{\partial \theta}, \, \mathbf{0}) \right) & = 6c_0 \int_0^{2\pi} ( ({\rm i}n)({\rm i}m)^2 - 2 ({\rm i}m) )e^{{\rm i}(n+m)\theta} d\theta \\ 
& = - 12\pi c_0 \, {\rm i} \, (nm^2 + 2 m) \, \delta_{n+m,0}
= - 12\pi c_0 \, {\rm i} \, (n^3 - 2n) \, \delta_{n+m,0}, \\
\ol{\mathrm{G}\beta_\mathbb{C}}\left( ({\rm i} e^{{\rm i}n\theta} \frac{\partial}{\partial \theta}, \, \mathbf{0}), \, (\mathbf{0}, \, {\rm i} e^{im\theta} \frac{\partial}{\partial \theta} ) \right) & = 6c_0 \int_0^{2\pi} ({\rm i}n)({\rm i}m) e^{{\rm i}(n+m)\theta} d\theta \\ 
& = - 12\pi c_0 \, {\rm i} \, nm \, \delta_{n+m,0}
= 12\pi c_0 \, {\rm i} \, n^2 \, \delta_{n+m,0}, \\
\ol{\mathrm{G}\beta_\mathbb{C}}\left( (\mathbf{0}, \, {\rm i} e^{{\rm i} n\theta} \frac{\partial}{\partial \theta}), \, (\mathbf{0}, \, {\rm i} e^{{\rm i} m\theta} \frac{\partial}{\partial \theta} ) \right) & = 6c_0 \int_0^{2\pi} 3 ({\rm i} m) e^{{\rm i}(n+m)\theta} d\theta \\ 
& = 12\pi c_0 \, {\rm i} \, (3m) \, \delta_{n+m,0}
= -12\pi c_0 \, {\rm i} \, (3n) \, \delta_{n+m,0}.
\end{align*}
Combining all the above, we get the desired result.

\end{proof}

A very similar Lie algebra appears in the physics literature on 3-dimensional gravity, see e.g. \cite[\S3]{ABDGPR}.

\vs
\line(1,0){400}
\vs

\end{document}